\documentclass[11pt,reqno]{amsart}
\usepackage[margin=1in,letterpaper]{geometry}
\usepackage{graphicx}
\usepackage{amssymb}
\usepackage{amsthm}
\usepackage{mathrsfs}
\usepackage{stmaryrd}
\usepackage{accents} 
\usepackage{enumitem} 

\usepackage{tikz}
\usetikzlibrary{snakes}

\usepackage{hyperref} 

\makeatletter\let\over\@@over\makeatother

\numberwithin{equation}{section}
\theoremstyle{plain} 
\newtheorem{theorem}{Theorem}[section] 
\newtheorem{proposition}[theorem]{Proposition} 

\newtheorem{lemma}[theorem]{Lemma}
\theoremstyle{remark}
\newtheorem{remark}[theorem]{Remark}

\definecolor{gamebookers}{RGB}{50,153,187}
\definecolor{vitaminc}{RGB}{255,153,0}

\newcommand{\bc}{\begin{center}}
\newcommand{\ec}{\end{center}}
\newcommand{\ba}{\begin{align*}}
\newcommand{\ea}{\end{align*}}
\newcommand{\bmc}{\begin{multicols}}
\newcommand{\emc}{\end{multicols}}

\newcommand{\al}{\alpha}
\newcommand{\gam}{\gamma}

\newcommand{\ep}{\epsilon}
\newcommand{\lam}{\lambda}

\newcommand{\tit}[1]{\textit{#1}}
\newcommand{\tbf}[1]{\textbf{#1}}
\newcommand{\trm}[1]{\textrm{#1}}
\newcommand{\ds}{\displaystyle}

\newcommand{\sgn}{\textrm{sgn}}
\newcommand{\tx}{\tilde{x}}
\newcommand{\ty}{\tilde{y}}
\newcommand{\tn}{\tilde{\eta}}
\newcommand{\tpsi}{\tilde{\psi}}
\newcommand{\tPsi}{\tilde{\Psi}}

\newcommand{\tgam}{\tilde{\gamma}}
\newcommand{\tGamma}{\tilde{\Gamma}}
\newcommand{\tom}{\tilde{\omega}}
\newcommand{\ts}{\tilde{s}}
\newcommand{\tr}{\tilde{r}}
\newcommand{\tilh}{\tilde{h}}

\newcommand{\by}{\times}

\DeclareMathOperator{\sech}{sech}

\newcommand{\calX}{\mathcal{X}}
\newcommand{\calY}{\mathcal{Y}}
\newcommand{\Ham}{\mathcal{H}}
\newcommand{\eHam}{\mathcal{H}^{\ep}}
\newcommand{\eK}{\mathcal{K}^{\ep}}

\newcommand{\symplecticform}{\Upsilon}
\newcommand{\calM}{\mathcal{M}}
\newcommand{\dw}{\dot{w}}

\newcommand{\calA}{\mathcal{A}}
\newcommand{\calB}{\mathcal{B}}
\newcommand{\calZ}{\mathcal{Z}}
\newcommand{\calF}{\mathscr{F}}

\newcommand{\genG}{\mathcal F}
\newcommand{\genI}{\mathcal I}
\newcommand{\genU}{\mathcal U}
\newcommand{\genX}{\mathcal X}
\newcommand{\genY}{\mathcal Y}

\newcommand{\Dn}{\Omega_{\eta}}
\newcommand{\tDn}{\tilde{\Omega}_{\tilde{\eta}}}

\newcommand{\R}{\mathbb{R}}

\newcommand{\calR}{R}
\newcommand{\calRbar}{\overline{R}}

\newcommand{\p}{\partial}
\newcommand{\bigO}{\mathcal{O}}
\newcommand{\Froude}{\dfrac{1}{F^2}}
\newcommand{\Fcr}{\dfrac{1}{F^2_{\mathrm{cr}}}}
\newcommand{\Fep}{\dfrac{1}{(F^{\ep})^2}}
\newcommand{\mucr}{\mu_{\mathrm{cr}}}

\newcommand{\hlam}{h^{\lambda}}
\newcommand{\hlamstar}{h^{\lambda_{\star}}}
\newcommand{\vlam}{v^{\lambda}}
\newcommand{\vlamk}{v^{\lambda_k}}
\newcommand{\vlamstar}{v^{\lambda_{\star}}}
\newcommand{\ulam}{u^{\lambda}}
\newcommand{\lamo}{\lambda_0}
\newcommand{\lamstar}{\lambda_{\star}}

\newcommand{\pf}{\begin{proof}}
\newcommand{\QED}{\end{proof}}

\newcommand{\be}{\begin{equation} }
\newcommand{\ee}{\end{equation}}
\newcommand{\bse}{\begin{subequations}}
\newcommand{\ese}{\end{subequations}}

\newcommand{\dist}{\operatorname{dist}}

\newcommand{\Dom}[1]{\operatorname{Dom}{#1}}
\newcommand{\supp}[1]{\operatorname{supp}{#1}}

\newcommand{\Xb}{{X_{\mathrm{b}} }}
\newcommand{\Yb}{{Y_{\mathrm{b}} }}

\newcommand{\odd}{\mathrm{o}}      
\newcommand{\even}{\mathrm{e}}      
\newcommand{\bdd}{\mathrm{b}}       
\newcommand{\F}{{\mathscr F}}

\newcommand{\Pc}{P^{\mathrm{c}}}
\newcommand{\Ph}{P^*}
\newcommand{\Xc}{\mathcal{X}^{\mathrm{c}}}
\newcommand{\Xh}{\mathcal{X}^*}

\newcommand{\reverser}{\mathcal S}  
          
\newcommand{\linear}{L}             
\newcommand{\X}{{\mathcal X}}

\newcommand{\cm}{{\mathscr C}}         
\newcommand{\loc}{{\mathrm{loc}} }          
\newcommand{\Phia}{\tilde\Phi}         

\newcommand{\flowforce}{\mathscr{S}}
  
\newcommand{\Udomain}{\mathscr{U}}  
  
\newcommand{\placeholder}{\,\cdot\,}
\newcommand{\maps}{\colon}

\newcommand{\n}[2][]{#1\lVert #2 #1\rVert}
\newcommand{\abs}[2][]{#1\lvert #2 #1\rvert}
\newcommand{\dell}{\partial}

\begin{document}

\title[Solitary waves with discontinuous vorticity]{Solitary water waves with discontinuous vorticity}

\author[A. Akers]{Adelaide Akers}
\address{Department of Mathematics and Economics, Emporia State University, Emporia, KS 66801} 
\email{aakers1@emporia.edu}  

\author[S. Walsh]{Samuel Walsh}
\address{Department of Mathematics, University of Missouri, Columbia, MO 65211} 
\email{walshsa@missouri.edu} 


\begin{abstract}
We investigate the existence of solitary gravity waves traversing a two-dimensional body of water that is bounded below by a flat impenetrable ocean bed and above by a free surface of constant pressure.  Our main interest is constructing waves of this form that exhibit a discontinuous distribution of vorticity.  More precisely, this means that the velocity limits both upstream and downstream to a laminar flow that is merely Lipschitz continuous.  We prove that, for any choice of background velocity with this regularity, there exists a global curve of solutions bifurcating from a critical laminar flow and including waves arbitrarily close to having stagnation points.  Each of these waves has an axis of even symmetry, and the height of their streamlines above the bed decreases monotonically as one moves to the right of the crest.
\end{abstract}

\maketitle

\setcounter{tocdepth}{1}
\tableofcontents

\section{Introduction} \label{introduction section}

Consider a two-dimensional body of water that lies above a perfectly flat ocean bed and below a region of air.  We  take the motion of the water to be governed by the incompressible Euler equations with an external gravitational force.  The air--water interface is as a free boundary along which the pressure is constant.  \emph{Traveling waves} are special solutions of this system that are independent of time when viewed in a certain moving reference frame.   
Of special importance are \emph{solitary waves}, which are traveling waves that are localized in the sense that their free surfaces are asymptotically flat and the velocity fields approach a laminar background current upstream and downstream.  

Countless works have been devoted to proving the existence of solitary waves and studying their qualitative properties.  However, nearly all of these efforts have focused on the case where the velocity field is irrotational and the background flow is trivial.  On the other hand, numerical simulations reveal that vorticity in the bulk, or the presence of a nontrivial underlying current, may strongly affect the structure of a wave.  In recent years, better understanding the influence of vorticity and wave-current interactions has been a major objective in both mathematics and the applied sciences.  Most relevant to this work, we note that small-amplitude solitary waves with vorticity were constructed independently by Hur \cite{hur2008solitary} and Groves and Wahl\'en \cite{groves2008vorticity}.  Wheeler \cite{wheeler2013solitary,wheeler2015pressure} proved the existence of large-amplitude solitary waves with an arbitrary smooth background current, and Chen, Walsh, and Wheeler \cite{chen2017existence} studied the analogous problem in the heterogeneous density regime.  

In the present paper, we are interested in solitary waves with a background flow that has rapidly varying vorticity.  This occurs, for example, when a wave passes over a strong countercurrent or the vorticity is organized in layers with sharp transition regions.  As we explain in more detail below, we study weak solutions where the background velocity is Lipschitz continuous.  This implies that the vorticity at infinity is merely bounded and measurable; in particular, it may be \emph{discontinuous}.  Constantin and Strauss \cite{constantin2011discontinuous} constructed a global curve of large-amplitude periodic traveling waves with similar regularity, but the solitary wave regime has remained open until now.   

\subsection{Main result} \label{main result section}
Now we formulate the problem more precisely and explain the main theorem.  Switching at the outset to the coordinate frame moving with the wave, we assume that the  water lies in the region 
	\begin{align*}
	 \Dn = \left\{ (x,y) \in \R^2 : -d<y<\eta(x)\right\},
	\end{align*}
where $\eta = \eta(x)$ is the free surface profile that determines the air--water interface, and $d > 0$ is the asymptotic depth.  Note that $\eta$ is a priori unknown; we must determine it as part of the solution.  The velocity field of the wave we denote by $\mathbf{u} = (u,v) : \overline{\Omega_\eta} \to \mathbb{R}^2$, and let $P : \overline{\Omega_\eta} \to \mathbb{R}$ be the pressure.  Recall that in two dimensions, the \emph{vorticity} is identified with the scalar quantity
\[ \omega := v_x - u_y.\]

The motion in the interior of the fluid is governed by the (steady) incompressible Euler equations.  In conservative form, they are:
\begin{subequations} \label{intro weak euler}
\be \label{euler tensor form}
	-c \mathbf{u}_x + \nabla \cdot \left( \mathbf{u} \otimes \mathbf{u} \right) = -\nabla P + \mathbf{g} = 0, \quad \nabla \cdot \mathbf{u} = 0 \qquad \textrm{in } \Omega_\eta,
	\ee
where $c > 0$ is the wave speed, and $\mathbf{g} := (0,-g)$, with $g > 0$ being the constant of gravitational acceleration.  The bed $\{ y = -d \}$ is assumed to be impermeable, while on the free surface $\{ y = \eta(x)\}$ we impose the kinematic and dynamic boundary conditions
\begin{align} \label{intro Ebc}
	\left \{
	\begin{array}{l l}
	v=0 & \text{ on } y=-d\\
	v=(u-c)\eta_x & \text{ on } y=\eta(x)\\
	P=P_{\trm{atm}} & \text{ on } y=\eta(x).
	\end{array}
	\right.
	\end{align}
Here $P_{\mathrm{atm}}$ is the atmospheric pressure.

Additionally, we shall always consider the case where there are no points of horizontal stagnation:
\be \label{no stagnation} 
\sup_{\Omega_\eta} \left( u - c \right) < 0.
\ee
This means that none of the particles have a horizontal velocity exceeding the speed of the wave itself.  One important consequence of \eqref{no stagnation} is that the integral curves of the relative velocity field $(u-c, v)$, which are called the \emph{streamlines},  can each be written globally as a graph of a single-valued function of $x$. 

A solitary wave is a solution of the system \eqref{euler tensor form}--\eqref{no stagnation} that exhibits the asymptotic behavior
\be \label{intro asymptotic conditions}
\eta \to 0, \qquad \mathbf{u} \to (U,0) \qquad \text{ as } x \to \pm \infty, ~\textrm{uniformly in $y$.}
\ee
\end{subequations}
Here, $U = U(y) \in C^{0,1}([-d,0])$ is a given arbitrary function describing the background current.  Observe that the vorticity at infinity is then $U_y \in L^\infty([-d,0])$.  Rather than prescribe $U$ directly, it is in fact more convenient to work with the family of shear flows
	\begin{align*}
	U(y)=c-FU^*(y)
	\end{align*}
where $F$ a dimensionless parameter called the \tit{Froude number} and $U^* \in C^{0,1}([-d,0])$ is a fixed positive function normalized so that
	\begin{align*}
	\ds gd^3 = \left(\int^{0}_{-d} U^*(y) \, dy \right)^2 \qquad \text{ or equivalently } \qquad \Froude=gd^3\left(\int^{0}_{-d} (c-U(y)) \, dy \right)^{-2}.
	\end{align*}
Physically, $U^*$ is simply a rescaling of the relative shear flow at $x=\pm \infty$.  It must be strictly positive in accordance with \eqref{no stagnation}.  The Froude number can be thought of as a non-dimensionalized wave speed.  It is in some sense the ratio between inertial and gravitational forces; later, we will uncover the existence of a critical Froude number $F_{\textrm{cr}}$ that plays an important role in both the existence and qualitative theory.  We refer to a solution with $F > F_{\textrm{cr}}$ as \tit{supercritical}, and one with $F<F_{\textrm{cr}}$ as \tit{subcritical}.


That said, our main contribution is the following.  

\begin{theorem}[Existence of large-amplitude solitary waves] \label{main theorem}
  Fix $\alpha \in (0,1/2]$, wave speed $c > 0$, gravitational constant $g >0$, asymptotic depth $d>0$, and positive asymptotic relative velocity 
  \be U^* \in C^{0,1}([-d,0]) \cap C^{2,\alpha}([-d, -d+\delta)) \cap C^{2,\alpha}((-\delta, 0]),\label{regularity U*} \ee
  for some $\delta \in (0,d)$.   
  There exists a continuous curve
  \begin{align*}
    \cm = \left\{ (u(t), v(t), \eta(t), F(s)) : t \in (0,\infty) \right\} 
  \end{align*}
  of solitary waves solving \eqref{intro weak euler} (in the distributional sense) and having the regularity
  \begin{align}
    \label{(u,v,eta) regularity} 
    (u(t), v(t), \eta(t)) \in 
    C_\bdd^{0,\alpha}(\overline{\Omega(s)}) \times C_\bdd^{0,\alpha}(\overline{\Omega(s)}) \times C_\bdd^{1,\alpha}(\R),
  \end{align}
  where $\Omega(t) := \Omega_{\eta(t)}$.  
  The solution curve $\cm$ has the following properties:
  \begin{enumerate}[label=\rm(\alph*)] 
  \item \label{extreme wave limit} $\cm$ contains waves that are arbitrarily close to having points of (horizontal) stagnation,
    \begin{equation}
      \label{c-u to 0} 
      \lim_{t \to \infty} \inf_{\Omega(t)}  |c-u(t)| = 0.
    \end{equation} 
 \item The left endpoint of $\cm$ is a critical laminar flow, 
    \begin{align*}
      \lim_{t \to 0} (u(t),v(t), \eta(t), F(t)) = (c-F_{\mathrm{cr}} U^*,0,0,F_{\mathrm{cr}}).
    \end{align*}
  \item Every solution in $\cm$ is \emph{symmetric} in the sense that $u(t)$ and $\eta(t)$ are even in $x$ and $v(t)$ is odd in $x$.  Moreover, the elements of $\cm$ are \emph{waves of elevation} in that every streamline (except the bed) lies strictly above its asymptotic height.  Finally, they  are \emph{monotonic}:   the height of each streamline above the bed is strictly decreasing in $x$ for $x > 0$.
  \end{enumerate}
\end{theorem}
\begin{remark} \label{main theorem remark} (i) Note that \eqref{regularity U*} asks that the background current be slightly smoother in a strip near the free surface and bed.  This turns out to be quite important when we apply some maximum principle arguments in Section~\ref{symm}. See Lemma~\ref{addt'l regularity at boundary} and Remark~\ref{improved regularity U* remark}.

(ii) For technical reasons, we work in function spaces that also ensure that the waves in $\mathscr{C}$ have the local Sobolev regularity 
\[ u,v \in W_{\mathrm{loc}}^{1,\frac{2}{1-\alpha}}(\Omega_\eta), \qquad \eta \in W_{\mathrm{loc}}^{2, \frac{2}{1-\alpha}}(\mathbb{R}), \qquad P \in W_{\mathrm{loc}}^{2, \frac{1}{1-\alpha}}(\Omega_\eta).\] 
This is discussed in the remarks following Theorem~\ref{equivalence theorem}.  

(iii) We actually prove in Theorem~\ref{symmetrythm} the much stronger statement that \emph{every} supercritical solitary wave with this regularity has the qualitative properties enumerated in part (c).  
\end{remark}


Before proceeding, let us make some remarks about how Theorem \ref{main theorem} relates to previous works in the literature.  The study of steady water waves stretches back centuries, but it was only in the early 1920s that rigorous existence theories were developed by Nekrasov \cite{nekrasov1921steady} and Levi-Civita \cite{levi1924determinazione}.  Both of these authors considered small-amplitude irrotational periodic waves in infinite depth.  Solitary waves are more challenging to treat analytically due to compactness issues stemming from the unboundedness of the domain.  The first constructions of small-amplitude irrotational solitary waves came in the form of long wavelength limits of periodic solutions (see \cite{Lavrentiev1954}, \cite{friedrichs1954existence}, and \cite{terkrikorov1960existence}).  Beale \cite{beale1977existence} used a generalized implicit function theorem of Nash--Moser type, and later Mielke \cite{mielke1988reduction} used spatial dynamics techniques.  Large-amplitude irrotational solitary waves were obtained by  Amick and Toland \cite{amick1981solitary}, who used global bifurcation methods and a sequence of approximate problems.  A similar result, using different ideas, was proved by Benjamin, Bona, and Bose \cite{benjamin1990solitary}.  

All of these works were carried out in the irrotational regime where $\mathbf{u}$ is divergence free and curl free.  This implies that each component of the velocity field is a harmonic function, and hence the rather unwieldy system \eqref{euler tensor form} can be replaced by Laplace's equation.  The problem can then be pushed to the boundary in several ways.   However, with vorticity, one is forced to contend with more complicated behavior in the bulk.  In particular, incompressibility permits us to introduce a \emph{(relative) stream function} defined up to a constant by
\be \psi_y = u -c, \qquad \psi_x = -v. \label{stream function definition} \ee
The level sets of $\psi$ are the streamlines of the flow.  In fact, the boundary conditions in \eqref{intro Ebc} imply that $\psi$ must be constant on the free surface and bed.   From the definition, we have $-\Delta \psi = \omega$.  One can show that, in the absence of stagnation \eqref{no stagnation}, the vorticity is  functionally dependent on $\psi$.  The velocity in the bulk is therefore captured by the semilinear elliptic problem 
\[ -\Delta \psi = \gamma(\psi) \qquad \textrm{in } {\Omega_\eta},\]
for some $\gamma$ usually called the \emph{vorticity function}.  Beginning with Dubreil-Jacotin \cite{dubreil1934determination}, a standard strategy for constructing rotational steady waves has been to fix a choice of $\gamma$, and then consider solutions of the corresponding free boundary elliptic equation.  Note that for solitary waves, the background current $U$ determines $\gamma$.  In our view, it is slightly more natural in the solitary wave context to prescribe $U$, and so that is how we have phrased Theorem \ref{main theorem}.   It is important to mention that, if $U \in C^{0,1}$, then the vorticity function will generally be $L^\infty$.  At this level of regularity, the functional dependence of $\omega$ on $\psi$ is not obvious, but we confirm it Theorem \ref{equivalence theorem}.  

The first rigorous existence theory for solitary waves with general vorticity was obtained concurrently by  Hur \cite{hur2008solitary} and Groves and Wahl\'en \cite{groves2008vorticity}.  Hur constructed a family of small-amplitude solitary water waves with arbitrary vorticity function $\gamma \in C^0$.  Her method employed a Nash--Moser iteration scheme that generalized the work of Beale on irrotational solitary waves \cite{beale1977existence}.  On the other hand, Groves and Wahl\'en used a spatial dynamics approach together with a center manifold reduction in the spirit of Mielke \cite{mielke1988reduction}.  They allowed for a general vorticity function $\gamma \in H^1$.  Wheeler \cite{wheeler2013solitary,wheeler2015pressure} proved the existence of large-amplitude solitary waves with an arbitrary vorticity function in the H\"older space $C^{1,\alpha}$, for $\al \in (0,1/2]$.  His global theory included an additional alternative that the waves remain bounded away from stagnation while the Froude number diverges to $+\infty$ along the solution curve.  Wheeler later showed that this possibility could be excluded for certain choices $\gamma$ using estimates on $F$ in \cite{wheeler2015froude}.  Finally, Chen, Walsh, and Wheeler \cite{chen2017existence} proved conclusively that the stagnation limit \eqref{c-u to 0} must occur for all (smooth) background flows.  

These results assume some degree of continuity on the vorticity.  However, recent numerical computations in \cite{KS2008} and \cite{KS2008_2} indicate that the discontinuity in the vorticity may bring about flow patterns that vary drastically from those in the continuous case.  This was one of the main motivations for Constantin and Strauss \cite{constantin2011discontinuous} to investigate the existence of periodic solutions with an arbitrary $\gamma$ that is bounded and measurable.  While we are interested here in solitary waves, many of our arguments draw inspiration from their ideas.


\subsection{Plan of the article} \label{plan subsection}

The argument leading to Theorem \ref{main theorem} is quite long, so we now discuss briefly the main difficulties to be overcome, the machinery we will use, and the overarching structure of the paper.  
%

As this is a free boundary problem, we begin by changing coordinates in order to fix the domain.  For rotational steady waves without stagnation, the traditional method for doing this is to use the Dubreil-Jacotin transformation (also called semi-Lagrangian variables).  However, it is not at all obvious at this level of regularity that this is a valid change of variables.  Indeed, confirming the equivalence of the three main formulations of the problem is our first major contribution; see Section \ref{equivalence}.

With this result in hand, we are permitted to apply the Dubreil-Jacotin transformation, which recasts the steady incompressible Euler system as a scalar quasilinear elliptic PDE with fully nonlinear boundary conditions posed on a fixed domain.  More precisely, the fluid domain $\Omega_{\eta}$ is mapped to the infinite strip $R = \mathbb{R} \times (-1,0)$.  For the purposes of this discussion, we can represent the problem as an abstract operator equation of the form $\mathscr{F}(\phi, F) = 0$, where $\phi$ is a new unknown that describes the deviation of the streamlines from their far-field heights.   

At this point, we encounter a second obstacle:  the unboundedness of $R$ has serious implications for the compactness properties of the linearized operator $\mathscr{F}_\phi(\phi, F)$.  In particular, it is well-known in the literature of solitary waves that $\mathscr{F}_\phi(0,F_{\mathrm{cr}})$ fails to be Fredholm.  We are therefore barred from using a Lyapunov--Schmidt reduction approach to construct small-amplitude waves as was done in the periodic case (see, for example, \cite{constantin2004exact,constantin2011discontinuous}).  A similar issue was faced by Wheeler \cite{wheeler2013solitary,wheeler2015pressure} and Chen, Walsh, and Wheeler \cite{chen2017existence}, who were able to prove that the linearized operator at a \emph{supercritical} wave is in fact Fredholm index $0$.   However, these authors studied classical solutions, and some delicate adaptations are necessary in our setting.

In place of Lyanpunov--Schmidt, we devote Section \ref{small amp} to constructing a family $\mathscr{C}_{\mathrm{loc}}$ of small-amplitude solitary waves using a spatial dynamics approach similar to that of Groves and Wahl\'en \cite{groves2008vorticity}, as well as Wheeler \cite{wheeler2013solitary}.   First, we rewrite the problem once more as an infinite-dimensional Hamiltonian system where the horizontal variable $x$ acts as time.  At the critical value of the Froude number, $0$ is an eigenvalue of algebraic multiplicity $2$ for the linearized problem and the rest of the spectrum is bounded away from the imaginary axis.  We are therefore able to invoke a variant of the center manifold  theorem for quasilinear elliptic equations on infinite cylinders pioneered by Mielke \cite{mielke1988reduction} and Kirchg\"assner \cite{kirchgassner1982wavesolutions,kirchgassner1988resonant}.   This reduces the infinite-dimensional problem to a planar Hamiltonian system that, in fact, is equivalent to the Korteweg--de Vries equation modulo a rescaling.  We prove that for every slightly supercritical Froude number, the reduced equation has a homoclinic orbit, and these lift up to give solitary wave solutions of the original Euler problem.

The weak regularity of our solutions also presents difficulties when carrying out the maximum principle arguments that are crucial to proving the existence of large-amplitude waves.  For this reason, we follow \cite{constantin2011discontinuous} and assume some additional smoothness for $U^*$ near the bed and free surface.  Using an elliptic regularity argument that exploits the translation invariance of the domain, we can then infer that the solutions likewise enjoy enough regularity near the boundary so that the Hopf boundary point lemma and Serrin corner-point lemma can be applied.   By a moving planes method, we then prove that every supercritical solitary wave solutions exhibits even symmetry, is monotone, and (necessarily) a wave of elevation.  

In section \ref{proof main result section}, we complete the proof of Theorem \ref{main theorem} by extending $\mathscr{C}_{\textrm{loc}}$ to a global curve $\mathscr{C}$ using a variation of the Dancer \cite{dancer1973bifurcation}, and Buffoni and Toland \cite{buffoni2003analytic} abstract global bifurcation theory introduced recently by Chen, Walsh, Wheeler \cite{chen2017existence}.  

Finally, for the convenience of the reader, in Appendix~\ref{quoted results section} we provide the statement of a number of results that we draw on in the paper.

\section{Preliminaries}
\label{form}

\subsection{Notation} \label{notation section}
Before we begin, we must fix some notation.  Let $D\subset \R^n$ be a possibly unbounded domain in $R^n$ for $n \geq 1$. We denote the space of test functions
\begin{align*}
  C^\infty_{\mathrm c}(D) &:= \left\{ \phi\in C^\infty(D):\ \supp{\phi} \subset\subset D  \right\}.
\end{align*}
 For $k \in \mathbb{N}$, and $\alpha\in (0,1]$, we write 
 \begin{align*}
  C^{k,\alpha}(D) &:= \left\{ u\in C^k(D):\ \|\phi u\|_{C^{k,\alpha}}<\infty \ \text{ for all } \phi\in C^\infty_{\mathrm c}(D) \right\}.
\end{align*}
Thus $C^{k,\alpha}$ refers to \emph{locally} $k$-times H\"older continuously differentiable functions.  Note that the special case $\alpha = 1$ corresponds to locally Lipschitz functions.  We also define in the obvious way the spaces $C^{\infty}_{\mathrm c}(\overline{D})$ and $C^{k+\alpha}(\overline D)$.    On the other hand, we denote by
\[
  C^{k,\alpha}_\bdd(\overline D) := \left\{ u\in C^k(\overline{D}):\ \|u\|_{C^{k,\alpha}}<\infty \right\},
\]
the space of \emph{uniformly} $k$-times H\"older continuously differentiable functions, which is a Banach space under the $C^{k,\alpha}$ norm.    

As we are often interested in asymptotic behavior, we will frequently work with spaces of the form
\[
  C^{k,\alpha}_0(\overline D) := \left\{ u\in C^{k+\alpha}_\bdd(\overline D):\ \lim_{r\to\infty}\sup_{|x| = r} |\partial^\beta u(x)| = 0 \text{ for } 0\leq |\beta| \leq k \right\}.
\]
It is easily seen that $C_0^{k,\alpha}(\overline{D})$ is a closed subspace of $C^{k,\alpha}_\bdd(\overline D)$ under the $C^{k+\alpha}(\overline D)$ norm. 

To keep clear the distinction between the local and uniform topologies, we adopt the convention that to say $u_j \to u$ in $C^{k,\alpha}(\overline{D})$ means precisely that $\| u_j - u \|_{C^{k,\alpha}(D)} \to 0$.  On the other hand, if we write $u_j \to u$ in $C_{\mathrm{loc}}^{k,\alpha}(\overline{D})$, it means only that $\phi u_j \to \phi u$ in $C^{k,\alpha}(\overline{D})$ for all $\phi \in C_c^\infty(\overline{D})$.  

We will also make extensive use of Sobolev spaces.  For $k \in \mathbb{N}$ and $p \in [1,\infty]$ we define
\[ W^{k,p}(D) := \left\{ u \in L^p(\Omega) : \partial^\beta u \in L^p(\Omega) \textrm{ for } 0 \leq |\beta| \leq k \right\}.\]
Likewise, we let $W_{\mathrm{loc}}^{k,p}(D)$  be the set of measurable functions on $\Omega$ such that $\phi u \in W^{k,p}(D)$ for all $\phi \in C_c^\infty(\overline{D})$.  It is well-known that this has Banach space structure when prescribed the norm
\[ \| u \|_{W_{\mathrm{loc}}^{k,p}(D)} :=  \sum_{m = 1}^\infty 2^{-m} \frac{\| \phi_m u \|_{W^{k,p}(D)}}{1+\| \phi_m u \|_{W^{k,p}(D)}},\]
where $\{ \phi_m \} \subset C_c^\infty(\overline{D})$ is a family of cut-off functions uniformly bounded in $C_\bdd^\infty$ and such that $B_m(0) \cap \overline{D} \subset \supp{\phi_m} \subset B_{2m}(0) \cap \overline{D}$.  Thus $u_j \to u$ in $W_{\mathrm{loc}}^{k,p}(D)$ if and only if $u_j \to u$ in $W^{k,p}(D^\prime)$, for all $D^\prime \subset \subset \overline{D}$.   As usual, if $p = 2$, we will write $H^k$ and $H_{\mathrm{loc}}^k$ in place of $W^{k,2}$ and $W_{\mathrm{loc}}^{k,2}$, respectively. 

Finally, for any of the above spaces, we may add a subscript of $\even$ or $\odd$ to indicate evenness or oddness with respect to the first variable.  

\subsection{Three formulations} \label{three formulations section}
The steady water wave problem has many alternative formulations, each one offering certain advantages.  In this work, we will in fact move between four equivalent versions of the governing equations.  The most fundamental is the (weak) \emph{velocity formulation} presented in \eqref{intro weak euler}.  We also make use of the \emph{stream function formulation}, which rewrites the system in terms of the stream function \eqref{stream function definition} and free surface profile.  Most of the qualitative theory and large-amplitude existence theory will be done in the \emph{height equation formulation}, where the unknown is the so-called height function of Dubreil-Jacotin.  Finally, the small-amplitude theory is conducted in the \emph{spatial dynamics formulation}.  The equivalence of these systems is quite standard when the regularity is classical.  However, with discontinuous vorticity and on an unbounded domain the issue is surprisingly subtle.  In this section, we discuss in more detail the first three of the above formulations.

\subsubsection*{Weak velocity formulation}

We begin by recalling the weak formulation of the governing equations.  Written in component form, \eqref{euler tensor form} becomes
\begin{subequations} \label{weak form}
	\begin{align} \label{velocity form interior}
	\left \{
	\begin{array}{l l}
	-cu_x +(u^2)_x+(uv)_y= -P_x & \\
	-cv_x+(uv)_x+(v^2)_y= -P_y-g & \text{ in } \Dn,\\
	u_x+v_y=0 & 
	\end{array}
	\right.
	\end{align}
and the boundary conditions are 
	\begin{align} \label{Ebc}
	\left \{
	\begin{array}{l l}
	v=0 & \text{ on } y=-d\\
	v=(u-c)\eta_x & \text{ on } y=\eta(x)\\
	P=P_{\trm{atm}} & \text{ on } y=\eta(x).
	\end{array}
	\right.
	\end{align}
Here again, $P_{\trm{atm}}$ is the (constant) atmospheric pressure and $g$ is the gravitational constant of acceleration.  For solitary waves, we also impose the asymptotic  conditions
	\begin{align} \label{Eac}
	\eta \to 0, \qquad v \to 0, \qquad u \to U(y)=c-FU^*(y), \qquad \text{ as } x \to \pm \infty,\textrm{ uniformly in $y$.}
	\end{align}
	Recall that we think of $U^*$ as fixed, and $F$ is a parameter.   As always, we require there to be no horizontal stagnation points:
\be \sup_{\Omega_\eta} \left( u - c \right) < 0.
  \label{no stag} \ee
  \end{subequations}
\subsubsection*{Stream function formulation}

We will eliminate the pressure by introducing the relative stream function $\psi$ defined by \eqref{stream function definition}.  The no stagnation condition \eqref{no stag} dictates that
\begin{subequations} \label{stream function form}
	\begin{align} \label{no stagnation condition}
	u-c=\psi_y<0,
	\end{align}
throughout the fluid.  In particular, this guarantees the existence of a \tit{vorticity function} $\gam$ that satisfies
	\begin{align*}
	-\Delta \psi = \omega = \gam(\psi).
	\end{align*}
This fact is easily shown for classical solutions, but not at all obvious in the weak setting that we now consider.  We will verify that it holds when we prove the equivalence theorem at the end of this section.

Taking that for granted momentarily, we may transform the weak velocity formulation \eqref{velocity form interior}--\eqref{Ebc} into the following free boundary elliptic problem

	\begin{align} \label{fbp}
	\left \{
	\begin{array}{l l}
	\Delta \psi = -\gam(\psi) & \text{ in } \Dn\\
	\psi=0 & \text{ on } y=\eta(x)\\
	\psi=m & \text{ on } y=-d\\
	\left|\nabla \psi \right|^2 + 2g(y+d) = Q & \text{ on } y=\eta(x)
	\end{array}
	\right.
	\end{align}
together with the asymptotic conditions
	\begin{align} \label{fbpa}
	\eta \to 0, \qquad \psi_x \to 0, \qquad \psi_y \to -FU^*(y), \qquad \text{ as } x\to \pm \infty, \textrm{ uniformly in $y$.}
	\end{align}
\end{subequations}
Here, $m>0$ is the volumetric mass flux
	\begin{align} \label{mdefn}
	\ds m:=F \int^{0}_{-d} U^*(y) \, dy,
	\end{align}
$Q$ is a constant to be determined later, and the vorticity function $\gam$ is given implicitly in terms of $U^*$ and $F$ by
	\begin{align*}
	\ds \gam(-s)=FU^*_y(y), \qquad \text{ where } s=F\int^{y}_{-d} U^* dy'.
	\end{align*}
Note that this is valid because $s$ is strictly increasing.

We continue by writing (\ref{fbp}) together with (\ref{fbpa}) in terms of the dimensionless variables
\[
(\tx,\ty) : =\left(\frac{x}{d},\frac{y}{d}\right), \quad \tn(\tx) := \dfrac{1}{d} \eta(x), \quad \tpsi(\tx,\ty) := \dfrac{1}{m} \psi(x,y), \quad \tgam(\tpsi) := \dfrac{d^2}{m} \gam(\psi),
\]
which result from rescaling lengths by $d$ and velocities by $m/d$.  The stream function formulation \eqref{stream function form} then becomes
\begin{subequations} \label{streamline form}
	\begin{align}
	\label{dfbp}
	\left \{
	\begin{array}{l l}
	\Delta \tpsi = -\tgam(\tpsi) & \text{ in } \tDn\\
	\tpsi =1 & \text{ on } \ty = -1\\
	\tpsi =0 & \text{ on } \ty = \tn(\tx)\\
	|\nabla \tpsi|^2 + \dfrac{2}{F^2}(\tn+1) = \tilde Q & \text{ on } \ty=\tn(\tx)
	\end{array}
	\right.
	\end{align}
with asymptotic conditions
	\begin{align}
	\label{dfbpa}
	\tn \to 0, \qquad \tpsi_x \to 0, \qquad \tpsi_y \to \dfrac{U^*(\ty d)d}{\int^0_{-d} U^*(y) \, dy} \qquad \text{ as } x \to \pm \infty, \textrm{ uniformly in $\tilde y$.}
	\end{align}
\end{subequations}
Note that this implies that
	\begin{align*}
	\ds \tpsi(\tx, \ty) \to \tPsi(\ty) :=\dfrac{\int^{\ty d}_{-1} U^*(y) \, dy}{\int^{0}_{-d} U^*(y) \, dy} \qquad \textrm{as $\tilde x \to \pm \infty$, uniformly in $\tilde y$.}
	\end{align*}
Lastly, the dimensionless vorticity function $\tgam$ is given in terms of $U^*$ at infinity according to
	\begin{align*}
	\ds \tgam(-\ts) = \dfrac{d^2}{m} \gam(-s) = \dfrac{d^2U^*_{\ty}(\ty d)}{\int^{0}_{-d} U^* \, dy} \qquad
	\text{ where } \ts = \dfrac{\int^{y}_{-d} U^* \, dy}{\int^{0}_{-d} U^* \, dy}.
	\end{align*}

\subsubsection*{Height function formulation}

Applying the Dubreil-Jacotin change of variables 
\begin{figure} [tb!]
\centering
\includegraphics[page=3,width=0.7\textwidth]{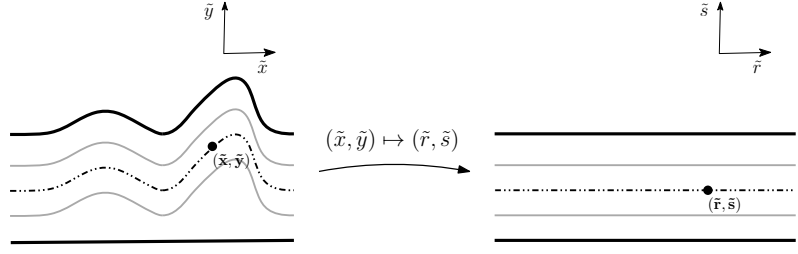}
\caption{Dubreil-Jacotin transformation}
\label{DBtransformation}
\end{figure}

	\begin{align*}
	(\tilde x, \tilde y) \mapsto (\tilde r, \tilde s) := \left( \tilde x, -\tpsi \right),
	\end{align*}
 sends the non-dimensional fluid domain $\tilde \Omega_{\tilde \eta}$ to the infinite strip $R := \mathbb{R} \times (-1, 0)$.   We also consider the new unknown 
 \[ h = h(\tilde r, \tilde s) := \tilde y + 1,\]
 which is called the \emph{height function}.  One can think of it as essentially the second coordinate of the inverse of the Dubreil-Jacotin transformation.  Physically, $h(\tilde r, \tilde s)$ is the vertical distance between the bed and the point sitting on the streamline $\{\tilde \psi = -\tilde s\}$ with $\tilde x = \tilde r$ ; see Figure \ref{DBtransformation}.  Written in these coordinates, \eqref{streamline form} becomes the following quasilinear elliptic problem in divergence form
\begin{subequations} \label{height function form}
	\begin{align}
	\label{hp}
	\left \{
		\begin{array}{l l}
		\ds \left( \dfrac{\tilde{h}_{\tr}}{\tilde{h}_{\ts}} \right)_{\tr} - \left(\dfrac{1+\tilde{h}^2_{\tr}}{2\tilde{h}^2_{\ts}} +  \tilde\Gamma \right)_{\ts} =0 & \text{ in } R \\
		\tilde{h}=0 & \text{ on } \ts=-1\\
		\ds \dfrac{1+\tilde{h}^2_{\tr}}{2\tilde{h}^2_{\ts}} + \Froude \tilde{h} = \dfrac{Q}{2} & \text{ on } \ts=0
		\end{array}
	\right.
	\end{align}
with asymptotic conditions
	\begin{align} \label{hpa}
	\tilde{h} \to \tilde{H}(\ts), \qquad \tilde{h}_{\tr} \to 0, \qquad \tilde{h}_{\ts} \to \tilde{H}_{\ts} \qquad \text{ as } \tr \to \pm \infty, \textrm{ uniformly in $\ts$.}
	\end{align}
Here we have eliminated $\tilde \gamma$ in favor of its primitive $\tilde \Gamma \in C^{0,1}([-1,0])$ defined by
	\begin{align} \label{Gamma}
	 \tilde \Gamma(s)= \int^{s}_{0} \tilde\gam(-z) \, dz.
	\end{align}
The asymptotic height function $\tilde{H}$ is related to the limiting stream function $\tilde \Psi$ and relative velocity $U^*$ via the differential equation
	\begin{align*}
	\left \{
	\begin{array}{l l}
	\ds \tilde{H}_{\ts}(\ts) = -\dfrac{1}{\tPsi_{\ty}\left(\tilde{H}(\ts)-1\right)} = -\dfrac{\int^{0}_{-d} U^*(y) \, dy}{U^{*}\left((\tilde{H}(\ts)-1) d \right) d} & \\
	\\
	\tilde{H}(-1)=0, \quad \tilde{H}(0)=1.
	\end{array}
	\right.
	\end{align*}
Finally, the no stagnation condition \eqref{no stag} has the following expression in the Dubreil-Jacotin variables:
\be \inf_{R} \tilde h_{\tilde s} > 0. \label{DJ no stag} \ee	\end{subequations}

\subsubsection*{Equivalence}
\label{equivalence}

The main result of this section is the following theorem stating the equivalence of the three formulations discussed above.  

\begin{theorem}[Formulation equivalence] \label{equivalence theorem} Let $\alpha \in (0,1)$ be given and set $p := 2/(1-\alpha)$.  Then the following statements are equivalent:
	\begin{enumerate}[font=\upshape, label=(\roman*)]
	\item \label{i} there exists a solution of the velocity formulation \eqref{weak form} with the regularity 
	\[ u, v \in C_{\bdd}^{0,\alpha}(\overline{\Omega_\eta}) \cap  W_{\mathrm{loc}}^{1,p}(\Omega_\eta), \quad \eta \in C_{\bdd}^{1,\alpha}(\overline{\Omega_\eta}) \cap  W_{\mathrm{loc}}^{2,p}(\mathbb{R}), \quad P \in C_{\bdd}^{1,\alpha}(\overline{\Omega_\eta}) \cap W_{\mathrm{loc}}^{2,\frac{p}{2}}(\Omega_\eta)\]
	 for a given $U^* \in C^{0,1}([-d,0])$; 
	\item \label{ii} there exists a solution of the stream function formulation \eqref{streamline form} with the regularity 
	\[ \tilde \psi \in C_\bdd^{1,\alpha}(\overline{\tilde \Omega_{\tilde \eta}}) \cap W_{\mathrm{loc}}^{2,p}(\tilde \Omega_{\tilde \eta}), \quad \tilde \eta \in C_{\bdd}^{1,\alpha}(\mathbb{R}) \cap  W_{\mathrm{loc}}^{2,p}(\mathbb{R}) \] 
	for a given $\tilde \gamma \in L^\infty([-1,0])$ or $\tilde \Psi \in C^{1,1}([-1,0])$;
	\item \label{iii} there exists a solution of the height formulation \eqref{height function form} with the regularity 
	\[ \tilde h \in C_\bdd^{1,\alpha}(\overline{R}) \cap W_{\mathrm{loc}}^{2,p}(R),\]
	for a given $\tilde \Gamma \in C^{0,1}([-1,0])$ or $\tilde H \in C^{1,1}([-1,0])$.
	\end{enumerate}
\end{theorem}

We pause to make some comments about the selection of spaces here, as they are admittedly somewhat unusual.  A version of Theorem \ref{equivalence theorem} was obtained by Constantin and Strauss \cite{constantin2011discontinuous} in their study of the periodic regime.  They introduced the idea of using Sobolev spaces rather than working entirely with H\"older continuous functions.   Note that by Morrey's inequality, $W^{1,p}_{\mathrm{loc}} \subset C^{0,\alpha}$, so this choice assumes more regularity than what at first glance appears necessary.   However, they observed that if one knows only that $\tilde h \in C^{1,\alpha}$, then the corresponding stream function $\tilde \psi$ does not necessarily satisfy \eqref{dfbp}, but rather a weaker (fully nonlinear) elliptic equation:
\be \partial_{\tilde x} \left( -\tilde \psi_{\tilde x} \tilde \psi_{\tilde y} \right) + \partial_{\tilde y} \left( \frac{1}{2} \tilde \psi_{\tilde x}^2 - \tilde \psi_{\tilde y}^2 + \Gamma(-\psi) \right) = 0.\label{very weak stream function form} \ee
  Nonetheless, Varvaruca and Zarnescu \cite{varvaruca2012equivalence} showed that the weak Euler formulation and this weaker stream function formulation \eqref{very weak stream function form} are equivalent when the H\"older exponent $\alpha \in (1/3, 1]$.  Most recently Sastre-Gomez \cite{SSG2017} established that the weak velocity and weak stream function forms are equivalent to the modified-height formulation (due to Henry \cite{Henry2013}), again for periodic solutions in H\"older spaces.  
  
The present paper follows the Constantin--Strauss approach of assuming that $\tilde h$ has second-order weak derivatives in $L^p$ largely because we are unable to prove some necessary bounds on the pressure working only with \eqref{very weak stream function form}; see Lemma \ref{pressure bound lemma} and Remark~\ref{pressure bound remark}. On the other hand, because our domain is unbounded, we do not wish to impose integrability requirements at infinity.  It is for these reasons that we ask for uniform boundedness in H\"older norm and only locally integrable weak differentiability.   It is also worth noting that we assume full Lipshitz continuity of the asymptotic data.  This is to simplify a number of results later.  In particular, it allows us to construct certain auxiliary functions that are important to the spectral theory in Section \ref{SLP} and later the qualitative theory in Section \ref{symm}.  One can also see that, for example, determining $\tilde H$ from $\tilde \Psi$ requires solving the ODE $\tilde H_{\tilde s} = -1/(\tilde \Psi_{\tilde y}(\tilde H - 1))$, which makes taking $\tilde \Psi$ in  $C^{1,1}$ quite natural.  

\begin{proof}[Proof of Theorem \ref{equivalence theorem}]   We closely follow the argument in \cite[Theorem 2]{constantin2011discontinuous}.  First, let us show that \ref{i} implies \ref{ii}.  Suppose that $(u,v)$, $P$, and $\eta$ have the stated regularity.  We define $\psi$ according to \eqref{stream function definition} and non-dimensionalize to obtain $\tilde \psi$ and $\tilde \eta$. Clearly they both have the required regularity.  Moreover, the asymptotic conditions \eqref{fbpa} follow from the definition of $\tilde \psi$ and $\tilde \eta$, along with the corresponding limits for $(u,v)$ and $\eta$ in \eqref{Eac}.  

The kinematic boundary condition and no penetration condition in the weak Eulerian formulation lead directly to 
	\begin{align*}
	\left \{
	\begin{array}{l l}
	-\tilde \psi_{\tilde x}=\tilde\psi_{\tilde y} \tilde \eta_{\tilde x} & \text{ on } \tilde{y}=\tilde\eta(\tilde x)\\
	\tilde\psi_{\tilde x}=0 & \text{ on } \tilde y=-1.
	\end{array}
	\right.
	\end{align*}
Note that as $\tilde \psi \in C^{1,\alpha}(\tilde\Omega_{\tilde\eta})$ and $\tilde \eta \in C^{1,\alpha}(\mathbb{R})$, these are equivalent to $\tilde \psi$ being constant on the free surface and bed.  Without loss of generality we may set $\tilde \psi=0$ on $\{\tilde y= \tilde \eta (\tilde x)\}$, which forces $\tilde \psi=1$ on $\{\tilde y = 0\}$.

We claim, and prove below, that $\tilde \omega = \tilde \gam(\tilde \psi)$ for some $\tilde\gam \in L^{\infty}([0,1])$.  Assuming the claim momentarily, it follows that (\ref{velocity form interior})--(\ref{Eac}) becomes
	\begin{align*}
	\left \{
	\begin{array}{l l}
	\Delta\tilde \psi=-\tilde\gam(\tilde\psi) & \text{ in } \tilde \Omega_{\tilde \eta}\\
	\tilde \psi=1 & \text{ on } \tilde y=-1\\
	\tilde\psi=0 & \text{ on } \tilde y= \tilde \eta(\tilde x).
	\end{array}
	\right.
	\end{align*}
with asymptotic conditions (\ref{fbpa}).  To check the remaining nonlinear boundary condition, define
	\begin{align*}
	\ds \tilde E: =\dfrac{(\tilde c-\tilde u)^2+\tilde v^2}{2} + \frac{1}{F^2}(\tilde y+1) +\tilde P- \tilde\Gamma(-\tilde\psi) \in C_\bdd^{0,\alpha}(\overline{\tilde \Omega_{\tilde \eta}}) \cap W^{1,\frac{p}{2}}_{\mathrm{loc}}(\tilde \Omega_{\tilde \eta}),
	\end{align*}
where
\[ \left(\tilde u(\tilde x, \tilde y) -\tilde c, \tilde v(\tilde x, \tilde y)\right) :=  \frac{m}{d} \left(  u(x,y) -c,  v(x,y)\right), \qquad \tilde P(\tilde x, \tilde y) =  \frac{m^2}{d^2}P(x,y) \]
are the non-dimensionaized relative velocity field and pressure.  A straightforward calculation shows that the gradient of $\tilde E$ vanishes throughout $\tilde \Omega_{\tilde \eta}$, which is the weak form of Bernoulli's law.  Indeed, this is equivalent to the first two equations in \eqref{velocity form interior}.  Furthermore, evaluating $E$ on the free surface and noticing that 
	\begin{align*}
	\tilde Q := 2\left( \tilde\Gamma(-\tilde \psi) - \tilde P \right) \Big|_{\tilde y=\tilde \eta(\tilde x)} = 2(\tilde\Gamma(0) - \tilde P_{\textrm{atm}}),
	\end{align*}
we obtain the missing part of (\ref{fbp}).

Now we wish to show that \ref{ii} implies \ref{i}.  Suppose we have a solution of the stream function problem \eqref{streamline form} with the regularity $\tpsi \in C_\bdd^{1,\alpha} \cap W_{\mathrm{loc}}^{2,p}$ and $\tilde{\eta} \in C^{1,\alpha}$.  Then  we recover the velocity field $(u, v)$ via \eqref{stream function definition}.  Letting the (non-dimensional) pressure be given by
	\be \label{pressure in terms of E} 
	\tilde P:= -\dfrac{1}{2} \left( \tilde\psi_{\tilde x}^2 + \tilde \psi_{\tilde y}^2\right)-\frac{1}{F^2}(\tilde y+1)+\tilde\Gamma(-\tilde\psi) + \dfrac{\tilde Q}{2} + \tilde P_{\trm{atm}} \in C_{\bdd}^{0,\alpha}(\overline{\tilde \Omega_{\tilde \eta}}) \cap W_{\mathrm{loc}}^{1,\frac{p}{2}}(\tilde \Omega_{\tilde \eta}),
	\ee
 and taking the gradient yields the Euler equations \eqref{velocity form interior} and the Bernoulli boundary condition.  This completes the proof that \ref{i} and \ref{ii} are equivalent.

In the process of showing \ref{ii} implies \ref{iii}, we will also prove the previous claim that $\tom = \tgam(\tpsi)$ for some $\tgam \in L^{\infty}([0,1])$.  Suppose again that we have a solution of the stream function problem \eqref{streamline form} with the stated regularity and consider the change of independent and dependent variables
	\begin{align} \label{ytos}
	\ts := -\tpsi(\tx,\ty), \qquad \tr :=\tx, \qquad \tilde{h}(\tr,\ts):=\ty+1, \qquad \text{ for } (\tr, \ts) \in R,
	\end{align}
so that
	\begin{align} \label{cov1}
	\tilde{h}_{\tr}=\dfrac{\tilde v}{\tilde u- \tilde c}, \qquad \tilh_{\ts}=\dfrac{1}{\tilde c- \tilde u}, \qquad v=\dfrac{-m\tilh_{\tr}}{\tilh_{\ts}d}, \qquad  u=c-\dfrac{m}{d\tilh_{\ts}},
	\end{align}
and
	\begin{align} \label{cov2}
	\p_{\tx} =\p_{\tr} - \dfrac{\tilh_{\tr}}{\tilh_{\ts}}\p_{\ts}, \qquad \p_{\ty}=\dfrac{1}{\tilh_{\ts}}\p_{\ts},
	\end{align}
hence
	\begin{align} \label{cov3}
	\p_x=\dfrac{1}{d}\left(\p_{\tr} - \dfrac{\tilh_{\tr}}{\tilh_{\ts}}\p_{\ts}\right) \qquad \text{and} \qquad \p_y=\dfrac{1}{d}\left(\dfrac{1}{\tilh_{\ts}}\p_{\ts} \right).
	\end{align}
In addition, for $\tom \in L_{\mathrm{loc}}^p(\tilde \Omega_{\tilde \eta})$, we see that the following identity holds in the sense of distributions:
	\begin{align*}
	\p_{\tr}\tom=\left(\p_x-\dfrac{v}{c-u}\p_y\right)\tom.
	\end{align*}
Taking the curl of the Euler equations \eqref{velocity form interior} yields
	\begin{align*}
	0 &= (u-c)\tom_x+v\tom_y = (c-u)\p_{\tr}\tom
	\end{align*}
It follows that $\p_{\tr}\tom=0$, so $\tom$ is a function only of $\ts$ throughout $R$.  We are therefore justified in letting $\tom = \tgam(\tilde \psi )$ for some vorticity function $\tgam \in L^p([0,1])$.  

Next, we confirm that $\gamma \in L^\infty$.  Let $\zeta \in C_c^\infty((-1,0))$ be a test function and let $\xi \in C_c^\infty(\mathbb{R})$ be a cut-off function with 
\[ 0 \leq \xi \leq 1, \qquad \xi = 1 \textrm{ on $(-1,1)$}, \qquad \supp{\xi} \subset (-2,2), \qquad \int_{\mathbb{R}} \xi \, dr = 1.\]
For each $\epsilon > 0$, we set $\xi^\epsilon := \xi(\epsilon \cdot)/(2\epsilon)$.  Also, for $n \geq 0$, we put $\varphi_n^\epsilon := \xi^\epsilon(r-n) \zeta(s) \in C_c^\infty(R)$, and compute  
\begin{align*}
\int_R \tilde \gamma(-\tilde s) \varphi_n^\epsilon \, d\tilde r \, d \tilde s & = \int_{-1}^0 \tilde \gamma(-\tilde s) \zeta(\tilde s) \, d\tilde s, \qquad \textrm{for all } \epsilon > 0, ~n \geq 1.
\end{align*}
On the other hand, from \eqref{ytos}--\eqref{cov3}, we find that
\begin{align*}
\int_{\tilde \Omega_{\tilde \eta}} \tilde \omega \varphi_n^\epsilon(\tilde x, -\tilde \psi) \, d \tilde x \, d \tilde y & = \int_R \left( -\tilde v \left( \tilde h_{\tilde s} \partial_{\tilde r} \varphi_n^\epsilon - \tilde h_{\tilde r} \partial_{\tilde s} \varphi_n^\epsilon  \right) + \tilde u \partial_{\tilde s} \varphi_n^\epsilon \right) \, d \tilde r \, d \tilde s \\
& = \frac{1}{2\epsilon} \int_{n-\epsilon}^{n+\epsilon} \xi(\epsilon \tilde r) \int_{-1}^0 \left( \tilde u - \tilde h_{\tilde r} \tilde v \right) \partial_{\tilde s} \zeta \, d\tilde s \, d\tilde r - \frac{1}{2} \int_{n-\epsilon}^{n+\epsilon} \xi^\prime(\epsilon \tilde r) \int_{-1}^0 v \tilde h_{\tilde s} \, d\tilde s \, d\tilde r.
\end{align*}
Taking $n \to \infty$ and $\epsilon \to 0$, keeping in mind the asymptotic conditions \eqref{hpa}, this at last gives 
\[ \int_{-1}^0 \tilde \gamma(-\tilde s) \zeta(\tilde s) \, d\tilde s = -\int_{-1}^0 \tilde U_{\tilde y}(\tilde y) \zeta(-\tilde \Psi(\tilde y)) \, d \tilde y.\]
By assumption, $\tilde U \in C^{0,1}$, and hence we have $\tilde \gamma  = -\tilde U_y(\tilde H(\cdot) - 1) \in L^\infty$.

Now, the identities \eqref{cov1} and no stagnation condition \eqref{no stag} immediately imply that $\tilde h \in C_\bdd^{1,\alpha}(\overline{R}) \cap W_{\mathrm{loc}}^{2,p}(R)$.  Applying the change of variables \eqref{cov2} to \eqref{dfbp} in the interior of $\tDn$, we see that
	\begin{align*}
	0 &=\Delta\tpsi + \tgam(\tpsi) =\left(\p_{\tr} - \dfrac{\tilh_{\tr}}{\tilh_{\ts}}\p_{\ts}\right)\left(\dfrac{\tilh_{\tr}}{\tilh_{\ts}}\right)+\dfrac{1+\tgam(-s)}{\tilh_{\ts}}\p_{\ts}\left(\dfrac{1}{\tilh_{\ts}}\right)+\tgam(-s)\\
	&=\left(\dfrac{\tilh_{\tr}}{\tilh_{\ts}}\right)_{\tr} + \left(\dfrac{1+\tilh_{\tr}^2}{2\tilh_{\ts}^2}\right)_{\ts}+\tgam(-s).
	\end{align*}
The boundary conditions in \eqref{hp} and asymptotic condition \eqref{hpa} follow similarly.  

Let us now verify that $\tGamma(-\tpsi) \in C^{0,1}(\tDn)$ and the chain rule holds for $\tGamma \in C^{0,1}([-1,0])$.  First, notice that $\tGamma(-\tpsi) \in C_\bdd^{0,\alpha}(\overline{\tDn})$ since $\Gamma \in C^{0,1}([-1,0])$ and $\tilde \psi \in C_\bdd^{1,\alpha}(\overline{\tDn})$.  We aim to justify the following:
	\begin{align*}
	\p_{\tx}\tGamma(-\tpsi) = -\tgam(\tpsi)\tpsi_{\tx} \in L_{\mathrm{loc}}^p(\tDn)
	\end{align*}
Let $\varphi \in C_c^\infty(\tDn)$ be a test function and observe the action of the distribution $\p_{\tx}\tGamma(-\tpsi)$ on $\varphi$ is given by
	\begin{align*}
	\ds  \int_{\tilde \Omega_{\tilde \eta}} \p_{\tx}\tGamma(-\tpsi) \varphi \, d\ty \, d\tx &=- \int_{R} \tilh_{\ts} \tGamma(\ts) \left(\p_{\tr} - \dfrac{\tilh_{\tr}}{\tilh_{\ts}} \p_{\ts} \right) \varphi \, d\tr \, d\ts\\
	&= -\int_{R} \dfrac{\tilh_{\tr}}{\tilh_{\ts}} \tgam(-\ts) \tilh_{\ts} \varphi \, d\tr \, d\ts = \int_{\tilde \Omega_{\tilde \eta}} \tpsi_{\tx} \tgam(\tpsi) \varphi \, d\ty \, d\tx
	\end{align*}
by first using the definition of the distributional derivative, applying the change of variables (\ref{ytos}) and subsequently (\ref{cov3}), performing the differentiation, utilizing the compact support of the distribution and integrating by parts, and finally changing back to the original variables.  Therefore, we have shown that $\p_{\tx}\tGamma(-\tpsi) = -\tgam(\tpsi)\tpsi_{\tx}$ in the distributional sense.  In a similar fashion, it is easy to prove the fact that
	\begin{align*}
	\p_{\ty}\tGamma(-\tpsi) = -\tgam(\tpsi) \tpsi_{\ty} \in L^{\infty}(\tDn).
	\end{align*}

It remains to verify that \ref{iii} implies \ref{ii}.  Given a solution $\tilh$ of \eqref{height function form} with $\tilh \in C_\bdd^{1,\alpha}(\overline{R}) \cap W_{\mathrm{loc}}^{2,p}(R)$, the free surface profile can be recovered by taking $\ds \tn := \tilh(\cdot,0)-1$.  Thus the fluid domain $\tilde \Omega_{\tilde \eta}$ is also known.  

Next consider the mapping $G:(\tr, \ts) \in R \mapsto (\tr, \tilh(\tr, \ts)-1) \in \mathbb{R}^2$.  By the no stagnation condition \eqref{DJ no stag} and inverse function theorem, it is easy to see that $G$ is a $C^{1,\alpha}$-diffeomorphism onto its range.  As for its inverse, the first component must be given by $(\tx, \ty) \mapsto \tx$; define $-\tpsi(\tx,\ty)$ to be its second component.  We then know that that $\tpsi=0$ on $\tilde y = \tilde \eta(\tilde x)$, $\ty =0$ and $\tpsi=-1$ on $\ty = -1$, by construction.  Moreover,
	\begin{align} \label{tpsi}
	\tpsi_{\tx}(\tx, \ty) = -\dfrac{\tilh_{\tr}(\tx, -\tpsi(\tx,\ty))}{\tilh_{\ts}(\tx,-\tpsi(\tx, \ty))} \quad \text{ and } \quad \tpsi_{\ty} = -\dfrac{1}{\tilh_{\ts}(\tx,-\tpsi(\tx, \ty))}.
	\end{align}
In particular this implies that $\tilde \psi \in W_{\mathrm{loc}}^{2,p}(\tilde \Omega_{\tilde \eta})$.  

Additionally, letting $\tilde Y := \tilde H - 1 \in C^{1,1}([-1,0])$, we see that $\tilde Y$ is monotone and thus has an inverse $-\tPsi \in C^{1,1}([-1,0])$.  The asymptotic condition \eqref{hpa} for $\tilde h$ then implies that $\tilde \psi$ has the desired limiting behavior in \eqref{dfbpa}.  The nonlinear boundary condition in \eqref{fbp} follows immediately by applying \eqref{tpsi} to the free surface boundary condition in \eqref{hp}.

Notice that by differentiating $\tpsi_{\ty}$ in (\ref{tpsi}) with respect to $\ty$, we obtain
	\begin{align*}
	\ds \tpsi_{\ty\ty} = \left(\dfrac{\tilh_{\ts\ts}}{\tilh_{\ts}^3}(\tx,-\tpsi(\tx,\ty))\right)
	\end{align*}
Similarly, if we reformulate $\tpsi_{\tx}$ in (\ref{tpsi}) to be
	\begin{align*}
	\tpsi_{\tx}(\tx,\ty)\tilh_{\ts}(\tx,-\tpsi(\tx,\ty)) = \tilh_{\tr}(\tx,-\tpsi(\tx,\ty))
	\end{align*}
and differentiate with respect to $\tx$, we obtain
	\begin{align*}
	\tpsi_{\tx\tx} = \dfrac{\tilh_{\tr\tr}}{\tilh_{\ts}} -2\dfrac{\tilh_{\tr\ts}\tilh_{\tr}}{\tilh_{\ts}^2} + \dfrac{\tilh_{\ts\ts}\tilh_{\tr}^2}{\tilh_{\ts}^3}.
	\end{align*}
Finally, combining these identities with \eqref{hp} gives 
	\[
	\tpsi_{\tx\tx} + \tpsi_{\ty\ty} = -\gam(\psi) \qquad \text{ in } \tilde{\Omega}_{\tilde \eta}. \qedhere
	\]
\QED

\subsection{Function spaces and the operator equation} 

Here and in the sequel, we drop the $\thicksim$ notation in both the stream function and the height function formulations.  Furthermore, we notice that the upstream and downstream conditions on $h$ allow us to write the function $\Gamma$ in terms of the asymptotic height function $H$ in the following way:
\be \Gamma(s) =  \frac{1}{2H_s(s)^2} - \frac{1}{2H_s(0)^2}.\label{Gamma and H relation} \ee
We can then eliminate $\Gamma$ in the height equation \eqref{hp}, giving the system 
\begin{equation} \label{heightFormulation}
  \left\{ 
  \begin{alignedat}{2}
    \left(-\frac{1+h_r^2}{2h_s^2} + \frac{1}{2H_s^2}\right)_s + \left( \frac{h_r}{h_s} \right)_r  &= 0 
    &\qquad& \textrm{in } R, \\
    \frac{1+h_r^2}{2h_s^2} - \frac{1}{2H_s^2} + \frac{1}{F^2} (h-1) &= 0 
    && \textrm{on } T, \\
    h &= 0 && \textrm{on } B,
  \end{alignedat} 
  \right. 
\end{equation}
where $B$ and $T$ are bottom and top boundaries of $R$, respectively.  It is often most convenient to work with the difference 
\[ \phi := h - H,\]
which measures the deflection of the streamlines from their asymptotic heights.  The height equation \eqref{heightFormulation} then becomes
 \begin{equation} 
    \left\{ 
    \begin{alignedat}{2}
      \left(
      -\frac{1+\phi_r^2}{2(H_s+\phi_s)^2} 
      + \frac{1}{2H_s^2}
      \right)_s
      + \left( \frac{\phi_r}{H_s+\phi_s} \right)_r 
      &= 0 &\qquad& \textrm{in } R, \\
      \frac{1+\phi_r^2}{2(H_s+\phi_s)^2} - \frac{1}{2H_s^2} + \frac{1}{F^2}  \phi &= 0 && \textrm{on } T, \\
      \phi &= 0 && \textrm{on } B.
    \end{alignedat} 
    \right. \label{phiHeightEqn}
  \end{equation}
The asymptotic conditions \eqref{hpa} likewise have the simple expression
	\begin{align} \label{phi asymptotic conditions}
	\phi, \,\phi_r, \, \phi_s \to 0 \qquad \text{as } r \to \pm \infty, \textrm{ uniformly in $s$.}
	\end{align}

Now let us fix the function space setting and introduce the operator equation.  For the domain we take
\[ {X} := \left\{ \phi \in C_{\bdd,\even}^{1,\alpha}(\overline{R}) : \phi|_B = 0 \right\} \cap C_0^1(\overline{R}) \cap W_{\mathrm{loc}}^{2,p}(R),\]
where as above $p := 2/(1-\alpha)$.  Note that both the bottom boundary condition and asymptotic condition \eqref{phi asymptotic conditions} are enforced in the definition of $X$.   As the codomain, we use the space ${Y} := {Y}_1 \times {Y}_2$ where
\[ {Y}_1 := \left\{ \partial_q F_1 + \partial_p F_2 \in \mathcal{D}_\even^\prime(R) : F_i \in C_{\bdd}^{0,\alpha}(\overline{R}) \cap C_0^0(\overline{R}) \cap W_{\mathrm{loc}}^{1,p}(R) \right\},\]
and 
\[ {Y}_2 := C_{\bdd,\even}^{0,\alpha}(T) \cap C_0^0(T) \cap W_{\mathrm{loc}}^{1-\frac{1}{p},p}(T).\]
Here $\mathcal{D}_\even^\prime(R)$ refers to the set of (even) distributions on $\Omega$.  We endow ${Y}_1$ with Banach space structure by prescribing the norm 
\begin{align*} \| f \|_{{Y}_1} := \inf\Big\{ \| F_1 \|_{C^{0,\alpha}(R)} + \| F_2 \|_{C^{0,\alpha}(R)} & + \| F_1 \|_{W_{\mathrm{loc}}^{1,p}(R)} +\| F_2 \|_{W_{\mathrm{loc}}^{1,p}(R)} : \\
&  f = \partial_q F_1 + \partial_p F_2, \, F_1, F_2 \in C_{\bdd}^{0,\alpha}(\overline{R}) \Big\}.\end{align*} 
Note that for $\phi \in C_\bdd^1(\overline{R}) \cap W_{\mathrm{loc}}^{2,p}(R)$, we indeed have $|\nabla \phi|^2 \in W_{\mathrm{loc}}^{1,p}(R)$, so that these spaces are consistent with the equations.  

Finally, define 
\[ \Udomain := \left\{ (\phi,F) \in {X} \times \mathbb{R} : \inf_R \left( \phi_s + H_s \right) >0, ~ F- F_{\mathrm{cr}} > 0 \right\}\]
which corresponds to solitary waves that are supercritical and satisfy the no stagnation condition \eqref{no stagnation condition}.    

The system \eqref{phiHeightEqn} can now be expressed as 
\be \mathscr{F}(\phi,F) = 0,\label{operatorEqn} \ee
where $\mathscr{F} = (\mathscr{F}_1, \mathscr{F}_2) : \Udomain \subset X \times \mathbb{R} \to {Y}$ is the real-analytic mapping
\be \begin{split} \label{F1&F2}
\mathscr{F}_1(\phi,F) & := \left(
      -\frac{1+\phi_r^2}{2(H_s+\phi_s)^2} 
      + \frac{1}{2H_s^2}
      \right)_s 
      + \left( \frac{\phi_r}{H_s+w_s} \right)_r \\
\mathscr{F}_2(\phi,F) & := \left( \frac{1+\phi_r^2}{2(H_s+\phi_s)^2} - \frac{1}{2H_s^2} + \frac{1}{F^2}  \phi \right)\bigg|_T.
\end{split} \ee

%
%

\section{Linear theory}
\label{linearized}

\subsection{Sturm--Liouville problem}
\label{SLP}

Let us begin by examining the eigenvalue problem for the linearized operator $\calF_{\phi}(0,F)$ restricted to variations $\dot{\phi}$ that are laminar in the sense that $\dot{\phi} = \dot{\phi}(s)$:
	\begin{align} \label{S-LProblem}
	\left \{ 
	\begin{array}{l l}
	-\left(\dfrac{\dot{\phi}_s}{H_s^3} \right)_s = \nu \dfrac{\dot{\phi}}{H_s} & -1 < s < 0\\
	\dot{\phi} = 0 & s=-1\\
	-\dfrac{\dot{\phi}_s}{H_s^3} + \mu\dot{\phi} = 0 & s=0
	\end{array}
	\right.
	\end{align}
where $\nu$ is the eigenvalue, and $\mu := 1/F^2$ is introduced for notational convenience throughout this section.  We begin by analyzing the case when $\nu =0$, and we wish to find the smallest value $\mu$ for which
		\begin{align} \label{eigenvalueProblem}
	\left \{ 
	\begin{array}{l l}
	-\left(\dfrac{\dot{\phi}_s}{H_s^3} \right)_s = 0 & -1 < s < 0\\
	\dot{\phi} = 0 & s=-1\\
	-\dfrac{\dot{\phi}_s}{H_s^3} + \mu\dot{\phi} = 0 & s=0
	\end{array}
	\right.
	\end{align}
 has nontrivial (weak) solution.  With that in mind, we consider the unique solution $\Phi$ to the initial value problem
	\begin{align} \label{IVP}
	\left \{ 
	\begin{array}{l l}
	-\left(\dfrac{\Phi_s}{H_s^3} \right)_s = 0 & -1 < s < 0\\
	\Phi = 0 & s=-1\\
	\Phi_s = 1 & s=-1
	\end{array}
	\right.
	\end{align}
along with the affine function
	\begin{align*}
	A(\mu) := -\dfrac{\Phi_s(0)}{H_s^3(0)} + \mu\Phi(0).
	\end{align*}
Notice that for fixed $\mu$, $\Phi$ solves \eqref{eigenvalueProblem} if and only if $A(\mu)=0$.  We arrive at the following lemma:

\begin{lemma}[Existence of the critical Froude number]  \label{critical Froude number lemma} There exists a unique critical Froude number $F_{\mathrm{cr}}$ given by
\be \label{def mucr} \Fcr =  \mucr :=  \left( \ds \int_{-1}^{0} H_s^3(t) \, dt \right)^{-1},\ee
such that
	\begin{enumerate}[font=\upshape,label=(\roman*)]
	\item \label{Fcr_i} the Sturm--Liouville problem \eqref{eigenvalueProblem} has a nontrivial solution $\dot{\phi} = \Phi$ for $\mu = \mucr$; and
	\item \label{Fcr_ii} $A(\mu) < 0$ for $ \mu < \mucr$ and $A(\mu) > 0$ for $\mu > \mucr$.
	\end{enumerate}
\end{lemma}


\pf  To prove \ref{Fcr_i}, first notice that we can solve \eqref{IVP} and write $\Phi$ explicitly:
	\begin{align*}
	\ds \Phi(s) = \dfrac{1}{H_s^3(-1)} \int_{-1}^s H_s^3(t) \, dt.
	\end{align*}
Using our explicit solution $\Phi$ to solve $A(\mu) = 0$ for $\mu$, we arrive at
	\begin{align*}
	\dfrac{\Phi_s(0)}{H_s^3(0)} &= \mu \Phi(0),
	\end{align*}
if and only if $\mu = \mucr$ defined in \eqref{def mucr}.    This proves \ref{Fcr_i}.  Part \ref{Fcr_ii} follows immediately from the definitions of $A(\mu)$, $\mucr$, and the fact that $\Phi(0) > 0$. 
\QED

With this in hand, we now analyze the full eigenvalue problem \eqref{S-LProblem} fixing $\mu = \mucr$.

\begin{lemma}[Spectrum of the Sturm--Liouville operator]  \label{S-L spectrum} Let $\Sigma$ denote the set of eigenvalues $\nu$ for the Sturm--Liouville problem \eqref{S-LProblem} with $\mu = \mucr$.  Then the following holds:
	\begin{enumerate}[font=\upshape,label=(\roman*)]
	\item \label{Spec_i} $\Sigma = \{ \nu_j \}_{j=0}^{\infty}$, where $\nu_j \to \infty$ as $j \to \infty$, and $\nu_j < \nu_{j+1}$ for all $j \geq 0$;
	\item \label{Spec_ii} $\nu_0 = 0$; and
	\item \label{Spec_iii} $\nu_j$ is both geometrically and algebraically simple for all $j \geq 0$.
	\end{enumerate}
\end{lemma}

\pf  Let us begin by introducing the solution $\Phi := \Phi(s;\nu)$ to the initial value problem
	\begin{align}\label{eigenvalueIVP}
	\left \{
		\begin{array}{l l}
		-\left( \dfrac{\Phi_s}{H_s^3} \right)_s = \nu \dfrac{\Phi}{H_s} & -1<s<0\\
		\Phi = 0 & \text{ for } s=-1\\
		\Phi_s=1 & \text{ for } s=-1.
		\end{array}
	\right.
	\end{align}
To see that this is valid, we make the $C^{1,1}$ change of variables
\begin{align*}
	 s  \in (-1,0) \mapsto t := \int_{-1}^{s} \dfrac{1}{H_s} \, ds \in \left(0, t_0 := \int_{-1}^0 H_s^{-1} \, ds\right), \qquad a(t) := \Phi(s),
	\end{align*}
which allows us to rewrite \eqref{eigenvalueIVP} as
	\begin{align*}\label{y-eigenvalueIVP}
	\left \{
		\begin{array}{l l}
		-\left( \dfrac{a_t}{H_s^4} \right)_t = \nu a, & 0 < t < t_0 \\
		a = 0 & \text{ for } t=0\\
		a_t=H_s(-1) & \text{ for } t=0.
		\end{array}
	\right.
	\end{align*}
Now, defining $b := a_t/H_s^4$, we see that the above problem can be formulated as the planar system
\be \label{a b planar} 
		\partial_t \begin{pmatrix} a \\ b \end{pmatrix} = \begin{pmatrix} 0 & H_s^4 \\ -\nu & 0 \end{pmatrix} \begin{pmatrix} a \\ b \end{pmatrix}, \qquad \begin{pmatrix} a(0) \\ b(0) \end{pmatrix} = \begin{pmatrix} 0 \\ 1 \end{pmatrix}.
\ee
As $H_s \in C^{0,1}$, it is now obvious that the corresponding initial value problem is well-posed and the solution depends smoothly on $\nu$.  

We next consider the associated function
	\begin{align*}
	B(\nu) := \dfrac{b(t_0;\nu)}{a(t_0;\nu)}.
	\end{align*}
The utility of $B$ lies in the observation that $\dot{\phi} = \Phi(\cdot; \nu)$ solves \eqref{S-LProblem} for $\nu$ provided
	\begin{align*}
	B(\nu) = \mucr H_s(0)^3.
	\end{align*}
Also, $B$ will have a simple pole at each eigenvalue $\nu_{\mathrm{D}}$ of the Dirichlet problem
	\begin{align} \label{DirichletIVP}
	-\left( \dfrac{\dot{\phi}_s}{H_s^3} \right)_s = \nu_{\mathrm{D}} \dfrac{\dot{\phi}}{H_s} \text{ for } -1<s<0, \quad \dot{\phi}(-1)=0, \quad \dot{\phi}(0)=0.
	\end{align}
It is immediately clear that the set of Dirichlet eigenvalues $\Sigma_{\mathrm{D}}  \subset (0, \infty)$.  Moreover, as \eqref{DirichletIVP} is a Sturm--Liouville problem, the eigenvalues are each simple and accumulate only at $+\infty$.  We are therefore justified in writing $\Sigma_{\mathrm{D}} := \{\nu_{\mathrm{D}}^{(j)} \}_{j =1}^{\infty} \subset (0, \infty)$.

Now, we have already argued that $a$ and $b$ depend smoothly on $\nu$.  Indeed,  
\[ 
		\partial_t \begin{pmatrix} a_{\nu} \\ b_{\nu} \end{pmatrix} = \begin{pmatrix} 0 & H_s^4 \\ -\nu & 0 \end{pmatrix} \begin{pmatrix} a_\nu \\ b_\nu \end{pmatrix} - \begin{pmatrix} 0 & 0 \\ 1 & 0 \end{pmatrix} \begin{pmatrix} a \\ b \end{pmatrix}, \qquad \begin{pmatrix} a_\nu(0) \\ b_\nu(0) \end{pmatrix} = \begin{pmatrix} 0 \\ 0 \end{pmatrix}.
\]
Using the above equation together with \eqref{a b planar}, we can compute $B^\prime$ via the Green's identity 
	\begin{align*}
	\ds B'(\nu) & = \dfrac{a(t_0) b_\nu(t_0) - a_\nu(t_0) b(t_0)}{a(t_0)^2} = \frac{1}{a(t_0)^2} \int_0^{t_0} \partial_t \left[ a(t) b_\nu(t) - a_\nu(t) b(t) \right] \, dt \\
	& = -\dfrac{H_s^3(0)}{\Phi(0; \nu)^2}\int_{-1}^0 \dfrac{\Phi(s;\nu)^2}{H_s(s)} \, ds < 0.
	\end{align*}
Hence $B$ is a strictly decreasing function of $\nu$.  Now, since $B$ has poles at each Dirichlet eigenvalue $\nu_{\mathrm{D}}^{(j)}$ for $j \geq 0$, it follows that $\nu \mapsto B(\nu)$ decreases from $+\infty$ to $-\infty$ on the intervals $(-\infty,\nu_{\mathrm{D}}^{(0)})$ and $(\nu_{\mathrm{D}}^{(j)}, \nu_{\mathrm{D}}^{(j+1)})$ for all $j \geq 1$.  Thus, for each such $j$, there exists a unique $\nu_j \in (\nu_{\mathrm{D}}^{(j)},\nu_{\mathrm{D}}^{(j+1)})$ for which
	\begin{align*}
	B(\nu_j) =  \mucr H_s(0)^3 = \Fcr H_s(0)^3.
	\end{align*}
Additionally, there exists at most one $\nu_0 \in (-\infty, \nu_{\mathrm{D}}^{(0)})$ for which the same holds; in light of  Lemma \ref{critical Froude number lemma}, it must in fact be true that $\nu_0 = 0$.  Recalling the definition of $B$, we see that $\{\nu_j\}_{j=0}^{\infty}$ are exactly the eigenvalues of \eqref{S-LProblem}, completing the proofs for \ref{Spec_i} and \ref{Spec_ii}.
 
To prove \ref{Spec_iii}, we notice once again that \eqref{S-LProblem} is a self-adjoint Sturm--Liouville problem, and hence has a countable set of simple eigenvalues $\{ \nu_j \}_{j=0}^{\infty}$ accumulating only at $+\infty$.
\QED

\subsection{Local properness and Fredholm index}
\label{properness}

In this section, we analyze the linearized operator $\calF_{\phi}(\phi, F)$ at an arbitrary $(\phi,F) \in \mathscr{U}$.  Specifically, our main objective is to prove that this operator has certain important compactness properties:  it is locally proper and Fredholm index $0$.  This is crucial to both the small- and large-amplitude existence theory.  Note that the assumption of supercriticality $F>F_{\textrm{cr}}$ is necessary throughout this section.  Our argument follows closely --- albeit with several modifications due to the lower regularity --- that given by Wheeler \cite{wheeler2013solitary} and Chen, Walsh, and Wheeler \cite{chen2017existence}.  

For future reference, we record the Fr\'echet derivative
\be \label{Fphi} 	\begin{split}
		\mathscr{F}_{1\phi}(\phi,F) \dot\phi &= \left( \dfrac{1}{\phi_s + H_s}\dot{\phi}_r - \dfrac{\phi_r}{(\phi_s + H_s)^2}\dot{\phi}_s \right)_r - \left( \dfrac{\phi_r}{(\phi_s+H_s)^2}\dot{\phi}_r - \dfrac{1+\phi_r^2}{(\phi_s+H_s)^3}\dot{\phi}_s \right)_s\\
		\mathscr{F}_{2\phi}(\phi,F) \dot\phi &= \left(\dfrac{\phi_r}{(\phi_s+H_s)^2}\dot{\phi}_r - \dfrac{1+\phi_r^2}{(\phi_s+H_s)^3}\dot{\phi}_s + \Froude\dot{\phi}\right) \bigg|_{T}.
	\end{split} \ee
When studying the compactness properties of elliptic problems posed on strips, it is important to consider the corresponding ``limiting operator" found by taking the limit of the coefficients as $r \to \pm \infty$.  For this problem, the limiting operator is nothing but $\mathscr{F}_\phi(0,F)$.

We begin with a preparatory result.  At  various stages in the analysis, we will need to make use of maximum principle arguments or exploit coercivity of certain bilinear forms associated to the linearized problem.  This is not possible in general, due to the zeroth order term in the boundary operator $\mathscr{F}_{2\phi}$ having the ``bad sign.'' However, using crucially the assumption of supercriticality, we can overcome this difficulty via an auxiliary function $\tilde \Phi$ --- a slight modification of the function $\Phi$ introduced in \eqref{IVP}.

\begin{lemma}\label{Phitheorem} Fix $0 < \ep \ll 1$.  There exists a unique solution $\tilde{\Phi} := \tilde{\Phi}(s; F, \ep) \in C^{1,1}([-1,0])$ to the initial value problem
	\begin{align}\label{Phi1}
	\left( \dfrac{\tilde{\Phi}_s}{H_s^3} \right)_s + \Froude\ep\tilde{\Phi} = 0 \, \text{ in } (-1,0), \quad \tilde{\Phi}(-1)=\ep, \quad \tilde{\Phi}_s(-1)=1.
	\end{align}
For supercritical waves, and for $\ep > 0$ sufficiently small,
	\begin{align}\label{Phi2}
	\tilde{\Phi} >0 \text{ for } -1 \leq s \leq 0, \quad \tilde{\Phi}_s >0 \text{ for } -1 \leq s \leq 0,
	\end{align}
and
	\begin{align}\label{Phi3}
	-\dfrac{\tilde{\Phi}_s}{H_s^3} + \dfrac{1}{F^2}\tilde{\Phi} <0 \, \text{ on } s=0.
	\end{align}
\end{lemma}

\begin{proof}  The existence and regularity of $\tilde \Phi$ follows as in the proof of Lemma \ref{S-L spectrum} It is easy to see that $\tilde{\Phi} = \Phi$ whenever $\ep = 0$.  It follows that the second inequality in \eqref{Phi2} and the inequality in \eqref{Phi3} hold for sufficiently small $\ep$.  Indeed, $\tilde{\Phi}_s$ is uniformly bounded away from zero for $\ep$ sufficiently small, and since $\tilde{\Phi} = 0$, integrating the second inequality in \eqref{Phi2} yields the first inequality.
\QED

Let $\mathbb{H}$ be the Hilbert space 
\[ \mathbb{H} := \left\{ \dot{\phi} \in H^1(R) : \dot{\phi}|_B = 0 \textrm{ in the trace sense} \right\},\]
and consider the (weak) solvability of the problem 
\[ \mathscr{F}_\phi(0,F) \dot{\phi} = \left( \partial_s f_1 + \partial_r f_2, \, g \right)\]
for $\dot{\phi} \in \mathbb{H}$, $F > F_{\mathrm{cr}}$, given $f_1, f_2 \in L^2(R)$, $g \in L^2(T)$.  Using the auxiliary function $\tilde \Phi$ and making the change of variables 
\[ \dot{\phi} =: \tilde{\Phi} v,\]
we see that this is equivalent to 
  \begin{equation}  \label{modifiedLinf}
    \left\{ 
    \begin{alignedat}{2}
 \left( \dfrac{v_r}{H_s} \right)_r + \left(\dfrac{v_s}{H_s^3} \right)_s + \dfrac{2\tilde{\Phi}_s}{\tilde{\Phi} H_s^3}v_s - \Froude\ep v  & =  \dfrac{1}{\tilde{\Phi}}\partial_s f_1  + \dfrac{1}{\tilde{\Phi}}\partial_r f_2&\qquad& \textrm{in } R, \\
      -\dfrac{v_s}{H_s^3} + \dfrac{1}{\tilde{\Phi}}\left( - \dfrac{\phi_s}{H_s^3} + \Froude\tilde{\Phi} \right)v & = \dfrac{1}{\tilde{\Phi}} g && \textrm{on } T, \\
    v & = 0 && \textrm{on } B,
    \end{alignedat} 
    \right. 
  \end{equation}
  where these equalities hold in the distributional sense.

\begin{lemma}[$H^1$ solvability] \label{weak invertibility lemma} 
  If $F > F_{\mathrm{cr}}$, then, for each $f_1, f_2 \in L^2(R)$ and $g \in L^2(T)$, there exists a unique $\dot \phi \in \mathbb H$ solving 
  \[ \F_\phi(0,F)\dot \phi = (\partial_s f_1 + \partial_r f_2 ,\, g)\]
  in the sense of distributions. Equivalently, $\dot v := \dot{\phi}/\Phia$ is a weak solution to \eqref{modifiedLinf}.
  \end{lemma}
  \begin{proof}
    By definition, a weak solution of \eqref{modifiedLinf} satisfies 
    \begin{equation} 
      \mathscr{B}[\dot v, \dot u] = \int_{R} \left( \frac{1}{\Phia} f_1 \dot u_r +  \left( \frac{\dot u}{\Phia} \right)_s f_2  \right) \, dr \, ds - \int_T \frac{1}{\Phia} (g - f_2) \dot u \, dr\label{weak v problem} 
    \end{equation}
    for each $\dot u \in \mathbb{H}$, where the bilinear form $\mathscr{B}\colon \mathbb{H} \times \mathbb{H} \to \R$ is given by
    \begin{align*}
      \mathscr{B}[\dot v, \dot u] := \int_{R} \left( \frac{1}{H_s^3} \dot v_s \dot u_s  + \frac{1}{H_s} \dot v_r \dot u_r \right) \, dr \, ds +
      \int_{T} \left( \frac{\Phia_s}{H_s^3} - \frac{1}{F^2} \Phia \right) \dot v \dot u \, dr - \int_{T} \frac{1}{H_s} \dot v_r \dot u\, dr.
    \end{align*}
    It is clear that $\mathscr{B}$ is bounded, and that the right-hand side is an element of the dual space of $\mathbb{H}$, since $\Phia \in C^{1,1}$.  On the other hand, we can estimate 
    \begin{align*} \mathscr{B}[\dot v, \dot v] &= \int_{R} \left( \frac{1}{H_s^3} \dot v_s^2  + \frac{1}{H_s} \dot v_r^2 \right) \, dr \, ds +
      \int_{T} \left( \frac{\Phia_s}{H_s^3} - \frac{1}{F^2} \Phia \right) \dot v^2 \, dr - \int_{T} \frac{1}{H_s} \partial_r ( \frac{1}{2} \dot v^2)\, dr \geq C \| \dot v \|_{\dot{H}^1(R)}^2.
    \end{align*}
    Here we have used the sign condition for $\Phia$ in \eqref{Phi3} to eliminate the second term on the right-hand side in the first line.  The bottom boundary condition $\dot{\phi} = 0$ on $B$ enables us to apply Poincar\'e's inequality to control the full $H^1(R)$ norm in terms of the homogeneous one.  This prove that $\mathscr{B}$, and hence the lemma follows from a standard application of Lax--Milgram.  \end{proof}

We next prove that $\mathscr{F}_\phi(0,F)$ is injective as a mapping between the smoother spaces $X_\bdd$ and $Y_\bdd = Y_{1,\bdd} \times Y_{2,\bdd}$ defined by 
\[ X_\bdd := \left \{ \phi \in C_\bdd^{1,\alpha}(\overline{R}) : \phi|_B = 0 \right\} \cap W_{\mathrm{loc}}^{2,p}(R),\]
and
\[ Y_{1,\bdd} = \left\{ \partial_r F_1 + \partial_s F_2  \in \mathcal{D}^\prime(R) : F_1, F_2 \in C_\bdd^{0,\alpha}(\overline{R}) \cap W_{\mathrm{loc}}^{1,p}(R) \right\}, \quad Y_{2,\bdd} := C_\bdd^{0,\alpha}(T) \cap W_{\mathrm{loc}}^{1-\frac{1}{p},p}(T). \]
The point here is that $X_\bdd$ and $Y_\bdd$ mandate the same amount of smoothness as $X$ and $Y$, respectively, but they do not require decay as $r \to \pm \infty$.  They are Banach spaces when endowed with the obvious norms.  This is the appropriate function space setting for the problem at infinity.  

\begin{lemma}[Strong injectivity] \label{strong injectivity lemma} 
  For all $F > F_{\mathrm{cr}}$, $\F_\phi(0,F)$ is injective as a mapping from $X_\bdd$ to $Y_\bdd$.
  \begin{proof}  
    Let $\dot \phi \in \Xb$ be a solution to $\F_w(0,F)\dot \phi = 0$, and put $\dot v := \dot \phi/\Phia$.  Then $\dot v$ solves \eqref{modifiedLinf} in the distributional sense for the data $f_0$, $f_1$, $f_2$, $g = 0$.   Moreover, for any $\delta > 0$, $\dot u := \sech(\delta r) \dot v$ is an element of $\mathbb H$. By an argument very similar to the one given in the proof of Lemma~\ref{weak invertibility lemma}, we can then show that conclude $\dot u \equiv 0$, provided that $\delta$ is sufficiently small. But then $\dot v \equiv 0$, and hence $\mathscr{F}(0,F)$ is injective, as claimed.
    \end{proof} 
\end{lemma}

%
%

\begin{lemma}[Local properness] \label{local properness lemma} 
  For all $F > F_{\mathrm{cr}}$, $\F_\phi(0,F) : \Xb \to \Yb$ is locally proper.  That is, for any compact set $K \subset \Yb$ and any closed and bounded set $D \subset \Xb$, $\F_\phi(0,F)^{-1}(K) \cap D$ is compact in $\Xb$.
\end{lemma}
\begin{proof}
Let $\{ (f_n,g_n) \} \subset \Yb$ be given with $(f_n,g_n) \to (f,g)$ in $\Yb$, and suppose that $\{ \dot\phi_n \} \subset \Xb$ is a uniformly bounded sequence such that $\mathscr{F}_\phi(0,F) \dot\phi_n = (f_n, g_n)$.  We must show that $\{ \dot\phi_n\}$ has a subsequence converging in $\Xb$.  

Since the coefficients of $\F_\phi(0,F)$ are independent of the horizontal variable $r$, we can use the results of Wheeler  \cite[Lemma~A.7]{wheeler2013solitary} to extract a subsequence converging in $C_{\bdd}^{1,\alpha}(\overline{R})$ to some $\dot \phi$ solving $\mathscr{F}_\phi(0,F) \dot \phi = (f,g)$.  To complete the argument, we must have also that (up to a subsequence) $\dot \phi_n \to \dot \phi$ in $W_{\mathrm{loc}}^{2,p}(R)$.  But this is a consequence of a priori estimates in Sobolev norms:  for any $M > 0$, let $R_M := R \cap \{ |r| < M \}$.  Then by virtue of the equations satisfies by $\dot \phi_n$ and $\dot \phi$, we have
\[ \| \dot \phi_n - \dot \phi \|_{W^{2,p}(R_M)} \leq C_M \left( \| \dot \phi_n - \dot \phi\|_{L^p(R_{2M})} + \| f_n - f \|_{Y_{1,\bdd}} + \| g_n - g \|_{Y_{2,\bdd}} \right),\]
for some constant $C_M > 0$.  As the right-hand side converges to $0$, we see that indeed $\dot \phi_n \to \dot \phi$ in $\Xb$.  
\end{proof}

Using this fact, we will be able to infer the surjectivity of $\F_w(0,F)$ from its weak invertibility and a limiting argument.

\begin{lemma}[Strong invertibility] \label{strong invertibility lemma} For all $F > F_{\mathrm{cr}}$, $\F_\phi(0,F) : \Xb \to \Yb$ is an isomorphism.
\end{lemma}
\begin{proof}
In view of Lemma~\ref{strong injectivity lemma}, it remains only to prove that $\F_w(0,F) : \Xb \to \Yb$ is surjective.  With that in mind, let $(f,g) \in \Yb$ be given.  By definition, this means that we can find $F_1, F_2 \in C_\bdd^{0,\alpha}(\overline{R}) \cap W_{\mathrm{loc}}^{1,p}(R)$ such that $f = \partial_s F_1 + \partial_r F_2$.  

  First, since $\F_\phi(0,F)$ has a trivial kernel by Lemma~\ref{strong injectivity lemma} and is locally proper by Lemma~\ref{local properness lemma}, a standard argument shows that it enjoys an improved Schauder estimate
  \begin{align}
    \label{schauderlater}
    \n{\dot \phi}_\Xb \le C\n{\F_\phi(0,F)\dot \phi}_\Yb.
  \end{align}
  In particular, the right-hand side contains no lower-order terms of the form $\| \dot\phi \|_{C^0(R)}$ or $\|\dot\phi\|_{L_{\mathrm{loc}}^p(R)}$.  
  
  Now, for $\delta > 0$, define
  \begin{align*}
    f^\delta := \partial_s \left(  \sech(\delta r) F_1 \right) + \partial_r \left( \sech(\delta r) F_2\right),
   \quad 
    g^{\delta} := \sech(\delta r) g \in L^2(T).
  \end{align*}
  Then we can use Lemma~\ref{weak invertibility lemma} to infer that, for each $\delta > 0$, there exists a unique solution $\dot \phi^\delta \in \mathbb{H}$ to $\F_\phi(0,F) \dot \phi^\delta = (f^\delta, g^\delta)$.  As $(f^\delta, g^\delta)$ also lies in $\Yb$ --- with a norm bounded uniformly in $\delta$ --- standard elliptic theory guarantees that $\dot \phi^\delta \in \Xb$ as well. Using the estimate \eqref{schauderlater}, we then find
  \begin{align*}
    \n{\dot \phi^\delta}_\Xb
    \le C\n{(f^\delta,g^\delta)}_\Yb
    \le C\n{(f,g)}_\Yb.
  \end{align*}
  This gives uniform in $\delta$ control of $\dot \phi^\delta$ in $C_\bdd^{1,\alpha}(\overline{R}) \cap W_{\mathrm{loc}}^{2,p}(R)$.   We can then extract a subsequence converging in $C^1_\loc(\overline R)$ to a function $\dot \phi \in C_\bdd^{1,\alpha}(\overline{R})$.  It is easy to see that $\F_\phi(0,F) \dot \phi = (f,g)$ in the sense of distributions.  By elliptic theory, $\dot{\phi}$ then has the improved regularity $\dot{\phi} \in \Xb$, which completes the proof.  
\end{proof}


\begin{lemma}[Fredholm]  \label{Fredholm} For all $(\phi, F) \in \Udomain$, the operator $\mathscr{F}_\phi(\phi,F)$ is Fredholm index $0$ both as a mapping from $X$ to $Y$ and as a mapping from $\Xb$ to $\Yb$.
\end{lemma}
\begin{proof} Let $(\phi,F) \in \Udomain$ be given.  By definition of the space $X$, we see that the coefficients of $\mathscr{F}_\phi(\phi,F)$ approach those of the limiting operator $\F_\phi(0,F)$ as $r \to \pm\infty$. We have already seen that $\F_\phi(0,F) \maps \Xb \to \Yb$ is invertible.  The lemma now follows from the soft analysis argument given in \cite[Lemmas~A.12 and A.13]{wheeler2015pressure}.
\QED

\section{Qualitative properties}
\label{quals}	

The purpose of this section is to investigate some qualitative features of solitary water waves.   In particular, we establish the equivalence between being supercritical and a wave of elevation, and then go on to prove that every supercritical wave is also symmetric and monotone.  It is important to note that these properties are not limited to small-amplitude waves, nor the specific family that we construct later in this work, but are universal among all solitary waves with this regularity.  

\subsection{Elevation}
\label{elevation}
We first endeavor to prove that a solitary wave $(\phi, F)$ is a wave of elevation if and only if $F > F_{\mathrm{cr}}$.  It turns out that this requires gaining a better understanding of the family of laminar waves that share the same vorticity function as the fixed background profile.  Here we follow the ideas of Wheeler in \cite[Section 2.1 and Section 2.2]{wheeler2013solitary}.

Note that a function $K = K(s)$ is a $r$-independent solution of the height equation \eqref{heightFormulation} provided that it satisfies the ODE
\begin{equation}  \label{laminar IVP}
    \left\{ 
    \begin{alignedat}{2}
      \left(
      -\frac{1}{2 K_s^2} 
      + \frac{1}{2H_s^2}
      \right)_s
      &= 0 &\qquad& \textrm{in } -1 < s < 0, \\
      \frac{1}{2K_s^2} -\frac{1}{2H_s^2} + \frac{1}{F^2}  (K-1) &= 0 && \textrm{on } s = 0, \\
      K &= 0 && \textrm{on } s = -1.
    \end{alignedat} 
    \right. 
  \end{equation}
We can solve \eqref{laminar IVP} explicitly to find a one-parameter family $\mathscr{T}$ of laminar flows
\be \label{definition of H(s,kappa)}
	\ds \mathscr{T} = \left\{  H(\placeholder; \kappa) : \kappa \in (-2\Gamma_{\mathrm{min}}, \infty) \right\}, \qquad H(s; \kappa) := \int_{-1}^{s} \dfrac{1}{\sqrt{\kappa + 2\Gamma(t)}} \, dt,
\ee
where we recall that $\Gamma$ is given by \eqref{Gamma and H relation} and has minimum value $\Gamma_{\mathrm{min}}$.  The family $\mathscr{T}$ includes the fixed background flow $H$:  there exists a unique value $\lambda$ such that
	\begin{align*}
	\ds 1 = H(0;\lambda) := \int_{-1}^0 \dfrac{1}{\sqrt{\lambda + 2\Gamma(t)}} \, dt,
	\end{align*}
and hence $H(s;\lambda) = H(s)$.  The asymptotic depth varies along $\mathscr{T}$; let $d^* \in [d, \infty]$ be its supremum over all these laminar flows:
	\begin{align*}
	\ds d^* := d\cdot \sup_{\kappa} H(0;\kappa).
	\end{align*}
	
Looking at the Bernoulli boundary condition in \eqref{laminar IVP}, we see that it is possible to solve for the Froude number in terms of the parameter $\kappa$.  In the next lemma, we investigate this dependence.  

\begin{lemma}  \label{A kappa} Define $\mu : (-2\Gamma_{\textrm{min}}, \infty) \to \R$ by
	\begin{align*}
	\ds \mu(\kappa) := \left \{
		\begin{array}{l l}
		\dfrac{1}{2} \left ( \dfrac{\kappa-\lambda}{1-H(0;\kappa)}\right) & \kappa \neq \lambda,\\
		\mu_{\mathrm{cr}} & \kappa = \lambda.
		\end{array}
	\right.
	\end{align*}
Then $\mu(\cdot)$ is $C^1$ and strictly increasing.  Moreover, if $d^* = \infty$, then $\lim_{\kappa \downarrow -2\Gamma_{\mathrm{min}}} \mu(\kappa) = 0$, and if $d^* < \infty$, then $\lim_{\kappa \downarrow -2\Gamma_{\mathrm{min}}} \mu(\kappa) > 0$.
\end{lemma}

\begin{proof} Clearly its the definition in \eqref{definition of H(s,kappa)},   $H(\placeholder; \kappa)$ is smooth in its dependence on $\kappa$.  Differentiating, we see that
	\begin{align*}
	\p_\kappa H(0;\kappa) = -\frac{1}{2} \int_{-1}^0 \left( \kappa + 2\Gamma(s) \right)^{-3/2} \, ds < 0.
	\end{align*}
Thus the (nondimensionalized) depth $H(0;\kappa)$ is a strictly decreasing function of $\kappa$ for  $\kappa \in (-2\Gamma_{\textrm{min}}, \infty)$.  A similar computation shows that it is also strictly convex.  Now, due to the fact that $H(0;\lambda)=1$, we have 
	\begin{align*}
	\ds \lim_{\kappa \to \lambda} \mu(\kappa) = \frac{1}{\partial_\kappa H(0; \lambda)} =  \mu_{\textrm{cr}},
	\end{align*}
where we have used the definition of $\mu_{\mathrm{cr}}$ in \eqref{def mucr}.  We conclude that $\mu$ is $C^1$.  A simple computation confirms that $\mu^\prime > 0$, and hence it is strictly increasing.  Finally, $\lambda > -2\Gamma_{\textrm{min}}$ and $d^* > d$ imply
	\begin{align*}
	\ds \lim_{\kappa \downarrow -2\Gamma_{\textrm{min}}} \mu(\kappa) = \dfrac{1}{2} \left( \dfrac{2\Gamma_{\textrm{min}} + \lambda}{d^*/d-1}\right),
	\end{align*}
which is clearly positive if $d^* < +\infty$ and $0$ if $d^* = +\infty$.
\QED

We arrive at the main theorem for this subsection.

\begin{theorem}[Elevation]  \label{supercritical elevation}    Every nontrivial solitary wave solution $(\phi, F) \in (C^{1,\alpha}(\overline{R}) \cap C_0^1(\overline{R})) \times \R$ of \eqref{phiHeightEqn} is supercritical if and only if it is a wave of elevation in the sense that
	\begin{align*}
	\phi > 0 \text{ in } \calR \cup T.
	\end{align*}
\end{theorem}
\begin{proof} That elevation implies supercriticality is due to a direct application of Theorem 1.1 in \cite{wheeler2015froude}, so we consider only the converse direction.  
Suppose that $(\phi, F)$ is a nontrivial solitary wave with $F > F_{\mathrm{cr}}$.  

We will first show that $\phi \geq 0$ on $T$, which is equivalent to having $h(r,0) \geq 1 \text{ for all } r \in \R$.  Physically, this means that the height of the free surface is everywhere above its limiting value.  By way of contradiction, assume $\inf_{r \in \mathbb{R}} h(r,0)< 1$.  Since $h(r,0) \to 1$ as $r \to \pm\infty$, we know that $h(r,0)$ must achieve its minimum at some $r_0 \in \R$.  Recalling that $\kappa \mapsto H(0;\kappa)$ is decreasing with $H(0;\lambda) =1$ and $H(0,\kappa) \to 0$ as $\kappa \to \infty$, we know that there exists a unique $\kappa_{*} > \lambda$ such that 
	\begin{align*}
	h(r_0,0) = H(0; \kappa_*) <1.
	\end{align*}
Define
	\begin{align*}
	\varphi(r,s) := h(r,s) - H(s;\kappa_*).
	\end{align*}
Taking the difference of the equations satisfied by $h$ and $H^{\kappa_*} :=H(s;\kappa_*)$, we see that $\varphi$ solves the elliptic problem
	\begin{align} \label{varphi elliptic equation}
	\left \{
		\begin{array}{l l}
		\mathscr{L}\varphi =0, & \text{ in } R\\
		\mathscr{B}\varphi = 0 & \text{ on } T\\
		\varphi = 0 & \text{ on } B
		\end{array}
	\right.
	\end{align}
where
\[
	\mathscr{L}\varphi := \left( \dfrac{\varphi_r}{h_s} \right)_r - \left( \dfrac{h_r}{2h_s^2}\varphi_r - \dfrac{h_s + H_s^{\kappa_*}}{2h_s^2(H^{\kappa_*}_s)^2}\varphi_s \right)_s,\qquad
	\mathscr{B}\varphi := \dfrac{h_r}{2h_s}\varphi_r + \dfrac{h_s + H^{\kappa_*}_s}{2h_s^2(H^{\kappa_*}_s)^2}\varphi_s + \Froude\varphi.
\]
Note that the coefficients above are $C_\bdd^{0,\alpha}(\overline{R})$, and so by Theorem~\ref{max principle}\ref{strong max principle}-\ref{hopf lemma}, the strong maximum principle and Hopf boundary point lemma can be applied. 

By design, $\varphi \geq 0$ on $T$ and $\varphi =0$ on $B$.  Since $\kappa_* > \lambda$, the asymptotic height condition \eqref{hpa} indicates that
	\begin{align*}
	\ds \lim_{r \pm \infty} \varphi(r,s) = H(s;\lambda) - H(s;\kappa) \geq 0,
	\end{align*}
uniformly in $s$.  An application of the maximum principle allows us to conclude that $\varphi \geq 0$ in $\overline{R}$.  As $\varphi(r_0,0) = 0$ and $\varphi \not\equiv 0$, we conclude from the Hopf boundary point lemma that
	\begin{align} \label{nonpositivity}
	0 > \varphi_s(r_0,0) = h_s(r_0,0) - H_s(0; \kappa_*).
	\end{align}
From \eqref{definition of H(s,kappa)} it follows that $H_s(0;\kappa_*) = \kappa_*^{-1/2}$, and 
hence by \eqref{nonpositivity}  we have $h_s(r_0,0) < \kappa_*^{-1/2}$.  However, upon comparing this to the nonlinear boundary condition in  \eqref{heightFormulation} evaluated $(r_0,0)$ we obtain
	\begin{align*}
	\frac{1}{F^2}  = \frac{1}{1-h(r_0,0)} \left( \frac{1}{2h_s(r_0,0)^2} - \frac{1}{2H_s(0)^2} \right) > \dfrac{1}{2} \dfrac{\kappa_*-\lambda}{1-H(0;\kappa_*)} = \mu(\kappa_*).
	\end{align*}
We already concluded $\kappa_* > \lambda$, and so Lemma \ref{A kappa} indicates that
	\begin{align*}
	\frac{1}{F^2} > \mu(\kappa_*) \geq \mucr,
	\end{align*}
contradicting our assumption that the Froude number is supercritical.  Therefore, $\phi \geq 0$ on $T$.

Now, we know that $\phi = 0$ on $B$ and $\phi \in C_0^1(\overline{R})$. As $h$ and $H$ are solutions of the height equation, and $h \geq H$ on $\partial R$, we can use the comparison principle for divergence form quasilinear elliptic equations \cite[Theorem 10.7]{gilbarg2001elliptic} to conclude that $\phi \geq 0$ in $R$.  Hence by the strong maximum principle, $\phi > 0$; see, for example, \cite[Theorem 2.5.2]{pucci2007book}.  Finally, we wish to show that $\phi > 0$ on $T$.  Seeking a contradiction, assume $\phi(r_0,0) = 0$ for some $r_0 \in \R$.  Since $\phi \geq 0$ in $\overline{R}$ and $\phi \not\equiv 0$, we apply the Hopf lemma to obtain
	\begin{align} \label{lambda inequality}
	\phi_s(r_0,0) = h_s(r_0,0) - \lambda^{-1/2} < 0.
	\end{align}
On the other hand, the nonlinear boundary condition of \eqref{phiHeightEqn} at $(r_0,0)$ yields
	\begin{align*}
	\lambda = h_s(r_0,0)^2,
	\end{align*}
contradicting the strict inequality \eqref{lambda inequality}.  Therefore, $\phi > 0$ in $\calR \cup T$.
\QED

\subsection{Improved regularity near the surface and bed} \label{improved regularity section}

Before going further, we must address a somewhat technical issue.   There are a number of places in the next subsection where we will need to use the Serrin corner-point lemma (see Theorem~\ref{max principle}\ref{edge point}).  Unfortunately, that is not possible in general for function that are merely $C^{1,\alpha} \cap W^{2,p}$ near the boundary.  To overcome this obstacle, we adopt the idea of Constantin--Strauss and assume that the background profile has an extra derivative of smoothness near the top and bottom.  In the next lemma, we show that via elliptic theory, the solutions of the height equation likewise inherit additional regularity.  Here we use a slightly modified version of \cite[Theorem 7]{constantin2011discontinuous}.

\begin{lemma}[Improved regularity]  \label{addt'l regularity at boundary} Any solution $h \in C_\bdd^{1,\alpha}(\overline{R})$ of the height equation \eqref{heightFormulation} exhibits the additional regularity $\partial_r^k h \in C_\bdd^{0,\alpha}(\overline{R})$, for all $k \geq 0$.  If we assume further that 
\be \frac{1}{H_s^{2}} \in C^{2,\al}([-1,-1+\delta]) \cap C^{2,\al}([-\delta,0]) \qquad \textrm{for some $\delta \in (0,1)$,} \label{extra regularity H} \ee
then $h$ is $C^{2,\al}_\bdd$ in a strip of width $\delta/2$ near both the top $T$ and the bottom $B$. 
\end{lemma}

\pf Let $h \in C_\bdd^{1,\alpha}(\overline{R})$ be given as above.  For any $\delta$, let $h^\delta := h(\cdot - \delta, \cdot)$ be a horizontal translate of $h$, and consider the finite difference $\xi^\delta := (h - h^\delta)/\delta$.  Note that the height equation is invariant under translations in $r$, and so we find that $\xi^\delta$ satisfies 
	\begin{align} \label{heightdelta}
	\left \{
		\begin{array}{l l}
		\calA\xi^{\delta} = 0 & \text{in } R\\
		\calB \xi^{\delta} = 0 & \text{on } T\\
		\xi^\delta = 0 & \textrm{on } B,
		\end{array}
	\right.
	\end{align}
where $\mathcal{A}$ is a uniformly elliptic operator in divergence and $\mathcal{B}$ is a uniformly oblique boundary operator:
\be \label{def A B operators} 
\begin{split}
	\calA \xi^\delta &:= \left \{ \dfrac{1}{h^\delta_s}\xi^\delta_r - \dfrac{h_r}{h^\delta_s h_s}\xi^\delta_s \right \}_r - \left \{ \dfrac{h^\delta_r + h_r}{2(h^\delta_s)^2}\xi^\delta_r - \dfrac{(1+h_r^2)(h^\delta_s + h_s)}{2(h^\delta_s)^2h_s^2}\xi^\delta_s \right \}_s,\\
	\\
	\calB\xi^\delta &:= \dfrac{h^\delta_r + h_r}{2(h^\delta_s)^2}\xi^\delta_r + \dfrac{(1+h_r^2)(h^\delta_s + h_s)}{2(h^\delta_s)^2h_s^2}\xi^\delta_s + \Froude\xi^\delta.
\end{split}
\ee
In particular, we note that the coefficients occurring in $\mathcal{A}$ are uniformly bounded in the $C^{0,\alpha}(R)$ norm, while those in $\mathcal{B}$ are bounded uniformly in $C^{0,\alpha}(T)$.  Applying Schauder estimates up to the boundary allows us to control 
\[ \| \xi^\delta \|_{C^{1,\alpha}(R)} < C \| \xi^\delta \|_{C^0(R)} < C\| h \|_{C^0(R)},\]
for some constant $C > 0$ independent of $\delta$; see Theorem \ref{schauder theorem}.  It follows that we can extract a subsequence converging in the limit $\delta \to 0$ locally in $C^1(\overline{R})$.  We therefore conclude that $h_r \in C^1(\overline{R})$.  On the other hand,  taking the limit $\delta \to 0$ in \eqref{heightdelta} reveals that $h_r$ satisfies $\mathscr{F}_\phi(h-H,F) h_r = 0$, and thus by elliptic theory it enjoys the improved regularity $h_r \in C_\bdd^{1,\alpha}(\overline{R})$ according to Theorem~\ref{schauder theorem}.  Repeating this procedure, we find that $\partial_r^k h$, for any $k \geq 0$.  

Next, suppose that $1/H_s^2$ has the additional regularity in the statement of the lemma.  We will only consider the situation near the free surface, as the argument for the bed is in fact simpler.  With that in mind, we introduce the infinite strips $S_1 := \mathbb{R} \times (s_1, 0)$ and $S_2 := \mathbb{R} \times (2s_1, 0)$.  By assumption, if $s_1$ is sufficiently close to 0, then
	\begin{align*}
	\left( \dfrac{1}{H_s^2} \right)_s \in C^{1,\al}\left([2s_1,0]\right).
	\end{align*}
For each $0 < \epsilon \ll 1$, let us redefine $h^\epsilon$ to be the downward translated height function $h^\epsilon := h(\cdot, \cdot - \epsilon)$, and study the finite difference $\zeta^{\ep} :=(h- h^\epsilon)/\ep$.  As before, we see that $\zeta^\epsilon$ solves a uniformly elliptic problem
	\begin{align} \label{heightepsilon}
	\left \{
		\begin{array}{l l}
		\calA\zeta^{\ep} = -\dfrac{1}{2\ep}\left( \dfrac{1}{(H_s^{\ep})^2} - \dfrac{1}{H_s^2} \right)_s & \text{in } S_2\\
		\calB \zeta^{\ep} = 0 & \text{on } T.
		\end{array}
	\right.
	\end{align}
The right-hand side above lies in $C_\bdd^{0,\alpha}([2s_1,0])$ by hypothesis, and so we may apply Schauder estimates up to $T$ to estimate
\[
	\| \zeta^{\ep} \|_{C^{1,\al}(S_1)} \leq C \left( \| h \|_{C^1(S_2)} + \| H_s^{-2}\|_{C^{2,\alpha}([2s_1,0])} \right),
\]
for some constant $C > 0$ independent of $\epsilon$. Thus $\zeta^{\ep}$ has $C^{1,\al}(\overline{S_1})$ norm bounded uniformly in $\epsilon$.  It follows that $\zeta^\ep$ has a convergent subsequence in $C^{1}(\overline{S_1})$ as $\epsilon \to 0$, meaning $h_s \in C^1(\overline{S_1})$. But, taking the limit $\epsilon \to 0$ in \eqref{heightepsilon}, we see also that $\mathscr{F}_{1\phi}(h-H,F)h_s = 0$ in $S_1$ and $\mathscr{F}_{2\phi}(h-H,F) h_s = 0$ on $T$. Schauder theory once more allows us to recover the extra H\"older continuity: $h_s \in C^{1,\alpha}(\overline{S_1})$.  \QED	

\begin{remark} \label{improved regularity U* remark}  The hypothesis \eqref{extra regularity H} in Lemma \ref{addt'l regularity at boundary} that $1/H_s^2$ is $C^{2,\alpha}$ near the top and bottom is equivalent to having $\Gamma \in C^{2,\alpha}$ near $s = -1$ and $s = 0$, and also $U^* \in C^{2,\alpha}$ near $y = 0$ and $y = -1$ as in assumption \eqref{regularity U*} in the main theorem.  
\end{remark}

\subsection{Symmetry and monotonicity}
\label{symm}

Next, we turn our attention to proving that supercritical solitary waves with discontinuous vorticity are symmetric and monotone.  As we have already shown that every such wave is a wave of elevation, we are effectively trying to establish certain symmetry properties for positive solutions to a quasilinear elliptic problem posed on an infinite cylinder.  This is a classical topic that has been studied extensively.  Our basic approach is to use a moving planes method in the spirit of Li \cite{li1991monotonicity}, and specifically the adaptation of his technique to water waves by Maia \cite{maia1997symmetry}.   Compared to these works, the difficulty we face is the reduced regularity and the presence of nontrivial dynamics at infinity in the form of the background current.  Here we will see the necessity of the additional smoothness assumption \eqref{extra regularity H}.  

\begin{theorem}[Symmetry and monotonicity] \label{symmetrythm} Suppose that the background flow $H$ exhibits the additional regularity \eqref{extra regularity H}. Let $(\phi,F) \in (C^{1,\alpha}(\calRbar) \cap C_0^1(\overline{R}) )\times \R$ be a supercritical solution of \eqref{phiHeightEqn}.  Then, up to a translation in $r$,  $\phi$ is symmetric and monotone; that is, there exists $r_* \in \R$ such that $r \mapsto \phi(r, \cdot)$ is even about $\{r=r_*\}$ and
	\begin{align*}
		\pm \phi_r > 0 \qquad \text{ for } \pm(r_* - r) >0, -1 < s \leq 0.
	\end{align*}
\end{theorem}	
	
For this analysis, it is more natural to work with the height function $h$ directly.  We start by considering the reflected functions
	\begin{align*}
	\hlam(r,s) := h(2\lam - r, s)
	\end{align*}
where $\{r = \lam\}$ is the axis of reflection.  We also introduce the difference
	\begin{align*}
	\vlam(r,s) := \hlam(r,s) - h(r,s).
	\end{align*}
Clearly, the line $\{r=r_*\}$ is an axis of even symmetry for $h$ if and only if $v^{r_*} \equiv 0$.

Define the $\lam$-dependent sets
	\begin{align*}
	T^{\lam} := \{ (r,0) : r<\lam\}; \quad \calR^{\lam} := \{(r,s) \in \calR : r< \lam\}; \quad B^{\lam} := \{ (r,-1) : r< \lam\}.
	\end{align*}


If $h$ solves the height equation \eqref{heightFormulation}, then for each $\lam$, $\vlam$ solves
	\begin{align} \label{vlameqn}
		\left \{
			\begin{array}{l l}
			\calA \vlam =0 & \text{ in } \calR^{\lam},\\
			\calB \vlam = 0 & \text{ on } T^{\lam},\\
			\vlam = 0 & \text{ on } B^{\lam},
			\end{array}
		\right.
	\end{align}
where $\mathcal{A}$ and $\mathcal{B}$ are the operators introduced in \eqref{def A B operators} with $h^\lambda$ replacing $h^\delta$.
Observe that $\mathcal{A}$ is a divergence form elliptic operator with $C_\bdd^{0,\alpha}(\overline{R})$ coefficients, while $\mathcal{B}$ is a uniformly oblique boundary operator whose coefficients lie in $C_\bdd^{0,\alpha}(T)$. 
\begin{lemma}\label{monotonicity} Under the hypotheses of Theorem \ref{symmetrythm}, there exists $K>0$ such that
	\begin{align}\label{Kcondition1}
	\vlam &\geq 0 \, \text{ in } \calR^{\lam} \quad \text{ for all } \lam < -K, \text{ and }\\
	h_r &\geq 0 \, \text{ in } \calR^{\lam} \quad \text{ for all } \lam < -K.
	\end{align}
\end{lemma}	
\begin{proof} It is important to note that, once again, the coefficient of the zeroth order term in $\calB$ has the wrong sign for maximum principle arguments.  In the supercritical regime, we can bypass this difficulty by modifying the problem using an auxiliary function like in Section \ref{properness}.  Let $\Phi$ be defined as in Lemma \ref{Phitheorem}; in particular, recall that
	\begin{align}
	\left( -\dfrac{1}{H_s^3}\Phi_s + \Froude\Phi \right)\Big|_{s=0} < 0 \quad \text{ for } 0 < \Froude < \Fcr.
	\end{align}
We also choose $\ep$ sufficiently small so that $\Phi > \ep$ on $(-1,0)$ and we may now define $\ulam$ by $\vlam = \ulam\Phi$.  Recalling that $\Phi$ is only a function of $s$, a straightforward calculation shows that
	\begin{align*}
	0 &= \calA \vlam = (\calA' \ulam)\Phi - \left[\dfrac{2\hlam_r}{(\hlam_s)^2} \Phi_s \right] \ulam_r +\left[ \dfrac{2(1+(\hlam_r)^2)}{(\hlam_s)^3} \Phi_s \right] \ulam_s + \calZ\ulam =: \tilde{\calA} \ulam,
	\end{align*}
where $\calA'$ is the principal part of $\calA$:
	\begin{align*}
	\calA':= \dfrac{1}{\hlam_s}\p_r^2 - \dfrac{2\hlam_r}{(\hlam_s)^2}\p_r\p_s + \dfrac{1+(\hlam_r)^2}{(\hlam_s)^2}\p_s^2,
	\end{align*}
and
	\begin{align*}
	\calZ := \calA\Phi &= \dfrac{1+(\hlam_r)^2}{(\hlam_s)^3}\Phi_{ss} + \left(\dfrac{1}{2H_s^2}\right)_s \left[ \dfrac{(\hlam_s)^2 + \hlam_sh_s + h_s^2}{(\hlam_s)^3} \right]\Phi_s.
	\end{align*}
A similar computation reveals that
	\begin{align*}
	0 = \calB \vlam = (\calB' \ulam)\Phi + (\calB' \Phi)\ulam =: \tilde{\calB}\ulam,
	\end{align*}
where $\calB'$ is the principal part of $\calB$:
	\begin{align*}
	\calB' := \dfrac{\hlam_r + h_r}{2h_s^2}\p_r - \dfrac{(\hlam_s + h_s)(1+(\hlam_r)^2)}{2(\hlam_s)^2h_s^2}\p_s.
	\end{align*}
Therefore, $\ulam$ solves the PDE
	\begin{align}\label{ulamPDE}
	\left \{
		\begin{array}{l l}
		\tilde{\calA} \ulam =0 & \text{ in } \calR^{\lam},\\
		\tilde{\calB} \ulam =0 & \text{ on } T^{\lam},\\
		\ulam = 0 & \text{ on } B^{\lam}.
		\end{array}
	\right.
	\end{align}
	
We claim that there exists $K >0$ large enough that $\ulam \geq 0 \, \text{ in } \calR^{\lam}$ for all $\lam \leq -K$.  By way of contradiction, assume that for every $K$, there exists some $\lamo \leq -K$ such that $u^{\lamo}$ takes on a negative value in $\calR^{\lamo}$.  Since $h$ is a wave of elevation, $h^{\lam}$ must also be a wave of elevation for any $\lam$.  Now we know that $\ulam = 0$ on $\{r=\lam\}$, and 
	\begin{align*}
	\ulam = \dfrac{\hlam -h}{\Phi} > \dfrac{H-h}{\Phi}.
	\end{align*}
Furthermore, as $r \to -\infty$, we must have that
	\begin{align*}
	\dfrac{H-h}{\Phi} \to 0.
	\end{align*}
Therefore, $u^{\lamo}$ taking on a negative value in $\calR^{\lamo}$ implies that there exists a point $(r_0, s_0) \in \calR^{\lamo} \cup T^{\lamo}$ such that
	\begin{align}\label{ulaminf}
	\ds u^{\lamo}(r_0, s_0) = \inf_{\calR^{\lamo}} u^{\lamo} <0.
	\end{align}
	
 \tbf{Case I:}  Suppose first that ${(r_0,s_0) \in \calR^{\lamo}}$. Then $u^{\lamo} \in C^1(\overline{R})$ attains its global minimum at the interior point $(r_0,s_0)$, and thus
 	\begin{align}\label{gradvlamo}
	0 &= \nabla u^{\lamo}(r_0,s_0) = \left[ \dfrac{\nabla v^{\lamo}}{\Phi} + v^{\lamo}\left(0, \dfrac{-\Phi_s}{\Phi^2} \right) \right](r_0,s_0) =\left [ \dfrac{\nabla v^{\lamo}}{\Phi} - v^{\lamo}\left(0, \dfrac{\Phi_s}{\Phi^2}\right) \right](r_0, s_0).
	\end{align}
Recalling that $\phi  \in C_0^1(\overline{R})$, we know that for each $\delta >0$, there exists a $K$ sufficiently large enough so that
	\begin{align*}
	|h(r_0,s)-H(s)| < \delta \quad \text{ on } \calR^{-K}.
	\end{align*}
Additionally, from the chain of inequalities
	\begin{align*}
	H(s_0) < h^{\lamo}(r_0,s_0) < h(r_0,s_0) < H(s_0)+\delta
	\end{align*}
we conclude that $|v^{\lamo}(r_0,s_0)| < \delta$.  Furthermore, \eqref{gradvlamo} allows us to obtain bounds on $\nabla v^{\lamo}(r_0,s_0)$:
	\begin{align*}
	|\nabla v^{\lamo}(r_0,s_0)| = \left | \dfrac{\Phi_s(s_0)}{\Phi(s_0)} v^{\lamo}(r_0,s_0) \right| < C\delta,
	\end{align*}
with $C$ depending on $\ep$.  From this we may conclude that (for K large enough)
	\begin{align*}
	\calZ(s_0) = \left( \dfrac{\Phi_s}{H_s^3} \right)_s\Big|_{s_0} + \mathcal{O}(\delta).
	\end{align*}
	
But conditions \eqref{Phi1} and \eqref{Phi2} guarantee that $\calZ(s_0) < 0$ for $K$ sufficiently large.  Hence $\calZ$ has the correct sign, and an application of the maximum principle, Theorem~\ref{max principle}\ref{strong max principle}, to  \eqref{ulamPDE} at $(r_0,s_0)$ leads to a contradiction.

 \tbf{Case II:} Assume instead that $(r_0,s_0)=(r_0,0) \in T^{\lamo}$.  By the Hopf lemma, 
	\[ u^{\lamo}_r(r_0,s_0) = 0 \, \text{ and } u^{\lamo}_s(r_0,0) >0;\]
see Theorem~\ref{max principle}\ref{hopf lemma}
This, of course, further implies that
	\begin{align}\label{est1}
	h^{\lamo}_r(r_0,0) = h_r(r_0,0)
	\end{align}
and a reiteration of the argument in the previous case guarantees that for $K$ sufficiently large,
	\be \label{est2} 
	|h(r_0,0)-H(0)| < \delta,\qquad
	 |h_s(r_0,0)-H(0)| < \delta, \qquad
	  |h^{\lamo}(r_0,0) - H(0)| < \delta.
 \ee
From the Bernoulli boundary condition in \eqref{heightFormulation} evaluated at $(r_0,0)$, we see that
	\begin{align*}
	\dfrac{1+(h^{\lamo}_r)^2}{2(h^{\lamo}_s)^2} + \Froude h^{\lamo} = \dfrac{1+h_r^2}{2h_s^2} + \Froude h.
	\end{align*}
Then \eqref{est1} and \eqref{est2} yield the estimate
	\begin{align*}
	|v^{\lamo}_r(r_0,0)| = |h^{\lamo}_r(r_0,0) - h(r_0,0)| < C\delta.
	\end{align*}
Together, these deductions show that 
	\begin{align*}
	\tilde{\calB}u^{\lamo} &= -\left ( \dfrac{(h^{\lamo}_s + h_s)( 1+ (h_r^{\lamo})^2)}{2(h_s^{\lamo})^2h_s^2} u^{\lamo}_s \right) \Phi + \left( -\dfrac{1}{H^3}\Phi_{ss} + \Froude \Phi + \bigO(\delta) \right) u^{\lamo} = 0,
	\end{align*}
at  $(r_0, 0)$.  But recalling the sign of $\Phi$ and the equation it satisfies, this means that the coefficient of $u^{\lamo}$ above is negative, and hence \eqref{ulaminf} implies that $u^{\lamo}_s(r_0,0) >0$, a contradiction.		\QED

With Lemma \ref{monotonicity} in hand, we may finally turn to the proof of Theorem \ref{symmetrythm}.

\begin{proof}[Proof of Theorem \ref{symmetrythm}]  Consider the set
	\begin{align*}
	\Lambda := \left\{\lamo : \vlam > 0 \, \text{ in } \calR^{\lam}, \forall \lam < \lamo \right\},
	\end{align*}	
which is nonempty in light of Lemma \ref{monotonicity}.  Then we may define
	\begin{align*}
	\lamstar := \sup \Lambda,
	\end{align*}
and consider two cases.

Assume first that  $\lamstar = +\infty$.  Retracing definitions, we must then have $\vlam \geq 0$ in $\calR^{\lam}$ for all $\lam$.  Since $\vlam$ solves \eqref{vlameqn} in $\calR^{\lam}$ and vanishes at infinity, using the strong maximum principle we can strengthen this to $\vlam > 0$ in $\calR^{\lam}$ for all $\lam$.  This implies that as $|r|\to \infty$, $h \to H$ and $\hlam \to H^{\ast}$, where $H^{\ast} > H$.  But this contradicts the fact that $\phi \in C_0^1(\calR)$.

 We may therefore suppose that $\lamstar < +\infty$.  By the continuity of $\vlam$, this means
	\begin{align*}
	v^{\lamstar} \geq 0 \, \text{ in } \calR^{\lamstar}.
	\end{align*}
Recall that $v^{\lamstar}$ satisfies the elliptic equation \eqref{vlameqn} with $\lam = \lamstar$; an application of the strong maximum principle tells us that either $v^{\lamstar} > 0$ in $R^\lamstar$, or else $v^\lamstar \equiv 0$ in $R^\lamstar$.   In fact, it must be the latter.  To see this, we argue by contradiction:  assume $v^{\lamstar} > 0 \text{ in } \calR^{\lamstar}$; then there exists a sequence $\{ \lam_k \}_{k=1}^{\infty}$  with $\lam_k \searrow \lamstar$ and a sequence of points $\{ (r_k, s_k ) \}_{k=1}^{\infty} \subset \overline{\calR}^{\lam_k}$ such that
	\begin{align*}
	\ds v^{\lam_k}(r_k,s_k) = \inf_{\calR^{\lam_k}} v^{\lam_k} < 0.
	\end{align*}

Since $v^{\lam_k} =0$ on $B^{\lam_k}$ and $\{ r=\lam_k \}$, the strong maximum principle guarantees that $(r_k,s_k) \in T^{\lam_k}$, implying that
	\begin{align*}
	\vlamk_s(r_k,0) \leq 0 \quad \text{ and } \quad \vlamk_r(r_k,0) = 0.
	\end{align*}
We want to show that $\{ r_k \}$ is bounded from below.  If $\{ r_k \}$ were not bounded from below, then we can assume that $r_k < -K$ for some $K$ satisfying \eqref{Kcondition1}.  Consider $u^{\lam_k} := \vlamk/\Phi$ as in Lemma \ref{Phitheorem}.  	This implies that
	\begin{align*}
	\tilde{\calB}u^{\lam_k}(r_k,0) > 0,
	\end{align*}
a contradiction.	
	
Hence $\{r_k\}$ is bounded from below by some such $-K$. Invoking Bolzano--Weierstrass, we can extract a subsequence $\{r_{k_j}\}_{j=1}^{\infty}$ such that
	\begin{align*}
	(r_{k_j},0) \to (r_{\star},0) \text{ in } \overline{T}^{\lamstar} \text{ as } j \to \infty,
	\end{align*}
for some $r_{\star} \in [-K, \lamstar]$.  We know that $\vlamstar >0 \text{ in } \calR^{\lamstar}$, which implies that
	\begin{align*}
	\ds \lim_{j \to \infty} v^{\lam_{k_j}}(r_{k_j},0) = \vlamstar(r_{\star},0).
	\end{align*}

If $r_{\star} < \lamstar$, we have, by continuity, that
	\begin{align*}
	\vlamstar(r_{\star},0) = \vlamstar_r(r_{\star},0) =0.
	\end{align*}
Now, an application of the Hopf lemma shows that $\vlamstar_s(r_{\star},0) < 0$.  Reconsidering the operator $\calB$, we know that 
\begin{align*}
	0= (\calB\vlamstar)(r_{\star},0) = -\dfrac{(\hlamstar_s + h_s)(1+(\hlamstar_r)^2)}{2h_s^2(\hlamstar_s)^2}\vlamstar_s(r_{\star},0) > 0,
	\end{align*}
which is impossible.  

Therefore, $(r_{\star},0)$ must be a corner point of $\calR^{\lamstar}$, that is, $r_{\star} = \lamstar$.  We will now proceed to show that the Serrin corner-point lemma is violated.  Note that this is valid in view of Lemma \ref{addt'l regularity at boundary}, which ensures that we have classical second derivatives of $h$ near $T$.  Since $\vlamstar_r(r_{\star},0) = 0 = \vlamstar(r_{\star},0)$, we have that $\hlamstar(r_{\star},0) =0$.  We may rewrite the top boundary condition by clearing the denominators:
	\begin{align} \label{Bvlammodified}
	0= (\hlam_s)^2(\hlam_r + h_r)\vlam_r - (\hlam_s + h_s)(1+(\hlam_r)^2)\vlam_s + 2\Froude h_s^2(\hlam_s)^2\vlam.
	\end{align}
Evaluating at $\lambda = \lamstar$, differentiating with respect to $r$, and then evaluating at $(r_{\star},0) = (\lamstar,0)$, we see that
	\begin{align*}
	\vlamstar_{rs}(\lamstar,0) = 0,
	\end{align*}	
where we have made uses of the fact that $h_s > 0$, and the identities
	\begin{align*}
	\hlamstar_r(\lamstar,0) =-h_r(\lamstar,0), \quad \hlamstar_s(\lamstar,0) = h_s(\lamstar,0), \quad \text{ and } \hlamstar_{sr}(\lamstar,0) = h_{rs}(\lamstar,0).
	\end{align*}				
Furthermore, since $\vlamstar(\lamstar,\cdot) \equiv 0$, it follows that
	\begin{align*}
	\vlamstar_s(\lamstar,0) = \vlamstar_{ss}(\lamstar,0) =0.
	\end{align*}
Lastly, notice that we may solve for $\vlamstar_{rr}$ in the equation $\calA\vlamstar =0$ in a strip near the top, which gives
	\begin{align*}
	\vlamstar_{rr} &= \hlamstar_s \Bigg( \dfrac{2\hlamstar_r}{(\hlamstar_s)^2} \vlamstar_{rs} - \dfrac{1+(\hlamstar_r)^2}{(\hlamstar_s)^3}\vlamstar_{ss} - \dfrac{h_ss(\hlamstar_r + h_r) - 2\hlamstar_sh_{rs}}{(\hlamstar_s)^3}\vlamstar_r \\
	& \qquad \qquad \qquad \qquad \qquad \qquad \qquad \qquad - \left(\dfrac{1}{H_s^2}\right)_s\dfrac{(\hlamstar_s)^2 + \hlamstar_sh_s + h_s^2}{(\hlamstar_s)^2} \vlamstar_s \Bigg ).
	\end{align*}
Evaluating at $(\lamstar,0)$, we see that
	\begin{align*}
	\vlamstar_{rr}(\lamstar,0) =0.
	\end{align*}
Hence $\vlamstar$ and all of its derivatives up to second order vanish at $(\lamstar,0)$.  Since $\vlamstar$ solves \eqref{vlameqn} in $\calR^{\lamstar}$, this violates the Serrin corner-point lemma, which concludes that either the first or second order derivatives (in the direction of the unit normal at the surface) must be negative; see Theorem~\ref{max principle}\ref{edge point}.  Having arrived at a contradiction, we infer that $\vlamstar \equiv 0 \text{ in } \calR^{\lamstar}$ and hence $h$ and $\phi$ are symmetric about the axis $\{ r= \lamstar \}$.

Now we wish to consider the monotonicity of $h$.  Once again, for $\lam < \lamstar$, we have $\vlam >0 \text{ in } \calR^{\lam}$.  Notice that $\vlam$ vanishes on $\{r=\lam\}$, and so it attains its minimum on $\overline{\calR}^{\lam}$ at each point along this line.  Yet another use of the Hopf boundary point lemma yields the strict inequality
	\begin{align*}
	 \vlam_r(\lam,s) < 0, \qquad \textrm{for all } \lambda < \lamstar, ~-1 < s < 0,
	\end{align*}
and further that
	\begin{align} \label{vlamtohlam}
	h_r(\lam,s) = -\dfrac{1}{2} \vlam_r(\lam,s) > 0 \text{ for } \lam < \lamstar, \, -1<s<0,
	\end{align}
where we used the fact that $\hlam_r(\lam,s) = -h_r(\lam,s)$.  On the top boundary $T^{\lamstar}$, we know that $h_r \geq 0$ by continuity.  By way of contradiction, assume that $h_r(\lam,0)=0$ for some $(\lam,0) \in T^{\lamstar}$.  Then using \eqref{Bvlammodified}, differentiating with respect to $r$ and evaluating at $(\lam,0)$, and using the identities
	\begin{align*}
	h_r(\lam,0) = -\hlam_r(\lam,0), \qquad \vlam(\lam,0) = \vlam_r(\lam,0) = 0;\\
	h_s(\lam,0) = \hlam_s(\lam,0), \qquad h_{rs}(\lam,0)=-\hlam_{rs}(\lam,0),
	\end{align*}	
we have
	\begin{align*}
	0= 2h_s(\lam,0)\vlam_{rs}(\lam,0),
	\end{align*}
which implies that $\vlam_{rs}(\lam, 0) = 0$.  Once again, working in a neighborhood of $T^\lambda$, we can express $\vlam_{rr}$ using the equation $\mathcal{A} v = 0$ to find that $\vlam_{rr}(\lam,0)=0$.  As we argued earlier, this leads to a violation of the Serrin corner-point lemma since $\vlam$ satisfies \eqref{vlameqn}.  We have therefore proved that $h_r >0$ on $T^{\lamstar}$. But then \eqref{vlamtohlam} implies that $h_r > 0$ on all of $\calR^{\lamstar} \cup T^{\lamstar}$.  The same is true for $\phi_r$, and hence, $\{r=\lamstar\}$ is an axis of even symmetry, as desired.
\QED

\section{Small-amplitude waves}
\label{small amp}

In this section, we will establish the existence of small-amplitude solitary waves, which recall in our notation corresponds to a pair $(\phi,F)$ satisfying $\calF(\phi,F)=0$.  The main result is the following.

\begin{theorem}(Small-amplitude solitary waves) \label{small-amp}  Under the hypotheses of Theorem~\ref{main theorem}, there exists a curve of small-amplitude solitary waves 
	\begin{align*}
	\mathscr{C}_{\mathrm{loc}}:= \{(\phi^{\epsilon},F^{\epsilon}) : \epsilon \in (0,\epsilon_{*}) \} \subset  W^{2,p}(R)  \times (F_{\mathrm{cr}}, \infty) \subset \Udomain,
	\end{align*}
where each $(\phi^{\ep}, F^{\ep})$ is a solution of \eqref{phiHeightEqn} with the following properties:
	\begin{enumerate}[label=$\mathrm{(\roman*)}$]
		\item \label{existence_continuity} \emph{(Continuity)} The map 
		\[  (0, \epsilon_*) \ni \epsilon \mapsto (\phi^{\ep},F^\ep) \in X \times \mathbb{R}\]
		 is continuous and $\| \phi^\epsilon \|_{X} \to 0$, $F^\epsilon \to F_{\mathrm{cr}}$ as $\ep \to 0$.
		\item \label{existence_invertibility} \emph{(Invertibility)}  The linearized operator $\mathscr{F}_{\phi}(\phi^{\ep},F^{\ep})$ is invertible as a mapping $X \to Y$ for each $\ep \in (0, \ep_{*})$.
		\item \label{existence_uniqueness} \emph{(Uniqueness)} If $(\phi,F) \in \Udomain$ satisfy $\mathscr{F}(\phi, F) = 0$, with $\| \phi \|_{X}$ and $F- F_{\mathrm{cr}}$ sufficiently small, then there exists a unique $\ep \in (0,\ep_{*})$ such that $\phi = \phi^\epsilon$ and $F = F^\epsilon$.  
	\end{enumerate}
\end{theorem}

As we mentioned in Section \ref{main result section}, waves of this form were constructed previously by Hur in \cite{hur2008solitary} for $\gam \in C^0$ as well as Groves and Wahl\'en in \cite{groves2008vorticity} for $\gam \in H^1$.   The analogues of statements (i), (ii), and (iii) above were first obtained by Wheeler, who worked with $\gam \in C^{1,\al}$ for some $\al \in (0,\frac{1}{2}]$; see \cite[Theorem 4.1]{wheeler2013solitary}.  

Like Groves and Wahl\'en, our approach is to reformulate the height equation as a spatial dynamical system with Hamiltonian structure.   This procedure is undertaken in Section \ref{hamiltonian}.  Over the course of Section \ref{COV}, we  then perform several changes of variable that transform Hamilton's equations into an evolution equation $u_r = Lu + N^{\ep}(u)$, where $r$ acts as the time variable, and $N^\epsilon$ is a quadratic nonlinearity.  

We saw in Section \ref{linearized} that there is a special value of the Froude number, $F_{\mathrm{cr}}$, at which the minimum eigenvalue of a Sturm--Liouville problem associated to the linearized operator $\mathscr{F}_\phi(0,F)$ passes through the origin.  In the Hamiltonian framework, it will turn out that this corresponds to a so-called $0^2$ resonance:  when $F$ is slightly subcritical, $L$ has precisely two imaginary eigenvalues, and as $F$ increases through $F_{\mathrm{cr}}$, they collide at the origin, then leave along the real axis.  

With that in mind, our plan is to show that, at the point of collision, there is a two-dimensional center manifold for the system.  Small-amplitude solitary waves are homoclinic orbits of the spatial dynamical problem, and hence they can be constructed by analyzing the planar Hamiltonian system that results from projecting down to the center manifold.  More precisely, we will utilize \cite[Theorem 4.1]{buffoni1996plethora}, which is a version of the center manifold reduction theorem specifically designed to exploit the extra structure of Hamiltonian systems; see Appendix~\ref{quoted CMR}.   In many respects, our argument is quite similar to that of Groves and Wahl\'en.  Indeed, they only assume the additional reguality $\gamma \in H^1$ when they begin to consider the spectral theory.  For that reason, we will omit some details and focus on the areas where new ideas are needed.
\subsection{Formulation as a Hamiltonian system}
\label{hamiltonian}

Using $\phi = \phi(r,s)$ as before, let us introduce the variable
	\begin{align} \label{w defn}
	w := \dfrac{\phi_r}{\phi_s+H_s}.
	\end{align}
Throughout this section, we will frequently suppress the dependence of $(\phi,w)$ on $r$, thinking of them instead as $C^1$ functions of $r$ taking values in Hilbert spaces of $s$-dependent functions.  Specifically, we work with two such spaces
	\begin{align*}
	\calX &:=\{ (\phi,w) \in H^1((-1,0)) \times L^2((-1,0)) : \phi(-1)=0 \}\\
	\calY &:=\{(\phi,w) \in H^2((-1,0)) \times H^1((-1,0)) : \phi(-1)=0 \}.
	\end{align*}
In order to enforce the no stagnation condition, we work in the set
	\begin{align*}
	\calM := \{(\phi,w) \in \calY: \phi_s(s)+H_s(s) >0 \text{ for each } s \in [-1,0] \}.
	\end{align*}
As $\calY$ is both dense and smoothly embedded in $\calX$,   $\calM$ is a \tit{manifold domain} in $\calX$ (see, for example, \cite{mielke1991book}).  We introduce a symplectic form $\symplecticform: T\mathcal{M} \times T\mathcal{M} \to \R$ defined by
	\begin{align*}
	\ds \symplecticform|_{m} \left((\dot\phi_1,\dot w_1),\, (\dot\phi_2,\dot w_2)\right) :=  \int^{1}_{0} (\dot w_2\dot\phi_1-\dot w_1\dot\phi_2) \, ds \qquad \textrm{for all } m \in \mathcal{M}, ~(\dot \phi_i, \dot w_i) \in T_{m}\mathcal{M}.
	\end{align*}
Here, and in the remainder of this section, we abuse notation by identifying the tangent space $T\mathcal{M}$ with $\mathcal{M}$ itself.  We also note that $\Upsilon$ is in fact independent of the base point.   Consider the Hamiltonian $\mathcal{H} \in C^\infty(\mathcal{M}; \mathbb{R})$ given by
	\begin{align} \label{phiHam}
	\ds \Ham(\phi, w) &:= \int^{0}_{-1} \left ( \dfrac{1}{2} \left (w^2- \dfrac{1}{(\phi_s+H_s)^2} \right) - \dfrac{1}{2H_s^2} \right)(\phi_s + H_s) \, ds \, + \dfrac{1}{2F^2} \phi(0)^2.
	\end{align}
Computing the variations of $\mathcal{H}$, we find that 
	\begin{align*}
	\ds \Ham_w(\phi,w)[\dw] &= \int^{0}_{-1} (\phi_s + H_s)w\dw \, ds, \qquad 
	\ds \Ham_{\phi}(\phi,w)[\dot{\phi}] = \int_{-1}^{0} (-w_r) \dot{\phi} \, ds
	\end{align*}
so that
\[ \partial_r \begin{pmatrix} \phi \\ w \end{pmatrix} = \begin{pmatrix} 0 & 1 \\ -1 & 0 \end{pmatrix} \nabla \mathcal{H}(\phi,w) .\]
%

Together,  $(\mathcal{M}, \symplecticform, \mathcal{H})$ constitutes an infinite-dimensional Hamiltonian system.  We claim that it corresponds precisely to the steady water wave problem \eqref{phiHeightEqn}.  Let $\mathcal{V}_{\Ham}$ be the Hamiltonian vector field  associated to $(\mathcal{M},\symplecticform,\Ham)$, and denote its domain by $\Dom{\mathcal{V}_{\Ham}}$.  By definition, $m \in \mathcal{M}$ is in $\Dom{\mathcal{V}_{\Ham}}$ provided that
	\begin{align*}
	\symplecticform_m(\mathcal{V}_{\Ham}|_m, \mathcal{W}|_m) = d\Ham |_m(\mathcal{W}|_m) \qquad \textrm{for all } \mathcal{W} \in T\mathcal{M} \subset T\calX.
	\end{align*}
A straightforward calculation shows that
	\begin{align*}
	d\Ham |_m(\mathcal{W}|_m) = \int_{-1}^0 -\dfrac{1}{2} &\left(w^2 + \dfrac{1}{(\phi_s + H_s)^2} -\dfrac{1}{H_s^2} \right)_s  \dot\phi \, ds \\
	&+ \int_{-1}^0 (\phi_s + H_s)w \dot w \, ds+ \dfrac{1}{2} \left(w^2(0) + \dfrac{1}{(\phi_s(0) + H_s(0))^2} + \Froude(\phi(0))\right)\dot\phi(0),
	\end{align*}
for $m = (\phi,w) \in \mathcal{M}$ and $\mathcal{W}|_m = (\dot\phi, \dot w) \in T\mathcal{M} |_m$, while
	\begin{align*}
\Dom{\mathcal{V}_{\Ham}} & =  \left\{ (\phi,w) \in \mathcal{M} : w(-1)=0,\,
	\left(\dfrac{1}{2}\left(w^2+\dfrac{1}{(\phi_s + H_s)^2} \right)  -\dfrac{1}{2H_s^2} + \Froude\phi \right)\bigg|_{s=0} = 0  \right \}.
	\end{align*}
It follows that Hamilton's equations are explicitly
\begin{subequations} \label{Hamiltonian formulation}
	\begin{equation} \label{Hamilton}
	\left \{
		\begin{array}{l l}
		\phi_r = w(\phi_s+H_s)\\
		w_r = \dfrac{1}{2}\left(w^2+\dfrac{1}{\phi_s + H_s} \right)_s - \left(\dfrac{1}{2H_{s}^{2}}\right)_s,
		\end{array}
	\right.
	\end{equation}
with the addtional conditions on the boundary
	\begin{align} \label{Hamilton boundary conditions}
	 \phi(r,-1)=w(r,-1)=0, \qquad
	\left(\dfrac{1}{2}\left(w^2 + \dfrac{1}{\phi_s + H_s} \right) -\dfrac{1}{2H_s^2} + \Froude \phi \right) \bigg|_{s=0} = 0.
	\end{align}
\end{subequations}
We recognize immediately that \eqref{Hamiltonian formulation} is indeed identical to \eqref{phiHeightEqn}.

In the above formulation, the steady water wave problem is interpreted as an ill-posed evolution equation where the horizontal variable $r$ operates as the time variable.  The solitary waves that we hope to construct then correspond to orbits of \eqref{Hamiltonian formulation} that are homoclinic to $0$.  Notice also that the invariance of the original equation with respect to reflection $r \mapsto -r$ translates to time-reversibility for the Hamiltonian system.  That is, if $(\phi,w)(r)$ is a solution to \eqref{Hamiltonian formulation}, then $\reverser(\phi,w)(-r)$ is also a solution, where the reverser $\reverser(\phi,w) := (\phi,-w)$.  

As we will do all of our analysis in a neighborhood of the critical Froude number, it is convenient at this point to reparameterize.  Specifically, for $\ep$ in a neighborhood of the origin in $\R$, we let 
	\begin{align} \label{F epsilon}
	\Fep := \left(\Fcr - \ep \right)^{-1/2}.
	\end{align}
It is important to notice that, with this convention, $\ep > 0$ implies that $F^{\epsilon}$ is supercritical.  We likewise define the $\ep$-dependent Hamiltonian by
	\begin{align*}
	\ds \Ham^{\ep}(\phi, w) = \int^{0}_{-1} \Bigg \{ \dfrac{1}{2} \left (w^2- \dfrac{1}{(\phi_s+H_s)^2} \right) - \dfrac{1}{2H_s^2} \Big \}(\phi_s +& H_s) \, ds + \dfrac{1}{2}\left( \Fep \right)\phi(0)^2,
	\end{align*}
so that Hamilton's equations \eqref{Hamilton} become
	\begin{align} \label{Ham_phi_w}
	\left \{
	\begin{array}{l l}
	\ds \phi_r = w(\phi_s + H_s)\\
		w_r = \dfrac{1}{2} \left( w^2+ \dfrac{1}{(\phi_s + H_s)^2} +\dfrac{1}{H^{2}_{s}} \right)_s,
	\end{array}
	\right.
	\end{align}
for which the domain of the Hamiltonian vector field is the set
	\begin{align*}
	 \Dom{\mathcal{V}_{\eHam}} := \left \{ (\phi,w) : w(-1)=0, \,\left(\dfrac{1}{2}\left( w^2 + \dfrac{1}{(\phi_s + H_s)^2} \right) -\dfrac{1}{2H_s^2} + \left(\Fep \right)\phi\right)\Bigg|_{s=0} = 0 \right \}.
	 \end{align*}
Linearizing the Hamiltonian system $(\mathcal{M}, \symplecticform, \mathcal{H}^0)$ about the equilibrium solution $(\phi,w)=(0,0)$ results in the problem
	\begin{align*}
	\partial_r \begin{pmatrix} \dot{\phi} \\ \dot{w} \end{pmatrix} = L \begin{pmatrix} \dot{\phi} \\\dot{w}\end{pmatrix},
	\end{align*}
where $L: \Dom{L} \subset \calX \to \calX$ is given by
	\begin{align} \label{linear operator}
	L\begin{pmatrix} \dot{\phi} \\ \dot{w}\end{pmatrix} :=
	\begin{pmatrix}
		H_s\dot{w}\\
		-\partial_s \left( H_s^{-3}  \dot \phi_s \right) 
	\end{pmatrix}
	\end{align}
with domain
	\begin{align*}
	\ds \Dom{L} := \left \{ (\dot{\phi}, \dot{w}) \in \calY : \dot{w}(-1) = 0, -\dfrac{\dot{\phi}(0)}{H_s^3(0)} + \frac{1}{(F^\epsilon)^2} \dot{\phi}(0) = 0 \right \}
	\end{align*}

The operator $L$ is closely related to the Sturm--Liouville problem introduced in Section \ref{SLP}.  Indeed, we have already characterized the spectrum of $L$, as the next lemma shows.  

\begin{lemma}[Spectral Properties]  \label{spectral properties} The linearized operator $L: \Dom{L} \subset \mathcal{X} \to \mathcal{X}$ satisfies the following:
\begin{enumerate}[font=\upshape,label=(\roman*)]
\item \label{Spec_L_i} The spectrum of $L$ consists of an eigenvalue at $0$ with algebraic multiplicity $2$, together with simple eigenvalues $\pm \sqrt{\nu_j}$, where $\nu_j$ are the nonzero eigenvalues of the corresponding Sturm--Liouville problem \eqref{SLP}.  The eigenvector and generalized eigenvector associated with the eigenvalue $\nu=0$ are
	\begin{align} \label{generalized eigenvectors}
	u_1 := \begin{pmatrix} \ds  \int^{s}_{-1} H^{3}_{s} \, dt \\  0 \end{pmatrix}, \qquad u_2:= \begin{pmatrix} 0\\  \ds \dfrac{1}{H_s} \int^{s}_{-1} H^{3}_{s} \, dt \end{pmatrix}.
	\end{align}

\item \label{Spec_L_ii} There exists $\Theta>0$ and $C>0$ such that
	\begin{align*}
	\|u\|_{\mathcal{Y}} \leq C\|(L-i\theta I)u\|_{\mathcal{X}}, \quad \text{and} \quad \|u\|_{\mathcal{X}} \leq \dfrac{C}{|\theta|} \|(L-i\theta I)u\|_{\mathcal{Y}}
	\end{align*}
for $u \in \Dom{L}$ and $\theta \in \R$ with $|\theta|>\Theta$.
\end{enumerate}
\end{lemma}

\pf The proof of (ii) is straightforward and follows from the exact same argument as in \cite[Appendix A.1]{chen2017existence}, so let us focus on (i).  Observe that $\lambda$ is an eigenvalue of $L$ provided that there exists a nontrivial solution  $\Phi$ to
	\begin{align*} \left \{
	\begin{array}{l l}
	-\left ( \dfrac{\Phi_s}{H^{3}_{s}} \right)_s = \lambda^2 \dfrac{\Phi}{H_s}, & \\
	\Phi(-1)=0
	\\
	\dfrac{1}{(F^\epsilon)^2} \Phi(0)-\dfrac{\Phi_s(0)}{H^{3}_{s}(0)}=0.
	\end{array}
	\right.
	\end{align*}
This is exactly the Sturm--Liouville equation \eqref{S-LProblem} with $\nu = \lambda^2$ and $\mu = 1/(F^\epsilon)^2$.  From Lemma \ref{S-L spectrum}, we know that $0$ is indeed an eigenvalue for $\epsilon = 0$, and the rest of the spectrum has the stated form.  Note that the algebraic multiplicity of $2$ for $0$ as an eigenvalue of $L$ results from the fact that $\nu = \lambda^2$ here.  Finally, an elementary calculation shows that the eigenvector $u_1$ and generalized eigenvector $u_2$ are given by the formulas in \eqref{generalized eigenvectors}.
\QED

\subsection{Further change of variables}
\label{COV}

Prior to applying the center manifold reduction theorem, we need to restructure \eqref{Ham_phi_w} to get rid of the nonlinearity in the boundary condition, which effectively flattens $\Dom{\mathcal{V}_{\eHam}}$.  We proceed by making the following change of variables:  let $\Lambda \times \mathcal{Z}$ be a neighborhood of the origin in $\mathbb{R} \times \calY$, and consider the mapping $\mathcal{G}: \Lambda \times \mathcal{Z}   \to H^1((-1,0))$ defined by
	\begin{align} \label{nearidentityF}
	\ds \mathcal{G} (\phi, w) := -\dfrac{1}{2}\left(w^2 + \dfrac{1}{(\phi_s + H_s)^2} \right) +  \dfrac{1}{2H^{2}_{s}} - \dfrac{1}{H_s^3}\phi_s
	\end{align}
so that $(\phi, w) \in \Dom{\mathcal{V}_{\mathcal{H}^0}}$ provided that $\mathcal{G}(\phi, w)|_{s=0} = 0$ and $w(-1) = 0$.   Here we have used the definition of $F_{\mathrm{cr}}$ in \eqref{def mucr}.  
Now let
	\begin{align*}
	\ds \zeta := \phi - H^{3}_{s}(0)(1+s) \int^{0}_{s} \mathcal{G}(\phi,w)(t) \, dt
	\end{align*}
and consider the mapping $\mathcal{G}_1: \mathcal{Z} \to H^2((-1,0)) \times H^1((-1,0))$ defined by $\mathcal{G}_1(\phi,w) = (\zeta,w)$.  We see that the nonlinear boundary condition in the definition of $\Dom{\mathcal{V}_{\eHam}}$ is equivalent to the following linear condition for $\zeta$:
	\begin{align} \label{G1boundaryconditions}
	\Fep \zeta(0) - \dfrac{1}{H_s^3(0)}\zeta_s(0) = 0.
	\end{align}
	
This is still parameter-dependent, though, so we make one more additional change of variables.  Consider the linear function $\mathcal{G}^{\ep}_{2} : \mathcal{Y} \to \mathcal{Y}$ defined by $\mathcal{G}^{\ep}_{2}(\zeta,w) := (\xi,w)$, where
	\begin{align*}
	\ds \xi &:= \zeta + \ep H^{3}_{s}(0)(1+s) \int^{0}_{s} \zeta(t) \, dt.
	\end{align*}			
It follows that the boundary conditions \eqref{G1boundaryconditions} become
	\begin{align*}
	\ds \Fcr \xi(0) - \dfrac{1}{H_s^3(0)}\xi_s(0) = 0.
	\end{align*}
Denoting $\mathcal{G}^{\ep}(\phi,w) := \mathcal{G}^{\ep}_2 \circ \mathcal{G}_1$, the following lemma verifies that $\mathcal{G}^{\ep}$ is a well-defined change of variables.

\begin{lemma} \label{near-identity G} For $\Lambda \times \mathcal{Z}$, a neighborhood of the origin in $\R \times \mathcal{Y}$, the following hold:
\begin{enumerate}[font=\upshape,label=(\roman*)]
	\item For each $\ep \in \Lambda$, $\mathcal{G}^{\ep}: \mathcal{Z} \to \mathcal{Y}$ is a diffeomorphism onto its image.  The mappings $\mathcal{G}^{\ep}$ and $(\mathcal{G}^{\ep})^{-1}$ depend smoothly on $\ep$.
	\item For each $(\ep, \phi, w) \in \Lambda \times \mathcal{Z}$, the derivative $D\mathcal{G}^{\ep}(\phi,w) : \calY \to \calY$ extends to an isomorphism $\widehat{D\mathcal{G}^{\ep}}(\phi,w) : \calX \to \calX$.  The operators $\widehat{D\mathcal{G}^{\ep}}$ and $(\widehat{D\mathcal{G}^{\ep}})^{-1}$ depend smoothly on $(\ep, \phi, w)$.
	\item The reverser is invariant under $\mathcal{G}^\epsilon$ in the sense that $\reverser = \mathcal{G}^{\epsilon} \circ \reverser \circ (\mathcal{G}^\epsilon)^{-1}$.
\end{enumerate}
\end{lemma}

\begin{proof} The proof follows exactly as in \cite[Lemma 4.7]{wheeler2013solitary}, which uses the facts
	\begin{align*}
	\mathcal{G}^{\ep}(0,0) = 0, \quad D\mathcal{G}^0(\phi,w) = \textrm{id}: \calY \to \calY,
	\end{align*}
and an application of the inverse function theorem. 
\QED

The ultimate result of the above myriad variable changes is that the Hamiltonian system $(\mathcal{M}, \mathcal{H}^\epsilon, \Upsilon)$ has been transformed to $(\mathcal{M}, \Upsilon_*, \mathcal{H}_*^\epsilon)$, where $\mathcal{H}_*^\epsilon := \mathcal{H}^\epsilon \circ (\mathcal{G}^\epsilon)^{-1}$, and likewise $\Upsilon_*^\epsilon$ is the pushforward of $\Upsilon$ under $\mathcal{G}^{\epsilon}$.   The corresponding Hamilton's equations now take the form
\be  u_r = L u + N^\epsilon(u), \label{final Hamiltonian} \ee
for $u = \mathcal{G}^\epsilon(\phi, w)$.  
The next lemma asserts the equivalence \eqref{final Hamiltonian}  with the original height equation and Eulerian formulations discussed earlier.  As it is entirely standard, we omit the proof (see, for example, \cite[Lemma 4.3, Lemma 4.4]{wheeler2013solitary}.)  

\begin{lemma}\label{change of vars lem}\hfill
  \begin{enumerate}[label=\rm(\roman*)]
  \item \label{elliptic to dynamic} Let $(\phi,F^\epsilon) \in X \times \mathbb{R}$ be a solution of the height equation \eqref{phiHeightEqn} with $\n \phi_X$ and $\abs \epsilon$ sufficiently small.  Then  
\[
          u := \mathcal{G}^\epsilon \left(\phi, \frac{\phi_r}{\phi_s + H_s}\right) \in C^4(\mathbb{R} ; \mathcal{X}) \cap C_0^1(\mathbb{R} ; \mathcal{X}) \cap C^3(\mathbb{R}; \mathcal{Z}) \cap C_0^0(\mathbb{R}; \mathcal{Z})
\]
   solves Hamilton's equations \eqref{final Hamiltonian} and is reversible in the sense that $u(-r)=\reverser u(r)$.
  \item \label{dynamic to elliptic} Conversely, suppose that $u \in C^4_0(\R,\X) \cap C^3_0(\R;\mathcal{Z})$ satisfies \eqref{final Hamiltonian} and $u(-r)=\reverser u(r)$. Then setting $(\phi,w) := (\mathcal{G}^\epsilon)^{-1}(u)$, it follows that $\phi \in X$ and solves $\F(\phi,F^\epsilon) = 0$. Moreover, the mapping $(\epsilon,u) \mapsto \phi$ is continuous $\R \by C^4_0(\R;\X) \cap C^3_0(\R;\mathcal{Z}) \to X$.
  \item \label{linear elliptic to dynamic} With $u$ and $\phi$ as above, suppose that $\dot \phi \in X$ is a nontrivial solution of the linearized problem $\F_\phi(\phi,F^\epsilon)\dot \phi = 0$. Then there is an associated nontrivial solution $\dot u \in C^2_\bdd(\R,\X) \cap C^1_\bdd(\R,\Dom{\linear})$ of the linearized Hamiltonian $\dot u_r = L \dot u + DN^\epsilon(u) \dot u$ satisfying $\dot u(-r) = \reverser \dot u(r)$.
  \end{enumerate}
\end{lemma}
Note that although the above lemma supposes that $u$ is quite smooth with respect to the time-like variable, this is perfectly justified in light of Lemma~\ref{addt'l regularity at boundary}.  Also, the small-amplitude solutions that we construct below are exponentially localized, so $\partial_r^k u$ is decays at infinity for any $k \geq 0$.  

\subsection{Center manifold reduction}
\label{CMR}

We have now laid the groundwork for the center manifold reduction.  Let $\Xc \subset \calX$ be the two-dimensional (generalized) eigenspace for $L$ associated with the eigenvalue $0$.  Also let $\Pc$ be the associated spectral projection, and write $\Ph := I-\Pc, \Xh := \Ph\calX$.  We denote $u_{\mathrm{c}} \in \Pc\Dom{L}$; more precisely, $u_{\mathrm{c}} = z_1 u_1 +z_2u_2$, where $u_1, u_2 \in \Dom{L}$ are the eigenvector and generalized eigenvector from Lemma~\ref{spectral properties}\ref{Spec_L_i}.

\begin{lemma}[Center manifold reduction]  \label{CMRthm} For any integer $k \geq 2$, there exists a neighborhood $\Lambda \times \mathcal{U}$ of the origin in $\R \times \Dom{L}$ such that, for each $\ep \in \Lambda$, there exists a two-dimensional manifold $\mathcal{W}^{\ep} \subset \mathcal{U}$ together with an invertible coordinate map
	\begin{align*}
	\chi^{\ep} := \Pc|_{\mathcal{W}^{\ep}} : \mathcal{W}^{\ep} \to \mathcal{U}^c := \Pc\mathcal{U},
	\end{align*}
with the following properties:
	\begin{enumerate}[font=\upshape,label=(\roman*)]
		\item \label{CMRi} Defining $\Psi^{\ep} : \mathcal{U}^{\mathrm{c}} \to \mathcal{U}^* := \Ph\mathcal{U}$ by
			\begin{align*}
			u_{\mathrm{c}} + \Psi^{\ep}(u_\mathrm{c}) := (\chi^{\ep})^{-1}(u_{\mathrm{c}}),
			\end{align*}
		the map $(\ep, u) \mapsto \Psi^{\ep}(u)$ is $C^k(\Lambda \times \mathcal{U}^c, \mathcal{U}^*)$.  Moreover, $\Psi^{\ep}(0) = 0$ for all $\ep \in \Lambda$ and $D\Psi^0(0) = 0$.
		\item \label{CMRii} Every initial condition $u_0 \in \mathcal{W}^{\ep}$ determines a unique solution $u$ of $u_r = Lu + N^{\ep}(u)$, which remains in $\mathcal{W}$ as long as it remains in $\mathcal{U}$.
		\item \label{CMRiii} If $u$ solves $u_r = Lu + N^{\ep}(u)$ and lies in $\mathcal{U}$ for all $r$, then $u$ lies entirely in $\mathcal{W}^{\ep}$.
		\item \label{CMRiv} If $u_{\mathrm{c}} \in C^1((a,b), \mathcal{U}^c)$ solves the reduced system
			\begin{align} \label{reduced system}
			(u_{\mathrm{c}})_r = f^{\ep}(u_{\mathrm{c}}) := Lu_{\mathrm{c}} + \Pc N^{\ep}(u_{\mathrm{c}} + \Psi^{\ep}(u_{\mathrm{c}})),
			\end{align}
		then $u= (\chi^{\ep})^{-1}(u_{\mathrm{c}})$ solves the full system $u_r=Lu+N^{\ep}(u)$.
		\item \label{CMRv} With $u_{\mathrm{c}}$ and $u$ as above, if $\dot{u}_{\mathrm{c}} \in C^1(\R,\mathcal{U}^{\mathrm{c}})$ solves the linearized reduced equation $(\dot{u}_{\mathrm{c}})_r = Df^{\ep}(u_{\mathrm{c}})\dot{u}_{\mathrm{c}}$, then $\dot{u} = \dot{u}_{\mathrm{c}} + D_u\Psi^{\ep}(u_{\mathrm{c}})\dot{u}_{\mathrm{c}}$ solves the full linearized system $\dot{u}_r = L\dot{u} + D_uN^{\ep}(u)\dot{u}$.
		\item \label{CMRvi} The reduced system \eqref{reduced system} can be transformed into a Hamiltonian system $(V^{\mathrm{c}},\upsilon,\eK)$ via a $C^{k-1}$ change of variables, where $V^{\mathrm{c}}$ is a neighborhood of the origin in $\R^2$, $\upsilon$ is the canonical symplectic form
			\begin{align*}
			\upsilon\left((z_1,z_2),(z_1', z_2')\right) := z_1z_2'-z_1'z_2, \quad (z_1, z_2), (z_1',z_2') \in \R^2,
			\end{align*}
		and the reduced Hamiltonian is given by
			\begin{align} \label{reduced Ham}
			\eK(z_1, z_2) := \eHam(z_1u_1 + z_2u_2 + \Theta^{\ep}(z_1u_1+z_2u+2)),
			\end{align}
		where $(\ep, u_{\mathrm{c}}) \mapsto \Theta^{\ep}(u_{\mathrm{c}})$ is of class $C^{k-1}(\Lambda \times \mathcal{U}^{\mathrm{c}},\mathcal{U})$ and satisfies $\Theta^{\ep}(0)=0$ for all $\ep \in \Lambda$, and $D_{u_{\mathrm{c}}}\Theta^0(0)=0$.  The system is reversible with reverser $S(z_1,z_2) = (z_1,-z_2)$.
	\end{enumerate}
\end{lemma}

\begin{proof} We already proved in Lemma \ref{spectral properties} that $L$ satisfies the spectral hypotheses \ref{H1} and \ref{H2} of Theorem \ref{BGTCMR}.  In particular, the only part of the spectrum on the imaginary axis is $0$, which we showed as algebraic multiplicity $2$.  Additionally, Lemma \ref{near-identity G} and a quick calculation verify that $N^0(0)=0$ and $D_uN^0(0)=0$, which satisfies \ref{H3} in Theorem \ref{BGTCMR}.  Applying that theorem proves statements \ref{CMRi}--\ref{CMRiv} above.  Lastly, part \ref{CMRvi} follows by undoing the near-identity transformation $\mathcal{G}^{\ep}$ in favor of working with the original variables, and then employing a parameter-dependent Darboux transformation to obtain the reversible Hamiltonian system $(V^{\mathrm{c}}, \upsilon,\mathcal{K}^{\ep})$ (see, for example, \cite[Theorem 4]{buffoni1999multiplicity}).
\QED

From the above Lemma, we obtain the following reduced Hamiltonian system governing the dynamics on the center manifold:
\[ \partial_r \begin{pmatrix} z_1 \\ z_2 \end{pmatrix} = \begin{pmatrix} 0 & 1 \\ -1 & 0 \end{pmatrix} \nabla {\mathcal{K}}^\epsilon(z_1, z_2) + \mathcal{R}(z_1, z_2, \epsilon)  \]
where $\mathcal{R}$ is a remainder term.   We can simplify  further by introducing the scaled variables $Z = (Z_1,Z_2)$ and $R$ defined as
\begin{equation} \label{r to R change of variables}
	\begin{split} 
	\ds Z_1 &:= |\ep|^{-1} \phi_{\mathrm{cr}}(0)^{-2} \left(\int^{0}_{-1} H^{5}_{s} \, ds \right) z_1 \\
	Z_2 &:= |\ep|^{-3/2} \phi_{\mathrm{cr}}(0)^{-2} \left(\int^{0}_{-1} \dfrac{\phi^{2}_{\mathrm{cr}}}{H_s}\, ds\right)^{1/2}\left(\int_{-1}^{0} H^{5}_{s}\, ds \right) z_2\\
	R &:= |\ep|^{1/2} \left( \phi_{\mathrm{cr}}(0)^2\int^{0}_{-1} \dfrac{\phi^{2}_{\mathrm{cr}}}{H_s} \, ds \right)^{1/2} r.
	\end{split}
	\end{equation}
Here $\phi_{\mathrm{cr}}$ is the eigenfunction corresponding to the eigenvalue $0$ of the Sturm--Liouville problem \eqref{eigenvalueProblem}.   Ultimately, this yields 
	\begin{equation} \label{Z remainder system}
	\partial_R \begin{pmatrix} Z_1 \\ Z_2 \end{pmatrix} = \begin{pmatrix} Z_2 \\ -(\sgn{\ep}) Z_1 - \dfrac{3}{2} Z_1^2 \end{pmatrix} + \hat{\mathcal{R}}(Z_1, Z_2, \epsilon) \end{equation}
with new remainder term $\hat{\mathcal{R}} = \bigO(\ep^{1/2})$.   The calculation leading to \eqref{Z remainder system} follows exactly as in \cite{groves2008vorticity}.  Further details are also provided in \cite[Appendix A.2]{chen2017existence}.

\begin{lemma}[Existence of $\phi^{\ep}$] \label{existence thm}  There exists $\ep_{*} > 0$ such that for each $\ep \in (0,\ep^{*})$, there is a corresponding solution $(\phi^{\ep},F^{\ep})$ to the height equation \eqref{phiHeightEqn}.
\end{lemma}

\begin{proof}  Passing to the limit as $\ep \searrow 0$, the system \eqref{Z remainder system} becomes
	\begin{align}
	(Z_1)_{RR} = Z_1 - \dfrac{3}{2}Z_1^2
	\end{align}
which is exactly the equation satisfied by the KdV soliton $Z_1^0(R) = \sech^{2}{(R/2)}$.  In the planar system \eqref{Z remainder system}, this corresponds to a orbit homoclinic to the origin.   Due to \cite[Proposition 5.1]{kirchgassner1988resonant}, we may exploit reversibility to conclude that the  phase portrait of \eqref{Z remainder system} is qualitatively the same for $\ep>0$ sufficiently small.  More precisely, there exists a reversible homoclinic orbit $(Z_1^{\ep},Z_2^{\ep})$ for $0 < \ep < \ep_*$, with $Z_1^{\ep} > 0$.  Since $(Z_1^{\ep},Z_2^{\ep})(0)$ depends continuously on $\ep$, we have uniform bounds
	\begin{align} \label{exponential localization in R}
	 |\partial_R^k(Z_1^{\ep},Z_2^{\ep})| \leq Ce^{-|R|/2}, \qquad \textrm{for } k = 0, 1, 2.
	\end{align}
Using the change of variables defined in \eqref{r to R change of variables}, we can write $(z_1^{\ep},z_2^{\ep})$ in terms of $(Z_1^{\ep}, Z_2^{\ep})$ to obtain a reversible homoclinic orbit $u_{\mathrm{c}}^{\ep} := z_1^{\ep}u_1 + z_2^{\ep}u_2$ of the reduced system \eqref{Z remainder system}.  Notice that
	\begin{align*}
	\ds  \| u_{\mathrm{c}}^{\ep}\|_{C^2(\mathbb{R}; \mathcal{Y})} \leq \| z_1^\ep \|_{C^2(\mathbb{R})} \| u_1 \|_{H^2} + \| z_2^\ep \|_{C^2(\mathbb{R})} \| u_2 \|_{H^1},	\end{align*}
so that applying the change of variables \eqref{r to R change of variables} and the uniform bound \eqref{exponential localization in R} yields
	\begin{align} \label{exponential localization in r}
	 |\partial_r^k(u_{\mathrm{c}}^{\ep})| \leq C_1\ep e^{-C_2 |\ep|^{1/2} |r|}, \qquad \textrm{for } k = 0, 1, 2, 
	\end{align}
with positive constants $C_1, C_2$ uniform in $\epsilon$ for $\epsilon$ small.

Using the pullback $(\chi^{\ep})^{-1}$, define 
	\begin{align} \label{uc defn}
	\ds u^{\ep} := (\chi^{\ep})^{-1}(u_{\mathrm{c}}^{\ep}) = u_{\mathrm{c}}^{\ep} + \Psi^{\ep}(u_{\mathrm{c}}^{\ep}). 
	\end{align}
	According to Lemma~\ref{CMRthm}\ref{CMRiv}, $u^{\ep}$ corresponds to an orbit of the full system \eqref{final Hamiltonian} that is homoclinic to $0$ and remains in the neighborhood of the origin $\mathcal{U} \subset \Dom{L}$.  Note also that, because $\Psi^\epsilon$ is of class $C^2$ taking values in a subspace of $H^2 \times H^1$, we have by the chain rule that $u^\epsilon$ is likewise exponentially localized in $r$.  From this observation and Lemma~\ref{change of vars lem}, it is easily seen that $\phi$ defined to be the first component of $(\mathcal{G}^{\epsilon})^{-1}(u^\epsilon)$ is a solitary wave solution of the height equation with the regularity $W^{2,p}(R)$.  In particular, this means that $\phi \in X$.
\QED

\subsection{Proof of small-amplitude existence}
\label{proof of small-amp}

We finally have all of the necessary components to prove the main result of this section.

\begin{proof}[Proof of Theorem \ref{small-amp}]  We already constructed the family of small-amplitude solitary wave solutions $(\phi^{\ep},F^{\ep})$ in Lemma \ref{existence thm}.  Recalling that $\ep>0$ implies $F^{\ep} > F_{\textrm{cr}}$, the invertibility of $\mathcal{F}_{\phi}(\phi^{\ep},F^{\ep})$ in part \ref{existence_invertibility} follows immediately from the equivalence of formulations and Lemma \ref{strong invertibility lemma}.

It remains to show parts \ref{existence_continuity} and \ref{existence_uniqueness}.  Recalling $\ep_{*}$ and $u_{\mathrm{c}}^{\ep}$ as in Lemma \ref{existence thm}, we see that, after possibly shrinking $\ep_{*}$, the exponential estimates also hold for $\phi^{\ep}$, and part \ref{existence_continuity} follows.

Now suppose we have a solution $(\phi, F^{\ep})$ of the height equation \eqref{phiHeightEqn} with $\ep + \| \phi \|_X < \delta$, where $\delta$ is to be determined.  By way of contradiction, assume $\phi \not\equiv \phi^{\ep}$.  We wish to show that $\phi$ is not a supercritical solution.

By the properties of the center manifold reduction theorem, we know that $\phi$ is determined by a homoclinic orbit $(z_1,z_2)$ of \eqref{reduced system}.  Since we know that this equation already has a homoclinic orbit $(z_1^{\ep}, z_2^{\ep})$, and that $\phi$ is not a translation of $\phi^{\ep}$, it is impossible for $(z_1,z_2)$ to be a translate of $(z_1^{\ep},z_2^{\ep})$.  By the properties of the phase portrait at the origin, we conclude that $z_1 <0$ for $|r|$ sufficiently large and
	\begin{align} \label{z limit}
	\ds \lim_{|r| \to \infty} \dfrac{z_2(r)}{z_1(r)} = \pm \ep^{1/2} + \mathcal{O}(\ep).
	\end{align}
Tracing back the various changes of variable, we see that
	\begin{align}
	\phi(r,s) = z_1(r)u_1(s) + \mathscr{R}(r,s),
	\end{align}
where the remainder term satisfies
	\begin{align} \label{remainder bounds}
	\| \mathscr{R}(r, \cdot) \|_{H^2(-1,0)} \leq C(|\ep| + |z_1| + |z_2|)(|z_1| + |z_2|),
	\end{align}
with constant $C$ independent of $\ep$.  Now taking $\delta$ small enough, we have
	\begin{align}
	\| \mathscr{R}(r,0) \|_{X} \leq \dfrac{u_1(0)}{2}(|z_1(r)| + |z_2(r)|).
	\end{align}
Shrinking $\delta$ further, \eqref{z limit} and \eqref{remainder bounds} yield
	\begin{align}
	\| \mathscr{R}(r,0) \|_X < u_1(0)|z_1(r)|,
	\end{align}
for $|r|$ sufficiently large.  Finally, for $|r|$ large enough that $r>0$ and $\mathscr{R}(r,0)$ has the above upper bound, we know that
	\begin{align} \label{elevation contradiction}
	\phi(r,0) = z_1(r)u_1(0) + \mathscr{R}(r,0) < 0.
	\end{align}
Since in Lemma~\ref{supercritical elevation} we proved that supercritical solutions are equivalent to waves of elevation, \eqref{elevation contradiction} contradicts the fact that $\phi$ is a supercritical solution.    Hence $\phi = \phi^{\ep}$, and the theorem is proved.
\QED

\section{Large-amplitude waves} \label{proof main result section}

Finally, in this section we prove the main result Theorem \ref{main theorem} by showing that the curve of small-amplitude solitary waves $\mathscr{C}_{\mathrm{loc}}$ lies inside a much larger family containing large-amplitude waves that are arbitrarily close to having points of horizontal stagnation.  We do this using analytic global bifurcation theory, which was first introduced by Dancer \cite{dancer1973bifurcation,dancer1973globalsolution,dancer1973globalstructure} and further refined by Buffoni and Toland \cite{buffoni2003analytic}.  Chen, Walsh, and Wheeler \cite{chen2016announcement,chen2017existence} recently developed a variant of these results that is particularly well-suited to the study of solitary waves.  Compared to the classical theory, it offers two main advantages: (i) the initial point of bifurcation may be at a point on the boundary of the set where $\mathscr{F}$ is Fredholm, and (ii) there is no a priori requirement that the solution set $\mathscr{F}^{-1}(0)$ is locally compact.  To accommodate the latter of these relaxed hypotheses, the Chen--Walsh--Wheeler leave open the additional ``undesirable alternative'' that the extended curve loses compactness;  see Appendix~\ref{global bifurcation appendix} and specifically Theorem~\ref{generic global theorem}\ref{gen alternatives}\ref{gen noncompact}.   Usually, one hopes to later exclude this scenario using finer qualitative properties of the solutions.  

In the next subsection, we preemptively show that the undesirable alternative cannot happen.  We may then apply the abstract theory to conclude that there exists a global curve $\mathscr{C}$ extending $\mathscr{C}_{\mathrm{loc}}$, and, as we follow it, a certain quantity must tend to infinity.  The last step is to verify that this blowup leads inexorably to the stagnation limit claimed in Theorem~\ref{main theorem}\ref{extreme wave limit}.   

\subsection{Local compactness} \label{local compactness section}

The flow force corresponding to the solitary wave $(\phi,F)$ is defined to be 
\be \flowforce (h,F) := \int_{-1}^0 \left[ \frac{1-h_r^2}{2 h_s^2} + \frac{1}{2H_s^2} - \frac{1}{F^2}(h-1)  \right] h_s \, ds.\label{definition flow force} \ee
Here, as usual, $h := \phi + H$.  A simple computation confirms that $\flowforce(\phi,F)$ is independent of $r$ for any $(\phi,F) \in X \times \mathbb{R}_+$ satisfying the height equation.  It is well known that the flow force is related to the Hamiltonian for the spatial dynamic formulation of the steady water wave problem.  Indeed, we exploited this connection ourselves when constructing the curve of small-amplitude solutions.  Our interest in $\flowforce$ here is somewhat different:  we will use the fact that it distinguishes the limiting height $H$ from all other laminar flow solutions of \eqref{heightFormulation}.  Following the ideas of Wheeler \cite{wheeler2013solitary}, this will allow us to rule out the loss of compactness alternative in Theorem \ref{generic global theorem}.  The key ingredient is the next lemma, which closely parallels \cite[Lemma 2.10]{wheeler2013solitary}. 

\begin{lemma} \label{flow force lemma} Suppose that $(\phi, F) \in C_\bdd^{1,\alpha}(\overline{R}) \times \mathbb{R}$ is a supercritical solution of the height equation \eqref{phiHeightEqn} but does not necessarily satisfy the asymptotic condition $\phi \in C_0^1(\overline{R})$. Denote $h := \phi + H$ and assume that $h_r \equiv 0$ and $\inf_R h_s > 0$.  Then $\flowforce (h,F) \geq \flowforce (H,F)$ with equality holding if and only if $h \equiv H$.
\end{lemma}
\begin{proof}
Recall from Section~\ref{elevation} that there is a one-parameter family $\mathscr{T}$ of laminar flows $H(\placeholder ; \kappa)$ having the vorticity function $\gamma$.  Indeed, we see immediately from their explicit formula \eqref{definition of H(s,kappa)} that $H(\placeholder; \kappa) \in C^{1,1}([-1,0])$ and depend  smoothly on $\kappa$.  Recall that this family is exhaustive in the sense that it includes every $r$-independent solution of the height equation of class $C^{1,\alpha}(\overline{R})$ with vorticity function $\gamma$.  In particular, $H = H(\placeholder; \lambda)$ for the unique parameter value $\lambda := 1/H_s(0)^2$.    

Evaluating the flow force of an element of $\mathscr{T}$, we find that 
\begin{align*} \flowforce (H(\placeholder; \kappa), F) &= \int_{-1}^0 \left[ \frac{1}{2H_s(s; \kappa)^2} + \frac{1}{2H_s(s)^2} - \frac{1}{F^2} \left( H(s; \kappa) - 1\right) \right] H_s(s; \kappa) \, ds \\
& = \left[ \frac{1}{2H_s(0; \kappa)^2} + \frac{1}{2H_s(0)^2} \right] H(0; \kappa) - \frac{1}{2F^2} \left( H(0; \kappa) - 1\right)^2 + \frac{1}{2 F^2} \\
& \qquad - \int_{-1}^0 \left[ \frac{1}{2H_s(s; \kappa)^2} + \frac{1}{2H_s(s)^2} \right]_s H(s; \kappa) \, ds   \\
& = \frac{\lambda - \kappa}{2} H(0; \kappa) - \frac{1}{2F^2} \left( H(0; \kappa) - 1 \right)^2 + \frac{1}{2F^2} +  \int_{-1}^0 \sqrt{\kappa + 2 \Gamma(s)} \, ds. \end{align*}
In deriving the third line from the second, we have made repeated use of the fact that 
\be \frac{\kappa-\lambda}{2} = \frac{1}{2H_s(0; \kappa)^2} - \frac{1}{2H_s(0)^2} = -\frac{1}{F^2} \left( H(0; \kappa) - 1 \right),\label{kappa lambda identity} \ee
which follows from the definition of $\Gamma$ and the boundary condition on the top in \eqref{heightFormulation}. Differentiating this identity with respect to $\kappa$ then gives  
\begin{align*}
\partial_\kappa \flowforce (H(\placeholder; \kappa), F) & = -\frac{1}{2} H(0; \kappa) + \frac{\lambda - \kappa}{2}  H_\kappa(0; \kappa) - \frac{1}{F^2}\left( H(0;\kappa) -1 \right) H_\kappa(0; \kappa) \\ 
& \qquad + \frac{1}{2} \int_{-1}^0 \frac{1}{\sqrt{\kappa+ 2\Gamma(s)}} \, ds \\
& =  \frac{1}{2 F^2} H_\kappa(0; \kappa) = -\frac{1}{2} \int_{-1}^0 \left(\kappa + 2\Gamma(s)\right)^{-3/2} \, ds < 0. 
\end{align*}
Thus the flow force is monotonically decreasing along the family $H(\cdot; \kappa)$ as $\kappa$ increases.  

Now let $(\phi,F)$ be given as in the statement of the lemma.  There exists a unique $\kappa_* \in (-2\Gamma_{\mathrm{min}}, \infty)$ such that $h = H(\placeholder; \kappa_*)$.  We then see immediately from the above considerations that $\flowforce (h,F) = \flowforce (H, F)$ if and only if $h \equiv H$.  Suppose instead that $h \not\equiv H$, meaning that $\kappa_* \neq \lambda$ and $H(0; \kappa_*) \neq 1$.    From \eqref{kappa lambda identity} it follows that
\[ \mu(\kappa_*) = \frac{1}{2} \frac{\kappa_* - \lambda}{1- H(0; \kappa_*)} = \frac{1}{F^2} < \frac{1}{F_{\mathrm{cr}}^2} = \mu(\lambda),\]
where $\mu(\cdot)$ is the function introduced in Lemma \ref{A kappa}.   In that same lemma, we showed that $\mu(\cdot)$ is strictly increasing, implying that $\kappa_* < \lambda$.  Since $\kappa \mapsto \flowforce (H(\placeholder; \kappa), F)$ is strictly decreasing, this at last gives us the inequality $\flowforce (H(\placeholder; \kappa_*), F) > \flowforce (H(\placeholder; \lambda), F) = \flowforce (H, F)$.
\end{proof}

\begin{proposition}[Local compactness] \label{local compactness proposition} Any sequence $\{(\phi_n, F_n)\} \subset \Udomain$ of solutions to the height equation \eqref{phiHeightEqn} satisfying 
\[ \sup_n \left(  \| \phi_n \|_{C^{1,\alpha}(R)} + \| \phi_n \|_{W_{\mathrm{loc}}^{2,p}}+ F_n + \frac{1}{F_n - F_{\mathrm{cr}}} + \frac{1}{\inf_R{\left(\partial_s \phi_n + H_s\right)}}  \right) < \infty\]
has a subsequence converging in $X \times \mathbb{R}$ to some solution $(\phi,F) \in \Udomain$ of the height equation.  
\end{proposition}
\begin{proof}  Let $\{(\phi_n, F_n)\}$ be given as above.  As they are supercritical, the qualitative theory developed in Section \ref{quals} ensures that each of these waves is necessarily even, monotonic, and a wave of elevation.   We first claim that the sequence is equidecaying in the sense that function
  \[ M(r) := \sup_n \sup_{s \in [-1,0]}  |\phi_n(r,s)| \qquad \textrm{for all }  r \in \mathbb{R},\]
  vanishes in the limit $r \to \pm\infty$.  Seeking a contradiction, assume that this not the case. Then there exists $\varepsilon > 0$ and a sequence $\{(r_n,s_n)\} \subset \overline{R}$ with $r_n \to \infty$ and $\phi_n(r_n, s_n) > \varepsilon$.  Passing to a subsequence, we can assume that $s_n \to s_* \in [-1,0]$ and $F_n \to F \in (F_{\mathrm{cr}}, \infty)$.  
 It is slightly more convenient at this stage to shift to the corresponding sequence of height functions $h_n := \phi_n + H$; recall that they solve \eqref{heightFormulation}.  Let $\tilde h_n := h_n(\cdot + r_n, s)$.  As the height equation is invariant under translation in the $r$-direction, $\tilde h_n$ is also a solution.   Clearly, $\{ \tilde h_n \}$ is uniformly bounded in $C_\bdd^{1,\alpha}(\overline{R})$, and thus passing to a subsequence, we can conclude that $\tilde h_n \to \tilde h$ in $C_{\mathrm{loc}}^1(\overline{R})$ for some $\tilde h \in C_\bdd^{1,\alpha}(\overline{R})$.  
 
 Consider now the asymptotic behavior and flow force associated to $\tilde h_n$ and $\tilde h$.  Because $\tilde h_n(r, \cdot) \to H$ as $|r| \to \infty$, we must have $\flowforce (\tilde h_n, F) = \flowforce (H, F)$.  Likewise, from its definition \eqref{definition flow force}, it is clear that the flow force is preserved by the $C_{\mathrm{loc}}^1(\overline{R})$ limit, and hence $\flowforce (\tilde h, F) = \flowforce (H, F)$.   On the other hand, because $\partial_r \tilde h_n(r,s) < 0$ for $r > 0$ and $s \in (-1,0)$, we know that $\partial_r \tilde h \leq 0$ in $R$.  In particular, this implies the existence of point-wise limits
\[ \tilde h(r,s) \to H_\pm(s) \qquad \textrm{as } r \to \pm\infty, ~\textrm{for all } s \in [-1,0].  \]
Naturally, the monotonicity of $\tilde h$ also gives $H_+(s) \leq \tilde h(r,s) \leq H_-(s)$ for all $(r,s) \in \overline{R}$.  

We claim that $(H_\pm,F)$ are both $r$-independent solutions of the height equation.  Consider a further translated sequence defined by $\hat h_{\pm n} := \tilde h(\cdot \pm n, \cdot)$.  The translation invariance of the height equation ensures that each $\hat h_{\pm n}$ is also a solution having the same upstream and downstream states.  Because $\{ \hat h_{\pm n} \}$ is uniformly bounded in $C_\bdd^{1,\alpha}(\overline{R})$, passing to a subsequence we have $\hat h_{\pm n} \to \hat h_\pm$ in $C_{\mathrm{loc}}^1(\overline{R})$ for some $\hat h_\pm \in C_\bdd^{1,\alpha}(\overline{R})$ that is itself a solution of the height equation and satisfies $\flowforce (\hat h_\pm, F) = \flowforce (H, F)$.  By definition, though, $\hat h_{\pm n}$ converges point-wise to $H_\pm$, and hence we have shown that $H_\pm \in C^{1,\alpha}$ is a $r$-independent solution of the height equation having flow force $\flowforce (H, F)$.  Lemma \ref{flow force lemma} then forces $H_+ = H_- = H$, which in turn means $\tilde h \equiv H$.  This is a contradiction as, by construction, $\tilde h(0,s_*) - H(s_*) \geq \varepsilon > 0$.  
  
The above argument confirms that $M(r) \to 0$ as $r \to \pm\infty$.  To complete the proof, we will show that this implies the existence of a convergent subsequence first in $C_\bdd^{1,\alpha}(\overline{R})$ and second in $W_{\mathrm{loc}}^{2,p}(R)$.  Equidecay and uniform boundedness of $\{ \phi_n \}$ allows us to infer that, upon passing to a subsequence, $\phi_n \to \phi$ in $C_{\mathrm{loc}}^1(\overline{R})$ and $C_\bdd^0(\overline{R})$, for some $\phi \in C_\bdd^{1,\alpha}(\overline{R}) \cap C_0^1(\overline{R})$.  Observe that the height equation for $\phi$ has the general form 

\begin{equation}
 \label{lem:compact:main} \left\{ 
 \begin{alignedat}{2}
   \mathcal \partial_{i} \mathscr{A}^i(s,\nabla \phi) &=0 && \text{ in } R,\\
   -\mathscr A^2(0,\nabla \phi) + \mathscr{G}(\phi, F)&=0 && \text{ on } T ,\\
   \phi &=0 && \text{ on } B,
 \end{alignedat} \right. 
\end{equation}
  where $\mathscr{A} = \mathscr{A}(s,\xi) : \mathbb{R} \times \mathbb{R}^n \to \mathbb{R}^2$ and $\mathscr{G} = \mathscr{G}(z, F) : \mathbb{R} \times (F_{\mathrm{cr}}, 0) \to \mathbb{R}$.  It follows that $v_n := \phi_n - \phi$ is a distributional solution of  
  \be
       \left\{     \begin{alignedat}{2} 
        \partial_i \left( a^{ij}_n \partial_j v_n \right)  &= 0 &\qquad& \text{in } R,\\
        -a_n^{2j} \partial_j v_n + \gamma_n v_n &= 0 && \text{on }  T,\\
        v_n &= 0 && \text{on } B,
      \end{alignedat} \right. 
\label{compactness v_n equation}
\ee
    where we are using the summation convention.  The coefficients $a^{ij}_n, \gamma_n$ are given by
 \[ a^{ij}_n :=  \int_0^1 \mathscr A^j_{\xi^i}(s, \nabla \phi_n^{(\lambda)})\, d\lambda, \qquad 
      \gamma_n := \int_0^1 \mathscr G_z(\phi_n^{(\lambda)}, F^{(\lambda)})\, d\lambda, \]
for $\phi_n^{(\lambda)} := \lambda \phi_n + (1-\lambda)\phi$ and $F^{(\lambda)} := \lambda F_n + (1-\lambda)F$.  The uniform boundedness of $\{\phi_n\}$ in $C_\bdd^{1,\alpha}(\overline{R})$ thus ensures that $a_n^{ij}$  and $\gamma_n$ are uniformly bounded in the $C^{0,\alpha}$ norm.  Likewise, we see that the above linear problem is uniformly elliptic (with an ellipticity constant independent of $n$) and the boundary condition is uniformly oblique.  We are therefore justified in applying the Schauder-type estimate Theorem \ref{schauder theorem} to infer that  
\[ \| v_n \|_{C^{1,\alpha}(R)} \leq C \| v_n \|_{C^0(R)}, \]
for some constant $C > 0$ independent of $n$.  Recall that $v_n = \phi_n - \phi$, and $\phi_n \to \phi$ in $C_\bdd^0(\overline{R})$.  The above bound then proves $\phi_n \to \phi$ in $C_\bdd^{1,\alpha}(\overline{R})$, as desired.  

Finally, we must show convergence in $W_{\mathrm{loc}}^{2,p}(R)$.  Of course $C_\bdd^{1,\alpha}(\overline{R})$ convergence already guarantees that $\phi_n \to \phi$ in $W^{1,\infty}(R)$ and thus $W_{\mathrm{loc}}^{1,p}(R)$ .  For each $m$,$n \geq 1$, let $v_{mn} := \phi_n - \phi_m \in W_{\mathrm{loc}}^{2,p}$.  As before, we find that this difference solves an elliptic problem
  \be
       \left\{     \begin{alignedat}{2} 
        \partial_i \left( a^{ij}_{mn} \partial_j v_{mn} \right)  &= 0 &\qquad& \text{in } R,\\
        -a_{mn}^{2j} \partial_j v_{mn} + \gamma_{mn} v_{mn} &= 0 && \text{on }  T,\\
        v_{mn} &= 0 && \text{on } B,
      \end{alignedat} \right. 
\label{compactness v_n equation}
\ee
  In this case, the coefficients $a^{ij}_{mn}, \gamma_{mn}$ are bounded uniformly in $C_\bdd^{0,\alpha}(\overline{R})$.  Applying standard elliptic theory, we deduce that
  \[ \| v_{mn} \|_{W_{\mathrm{loc}}^{2,p}(R)} \leq C \| v_{mn} \|_{L_{\mathrm{loc}}^p(R)},\]
  where the constant $C$ is independent of $m$ and $n$.  Thus $\{\phi_n\}$ is a Cauchy sequence in the $W_{\mathrm{loc}}^{2,p}(R)$ topology.  This completes the proof of the proposition.
\end{proof}

\subsection{Global continuation} \label{global continuation section}

Using the local compactness result established in the previous subsection in concert with Theorem~\ref{generic global theorem}, we can now prove the existence of a global curve of solutions extending the family $\mathscr{C}_{\mathrm{loc}}$ of small-amplitude solitary waves.

\begin{theorem}[Global continuation] \label{preliminary large amplitude theorem}
  The local curve $\mathscr{C}_{\mathrm{\loc}}$ is contained in a continuous curve of solutions, parametrized as
  \begin{align*}
    \mathscr{C} = \{(\phi(t),F(t)) : 0 < t < \infty\} \subset \Udomain
  \end{align*}
  with the following properties.
  \begin{enumerate}[label=\rm(\alph*)]
        \item \label{reparam} Near each point $(\phi(t_0),F(t_0)) \in \cm$, we can reparametrize $\cm$ so that the mapping $t\mapsto (\phi(t),F(t))$ is real analytic.
  \item \label{reconnect} $(\phi(t),F(t)) \not\in \cm_\loc$ for $t$ sufficiently large.
  
  \item \label{nodal} Each wave in $\mathscr{C}$ is a monotonic wave of elevation in the sense that 
  \[ \phi(t) > 0 \quad \textrm{on } R \cup T \qquad \textrm{and}  \qquad \partial_r \phi_n(t) \leq 0 \quad \textrm{on } \{r \geq 0\}.\]  
  
  \item \label{alternatives} As $t \to \infty$,
      \begin{align}
        \label{global blowup}
        N(t):= \n{\phi(t)}_{C^{1,\alpha}(R)} + \n{\phi(t)}_{W_{\mathrm{loc}}^{2,p}(R)} + \frac 1{\inf_R (\phi_s(t)+H_s)} +
        F(t) + \frac 1{F(t)-F_{\mathrm{cr}}} \to \infty.
        \hspace*{-2em}
      \end{align}
  \end{enumerate}
\end{theorem}
\begin{proof}  In the construction of the local curve $\mathscr{C}_{\mathrm{loc}}$, it was already confirmed that it satisfies \eqref{gen limit inv}, where the Froude number $F$ plays the role of the parameter $\lambda$, and $F_{\mathrm{cr}}$ that of $\lambda_*$.  We proved in Lemma \ref{Fredholm} that when $(\phi,F)$ is a supercritical solitary wave, the linearized operator $\mathscr{F}_\phi(\phi,F) : \Udomain \subset {X} \to {Y}$ is Fredholm index $0$.  Thus the hypotheses of Theorem \ref{generic global theorem} are met, furnishing the existence of a global curve of solutions $\mathscr{C} \subset \Udomain$ with the stated parameterization, and exhibiting properties \ref{reparam} and \ref{reconnect} above.  Likewise, the monotonicity asserted in \ref{nodal} is a consequence of the fact these waves are supercritical according to Theorem~\ref{supercritical elevation} and Theorem~\ref{symmetrythm}.  

Finally, to confirm that \ref{alternatives} occurs, suppose instead that $N$ defined in \eqref{global blowup} remains bounded along the continuum.  We must therefore be in the situation described in Alternative \ref{gen alternatives}\ref{gen noncompact} of Theorem \ref{generic global theorem}, meaning that there exists a sequence $\{t_n\} \subset (0,\infty)$ such that $\sup_n N(t_n) < \infty$, but the corresponding sequence of solutions $\{(\phi(t_n), F(t_n))\}$ has no convergent subsequence in ${X} \times \mathbb{R}$.  But this is impossible, since the boundedness of $\{N(t_n)\}$ implies that the sequence is uniformly supercritical, and hence Proposition \ref{local compactness proposition} guarantees that it is precompact in ${X} \times \mathbb{R}$.  
\end{proof}

\subsection{Bounds}

We next explore the physical meaning for the blowup in Theorem \ref{preliminary large amplitude theorem}\ref{alternatives}.  This entails deriving uniform bounds for the various quantities represented by terms in the function $N(s)$.  In so doing, we make use of some qualitative theory developed by Chen, Walsh, and Wheeler \cite{chen2017existence} in their study of solitary stratified water waves.  While these results suppose more smoothness than our functions exhibit, a simple examination reveals that some of them can be trivially extended to the lower regularity regime.   This is true of the next two lemmas that we present with proof.  First, we have the following upper bound on the Froude number; see \cite[Theorem 4.7]{chen2017existence}.

\begin{theorem}[Upper bound on $F$] \label{upper bound F}
Let $(\phi,F) \in {X} \times \R$ be a solution of the height equation.  Then the Froude number $F$ satisfies the bound 
  \begin{align}
    \label{eqn:hurray}
    F^2 
    \le 
    \frac{2}{\pi}
    \n{H_s}_{L^\infty}^2    \n{\phi_s(0,\cdot)+H_s(0)}_{L^\infty}.
  \end{align}
\end{theorem}

On the other hand, we can show that the Froude is effectively bounded from below along $\mathscr{C}$ provided that certain norms remain bounded.  The following result is proved in \cite[Lemma 6.9]{chen2017existence}.  


\begin{lemma}[Asymptotic supercriticality] \label{asymptotical supercritical lemma}  
  If $\n{\phi(t)}_{C^{1,\alpha}(R)}$ is uniformly bounded along $\cm$, then
  \begin{equation}
    \liminf_{t \to \infty} F(t) >  F_{\mathrm{cr}}. \label{liminf F > Fcr} 
  \end{equation}
\end{lemma}

By contrast, the next two estimates require closer attention.  First, we must control the pressure uniformly along the continuum.  Through Bernoulli's law, this furnishes bounds on the magnitude of the velocity.  

\begin{lemma}[Bounds on the pressure and velocity] \label{pressure bound lemma} For any solitary wave with Froude number $F > F_0$, the pressure and velocity field satisfy the bounds 
\be \label{pressure bound} 
P - P_{\mathrm{atm}} + \frac{1}{2} \| \gamma_+   \|_{L^\infty} \psi \geq 0 \qquad \textrm{in } \Omega_\eta,
\ee
and 
\be \label{velocity field bound} 
(u-c)^2 + v^2 \leq 2 \| \gamma_+ \|_{L^\infty} + \frac{2}{F_0} + 2 \| E \|_{L^\infty} \qquad \textrm{in } \Omega_\eta.
\ee
\end{lemma}
\begin{proof}
The inequality \eqref{pressure bound} essentially comes from the pressure bounds for periodic waves with continuous vorticity obtained by Varvaruca \cite[Theorem 3.1]{varvaruca2009extreme}.  Here, we can adapt his argument to the solitary wave setting following \cite[Proposition 4.1]{chen2017existence} and relax the assumptions on the regularity  as in the proof of Theorem 1 in \cite[Section 7]{constantin2011discontinuous}.

Let 
\[ \varphi := P - P_{\mathrm{atm}} + \frac{1}{F^2} y \in C_{\bdd}^{1,\alpha}(\overline{\Omega_\eta}) \cap W_{\mathrm{loc}}^{2,\frac{p}{2}}(\Omega_\eta) \]
be the dynamic pressure (that is, the deviation of the pressure from hydrostatic).  Using the expression for $P$ in \eqref{pressure in terms of E}, the fact that $\nabla \psi \to (0, \Psi_y)$ uniformly in $y$, and the equations satisfied by $\Psi$, we can see that $\varphi \in C_0^0(\Omega_\eta)$. On the other hand, in light of \eqref{velocity form interior}, $\varphi$ and $\psi$ satisfy
\be \varphi_x = -\left( \psi_y^2 \right)_x + \left( \psi_x \psi_y \right)_y, \quad \varphi_y = -\left( \psi_x^2 \right)_y + \left( \psi_x \psi_y \right)_x \qquad \textrm{in } \Omega_\eta. \label{weak pressure psi system} \ee
As $\psi \in W_{\mathrm{loc}}^{2,p}$, we can distribute the derivatives above revealing that
\be \nabla \varphi = \left( -\psi_{xy} \psi_y + \psi_x \psi_{yy}, \, -\psi_{xy} \psi_x + \psi_{xx} \psi_y \right), \label{gradient dynamic pressure} \ee
which holds almost everywhere in $\Omega_\eta$.  Using the fact that 
\[ \Delta \varphi = (\Delta \psi)^2 - \psi_{xx}^2 - \psi_{yy}^2 - 2 \psi_{xy}^2 = \psi_{xx} \psi_{yy} - 2 \psi_{xy}^2,\]
and taking the divergence of \eqref{gradient dynamic pressure}, we ultimately arrive at the equation 
\be \Delta \varphi + 2\frac{\nabla \varphi + \Delta \psi \nabla \psi}{|\nabla \psi|^2} \cdot \nabla \varphi = 0.\label{dynamic pressure elliptic equation} \ee
Note that the coefficient of the first-order term $\nabla \varphi$ above is bounded, as $\Delta \psi = -\gamma(\psi) \in L^\infty(\Omega_\eta)$.  

Next, set $\theta := P - P_{\mathrm{atm}} + M \psi$, for  a constant $M$ to be determined later.  From the above equation satisfied by $\varphi$, we find that 
\[ \Delta \theta + 2 \frac{\nabla \theta - (\Delta \psi + 2M) \nabla \psi}{|\nabla \psi|^2} \cdot \nabla \theta = -M( \Delta \psi + 2M) - \frac{1}{F^2} \frac{1}{|\nabla \psi|^2} \left( \frac{1}{F^2} - (\Delta \psi + 2M) \psi_y \right) .\]
Again, we note the first-order in $\theta$ coefficients on the left-hand side above are in $L^\infty(\Omega_\eta)$.  Taking $M > \frac{1}{2} \| \gamma_+\|_{L^\infty}$, we see that the right-hand side is indeed non-positive.  Also, since $\varphi$ decays at infinity, we have 
\[ \lim_{x \to \pm \infty} \inf_{-1 < y < \eta(x)} \theta(x,y) =  \lim_{x \to \pm \infty} \inf_{-1 < y < \eta(x)} \left( -\frac{1}{F^2} y + M \psi \right) \geq 0.\]
Furthermore, evaluating $\theta_y$ on the bed using \eqref{gradient dynamic pressure}, we see that
\[ \theta_y = P_y + M \psi_y = - \frac{1}{F^2} + M \psi_y < 0 \qquad \textrm{on } \{y = -d\}.\] 
Thus $\theta$ cannot attain its minimum on the bed.  As $\theta$ vanishes identically on the free surface,  Theorem~\ref{max principle}\ref{weak max principle} tells us that $\theta$ obeys the weak minimum principle, and hence $\theta \geq 0$.  Retracing definitions, this is exactly the bound \eqref{pressure bound}.   

To control the velocity field, we may use \cite[Proposition 4.1]{chen2017existence} setting the density to be constant.  Bernoulli's law and the above pressure estimate then give \eqref{velocity field bound}.  As $\nabla \psi$ and $E$ are in $C_\bdd^{0,\alpha}$, no modification of the argument is required.  
\end{proof}

\begin{remark} \label{pressure bound remark}  Observe that in the above argument we use strongly the assumption that $\gamma \in L^\infty$.  Also, our choice to formulate an existence theory with  the additional regularity $u, v \in W_{\mathrm{loc}}^{1,p}$ --- and hence $P \in W_{\mathrm{loc}}^{2,p/2}$ --- is entirely because we cannot otherwise justify the move from the relation \eqref{weak pressure psi system} to the elliptic equation \eqref{dynamic pressure elliptic equation} that is at the heart of the proof.  
\end{remark}

\begin{lemma}[Uniform boundedness in $X$] \label{uniform boundedness lemma} For each $K > 0$, there exists a constant $C = C(K) > 0$ such that, if $(\phi, F) \in \mathscr{C}$ and $\| \phi_s \|_{L^\infty(R)} < C$, then 
\[ \| \phi \|_{C^{1,\alpha}(R)} + \| \phi \|_{W_{\mathrm{loc}}^{2,p}(R)} < K.\]  
\end{lemma}
\begin{proof}
Let $K > 0$ be given.  We will write $C$ to denote a generic positive constant depending only on $\| \phi_s \|_{L^\infty}$.   First, observe that in light of Lemma \ref{upper bound F}, we have $F < C$.  Returning to the (nondimensionalized) Eulerlian variables, we see also that 
\[ \frac{1+\phi_r^2}{2\left(\phi_s+H_s\right)^2} = \frac{1}{2} \left( (u-c)^2 + v^2 \right) < \| \gamma_+ \|_{L^\infty} + \frac{1}{F_{\mathrm{cr}}^2} + \| E \|_{L^\infty}, \]
where the last inequality comes from \eqref{velocity field bound}. Thus $\| \phi \|_{C^1(R)} < C$.  By Lemma \ref{asymptotical supercritical lemma}, this further ensures that $1/(F - F_{\mathrm{cr}}) < C$

We must now upgrade this to control of the full $C^{1,\alpha}(\overline{R})$ norm, as well as uniform boundedness in $W_{\mathrm{loc}}^{2,p}(R)$.  First we consider the H\"older estimates of $\nabla \phi$.  In a neighborhood of the top or bottom, $\phi$ enjoys improved regularity in view of Lemma \ref{addt'l regularity at boundary}.  Let $\delta \in (0,1)$ be given as in \eqref{extra regularity H} and denote $R_\delta := \mathbb{R} \times (-1+\delta, -\delta) \subset R$.  Arguing exactly as  as in \cite[Proposition 5.12]{wheeler2013solitary}, we find that 
\[ \| \phi \|_{C^{1,\alpha}(\overline{R} \, \setminus R_{\delta/2})} < C.\]  

The main new challenge is the interior regularity.  For that, we note that it was also established in Lemma \ref{addt'l regularity at boundary} that 
$\phi_r \in C_\bdd^{1,\alpha}(\overline{R})$ and solves the elliptic problem 
\[ \mathscr{F}_{1\phi}(\phi,F) \phi_r = 0 \qquad \textrm{in  } R.\]
Looking at the form of the linearized operator \eqref{Fphi}, we see that the $C^0$ norms of the coefficients are indeed uniformly bounded in terms of $C$.  The De Giorgi--Nash-type estimate \cite[Theorem 8.24]{gilbarg2001elliptic} then gives
\[ \| \phi_r \|_{C^{0,\alpha^\prime}(R_{1,\delta})} \leq C \| \phi_r \|_{C^0(R)} < C, \]
where $R_{1,\delta} := R_\delta \cap \{ |r| < 1 \}$, and $\alpha^\prime = \alpha^\prime(C) \in (0,1)$.  Using the same estimate on horizontal translates of $R_{1,\delta}$ allows us to conclude that $\phi_r \in C_\bdd^{0,\alpha^\prime}(R_\delta)$ with a norm controlled by $C$.  But then this means that the coefficients in $\mathscr{F}_\phi(\phi,F)$ are actually bounded uniformly in the $C^{0,\alpha^{\prime\prime}}(R_\delta)$ norm by $C$, with $\alpha^{\prime\prime}:= \min\{\alpha, \alpha^\prime\}$.  We may then apply the interior version of the Schauder-type estimate Theorem~\ref{schauder theorem} to conclude that 
\[ \| \phi_r\|_{C^{1, \alpha^{\prime\prime}}(R_\delta)} < C.\]
In particular, this gives $\| \phi_r \|_{C^{0,\alpha}(R_\delta)} < C$.

To obtain the analogous bound for $\phi_s$, we first note that by rearranging the terms in the equation satisfied by $\phi$ \eqref{phiHeightEqn} and using the above information, we have
\[ \left( \frac{1+ \phi_r^2}{2(\phi_s + H_s)^2} \right)_s = \left( \frac{\phi_r}{H_s + \phi_s} \right)_r - \left( \frac{1}{2H_s^2} \right)_s \in C_\bdd^{0,\alpha^{\prime\prime}}(R_\delta) \]
with a bound controlled by $C$.   Distributing the derivative on the left-hand side, it follows from this that 
\[ \left\| \left( \frac{1}{\phi_s + H_s} \right)_s \right\|_{C^{0,\alpha^{\prime\prime}}(R_\delta)} < C,\]
from which the estimate $\| \phi_s \|_{C^{0,\alpha}(R_\delta)} < C$ can be easily deduced.  Combining this with the estimates near the top and bottom, we have shown $\| \phi \|_{C^{1,\alpha}(\overline{R})} < C$.

 Finally, the control of $\phi$ in $W_{\mathrm{loc}}^{2,p}(R)$ is achieved using the argument in \cite[Lemma 5]{constantin2011discontinuous}.  There, the authors are studying periodic solutions, and in fact prove uniform control over the $W^{2,p}$ as a means of bounding the solutions in $C^{1,\alpha}$.  As we are only concerned with local Sobolev regularity, their approach works essentially verbatim.  \end{proof}

\subsection{Proof of large-amplitude existence}

\begin{proof}[Proof of Theorem \ref{main theorem}]
Due to the formulation equivalence, it suffices to prove this result working entirely with the height equation.  In Theorem \ref{preliminary large amplitude theorem}, we already constructed the global curve $\mathscr{C}$ and concluded that $N(t) \to \infty$ as $t \to \infty$, where $N$ is the quantity defined in \eqref{global blowup}.    
In view of Lemma~\ref{asymptotical supercritical lemma}, we know that if $F(t) \to F_{\mathrm{cr}}$, then it must also be the case that one of $\|\phi(t) \|_{C^{1,\alpha}(R)}$ and $\| \partial_s \phi(t) \|_{L^\infty}$ also blowup in the limit.  Likewise, the upper bound on the Froude number \eqref{eqn:hurray} implies that if $F(t) \to \infty$, then $\| \partial_s \phi(t) \|_{L^\infty} \to \infty$.  Finally, Lemma~\ref{uniform boundedness lemma} tells us that the blowup of $\| \phi(t) \|_{C^{1,\alpha}(R)}$ forces $\| \partial_s \phi(t) \|_{L^\infty}$ to diverge to infinity.  This proves part (i).

The fact that $\mathscr{C}$ begins at the critical laminar flow is a consequence of the construction of $\mathscr{C}_{\mathrm{loc}}$, thus part (ii) is obvious.  Finally, the assertion in part (iii) that the solutions in $\mathscr{C}$ are monotone waves of elevation was already established in Theorem \ref{preliminary large amplitude theorem}\ref{nodal}.  The proof is therefore complete.
\end{proof}

\section*{Acknowledgements} 
This material is based upon work supported by the National Science Foundation under Grant No. DMS-1439786 while the authors were in residence at the Institute for Computational and Experimental Research in Mathematics in Providence, RI, during the Spring 2017 semester.

The research of SW is supported in part by the National Science Foundation through Grant No. DMS-1514910.  

\appendix

\section{Quoted results} \label{quoted results section}

\subsection{Some results from elliptic theory}
\label{elliptic theory appendix}

Let $\Omega \subset \R^n$ be a connected, open set (possibly unbounded), and consider the second-order differential operator $L$ in divergence form given by
	\begin{align} \label{L}
	\ds Lu :=  \p_j(a^{ij}(x) \p_i u )+  b^i(x)\p_i u + c(x) u ,
	\end{align}
where we are using the summation convention and writing $\p_i := \p_{x_i}$.   Throughout this section, we suppose that the coefficients $a^{ij}, b^i, c$ are at least $L^\infty(\Omega)$; $c \leq 0$ in $\Omega$;  the operator is uniformly elliptic in the sense that there exists $\lambda >0$ with
	\begin{align} \label{uniform ellipticity condition}
	 a^{ij}(x) \xi_i\xi_j \geq \lambda |\xi|^2, \quad \text{ for all } \xi \in \R^n, x \in \overline{\Omega};
	\end{align}
and that $a^{ij}$ is symmetric.  

Throughout the paper, we make frequent use of the maximum principle and several of its variations, namely the Hopf boundary lemma and Serrin corner-point lemma.  For the latter, we apply the original version introduced in \cite{serrin1971symmetry} that requires the solution to possess classical regularity.  Both the maximum principle and Hopf lemma, though, hold in the weak solution setting.  A classical references for this material are \cite{fraenkel2000introduction,gilbarg2001elliptic}.

\begin{theorem}[Maximum pinciples] \label{max principle}  
    Let $u \in C^1(\Omega) \cap C^0(\overline{\Omega})$ be a weak subsolution $Lu \geq 0$ in $\Omega$.  

  \begin{enumerate}[label=\rm(\roman*)]
  \item \label{weak max principle}{\rm (Weak maximum principle)} It holds that $\sup_{\Omega} u = \sup_{\partial \Omega} u$.  
  \item \label{strong max principle} {\rm (Strong maximum principle)} Assume that the coefficients $a^{ij}, b^i, c \in C^{0,\alpha}(\overline{\Omega})$.  Then if $u$ attains its maximum value on $\overline{\Omega}$ at a point in the interior of $\Omega$, then $u$ is a constant function.  

  \item \label{hopf lemma} {\rm (Hopf boundary lemma)} Assume that  $a^{ij}, b^i, c \in C^{0,\alpha}(\overline{\Omega})$.  Suppose that $u$ attains its maximum value on $\overline{\Omega}$ at a point $x_0 \in \partial \Omega$ for which there exists an open ball $B \subset \Omega$ with $\overline{B} \cap \partial\Omega = \{ x_0 \}$.    Then either $u$ is a constant function or 
    \[ \nu \cdot \nabla u(x_0) > 0,\]
    where $\nu$ is the outward unit normal to $\Omega$ at $x_0$.
  \item \label{edge point} {\rm (Serrin corner-point lemma)} 
    Let $x_0 \in \partial\Omega$ be a ``corner point" in the sense that near $x_0$ the boundary $\partial \Omega$ consists of two transversally intersecting $C^2$ hypersurfaces $\{\gamma(x) = 0\}$ and $\{\sigma(x) = 0\}$. Suppose that $\gamma, \sigma < 0$ in $\Omega$, $u>0$ in $\Omega$ and $u(x_0) = 0$. Assume further that  $u, a^{ij} \in C^2$, $b^i$, $c \in C^0$ in a neighborhood of $x_0$, and 
    \begin{equation}\label{bluntness2}
      B(x_0) = 0, \quad \text{and } \quad \partial_\tau B(x_0) = 0
    \end{equation}
    for every differential operator $\partial_\tau$ tangential to $\{\gamma=0\} \cap \{\sigma=0\}$ at $x_0$. Then for any unit vector 
    $s$
    outward from $\Omega$ at $x_0$, either
    \begin{equation*}
      {\partial_s u}(x_0) < 0 \ \text{or }\ {\partial^2_s u}(x_0) < 0.
    \end{equation*}
  \end{enumerate}
\end{theorem} 

We rely considerably on a priori estimates for solutions of elliptic problems.  In particular, the following result takes the Schauder-type estimate of \cite[Theorem 3]{constantin2011discontinuous} and adapts it to problems posed on an infinite strip using the idea behind \cite[Lemma A.1]{wheeler2013solitary}.

\begin{theorem}[Schauder estimate] \label{schauder theorem} Suppose $\Omega = \mathbb{R} \times \mathcal{B}$, where $\mathcal{B} \subset \mathbb{R}^{n-1}$ is a bounded $C^{1,\alpha}$ domain, and assume that $a^{ij}$, $b^i$, $c \in C_\bdd^{0,\alpha}(\overline{\Omega})$.  Let 
\[ B := \left( \beta^i(x) \partial_i + \gamma(x) \right)|_{\partial \Omega},\]
with $\beta^i$, $\gamma \in C_\bdd^{0,\alpha}(\partial\Omega)$, and $\inf_{\partial\Omega} |\nu_i \beta^i | \geq \delta > 0$, where $\nu$ is the outward unit normal to $\Omega$.   If $u \in C_\bdd^0(\overline{\Omega}) \cap C^{1}(\overline{\Omega})$ is a weak solution of 
\[ Lu = \nabla \cdot F \textrm{ in } \Omega, \qquad Bu = g \textrm{ on } \partial \Omega,\]
for $F \in C_\bdd^{0,\alpha}(\overline{\Omega}; \mathbb{R}^n)$ and $g \in C_\bdd^{0,\alpha}(\partial \Omega)$,  then $u \in C_\bdd^{1,\alpha}(\overline{\Omega})$ and obeys the estimate 
\[ \| u \|_{C^{1,\alpha}(\Omega)} \leq C \left( \| u \|_{C^0(\Omega)} + \| F \|_{C^{0,\alpha}(\Omega; \mathbb{R}^n)} + \| g \|_{C^{0,\alpha}(\partial \Omega)} \right).\]
where $C > 0$ is a constant depending only on $\lambda$, $\delta$, $\Omega$, and the norms of the coefficients.  
\end{theorem}

Observe that the above theorem allows to upgrade local bounds $C^0$ bounds on $\nabla u$ to \emph{uniform} H\"older regularity provided $u \in C_\bdd^0$ and solves the elliptic equation.  This fact is used many times over the course of the paper.

Naturally, one can prove many results in the spirit of Theorem \ref{schauder theorem}:  interior estimates, estimates up to a boundary portion, estimates for the Dirichlet problem, and so on.  We will not bother to list them all.

\subsection{Center manifold reduction}
\label{quoted CMR}

\begin{theorem}[Buffoni, Groves, Toland \cite{buffoni1996plethora}]  \label{BGTCMR} Suppose that $(\calX, \Upsilon^{\ep}, \eHam)$ is a one-parameter family of reversible Hamiltonian systems, where $\calX$ is a Hilbert space, $\Upsilon^{\ep}$ is a symplectic form on $\calX$, and $\eHam$ the Hamiltonian.  Write the corresponding Hamilton equation in the form
	\begin{align} \label{linear + nonlinear}
	u_r = Lu + N^{\ep}(u),
	\end{align}
where $u(r)$ is assumed to lie in $\calX$ for each $r$.  We assume that $L: \mathcal{D}(L) \subset \calX \to \calX$ is a densely defined, closed linear operator.  Suppose that $0$ is an equilibrium for \eqref{linear + nonlinear} at $\ep =0$ and that the following conditions hold:
	\begin{enumerate}[label=$(\mathrm{H}\arabic*)$]
	\item \label{H1} The spectrum $\sigma(L)$ of $L$ contains at most finitely-many eigenvalues on the imaginary axis; each of which has finite multiplicity.  Moreover, $\sigma(L) \cap i\R$ is separated from $\sigma(L) \backslash i\R$ in the sense of Kato.  Let $\Pc$ denote the spectral projection corresponding to $\sigma(L) \cap i\R$ and put $\calX^{\mathrm{c}} := \Pc\calX, \calX^* := (1-\Pc)\calX$.  We let $n$ be the finite dimension of $\calX^{\mathrm{c}}$.
	\item \label{H2} There exists $C>0$ such that the operator $L$ satisfies the resolvent estimate
		\begin{align} \label{resolvent est}
		\ds \|u\|_{\calX} \leq \dfrac{C}{1+ |\xi|} \|(L-i\xi I)u\|_{\calX},
		\end{align}
	for all $\xi \in \R$ and $u \in \calX^*$.
	\item \label{H3} There exists a natural number $k$, and interval $\Lambda \subset \R$ containing $0$, and a neighborhood $\mathcal{U}$ of $0$ in $\mathcal{D}(L)$ such that $N$ is $C^{k+1}$ in its dependence on $(\ep, u)$ on $\Lambda \times \mathcal{U}$.  Moreover, $N^0(0)=0$ and $D_uN^0(0)=0$.
	\end{enumerate}
Then after possibly shrinking the interval $\Lambda$ and neighborhood $\mathcal{U}$, we have that, for each $\ep \in \Lambda$, there exists an $n$-dimensional local center manifold $\mathcal{W}^{\ep} \subset \mathcal{U}$ together with an invertible coordinate map
	\begin{align*}
		\chi^{\ep} := \Pc|_{\mathcal{W}^{\ep}} : \mathcal{W}^{\ep} \to \mathcal{U}^{\mathrm{c}} := \Pc\mathcal{U},
	\end{align*}
with the following properties:
	\begin{enumerate}[font=\upshape,label=(\roman*)]
		\item Defining $\Psi^{\ep} : \mathcal{U}^{\mathrm{c}} \to \mathcal{U}^* := P^*\mathcal{U}$ by $u_{\mathrm{c}} + \Psi^{\ep}(u_{\mathrm{c}}) := (\chi^{\ep})^{-1}(u_{\mathrm{c}})$, the map $(\ep, u) \mapsto \Psi^{\ep}(u)$ is $C^k(\Lambda \times \mathcal{U}^{\mathrm{c}}, \mathcal{U}^*)$.  Moreover, $\Psi^{\ep}(0) = 0$ for all $\ep \in \Lambda$ and $D\Psi^0(0) = 0$.
		\item Every initial condition $u_0 \in \mathcal{W}^{\ep}$ determines a unique solution $u$ of \eqref{linear + nonlinear}, which remains in $\mathcal{W}$ as long as it remains in $\mathcal{U}$.
		\item If $u$ solves \eqref{linear + nonlinear} and lies in $\mathcal{U}$ for all $r$, then $u$ lies entirely in $\mathcal{W}^{\ep}$.
		\item If $u_{\mathrm{c}} \in C^1((a,b), \mathcal{U}^{\mathrm{c}})$ solves the reduced system
			\begin{align} \label{reduced system reprise}
			(u_{\mathrm{c}})_r = f^{\ep}(u_{\mathrm{c}}) := Lu_{\mathrm{c}} + \Pc N^{\ep}(u_{\mathrm{c}} + \Psi^{\ep}(u_{\mathrm{c}})),
			\end{align}
		then $u= (\chi^{\ep})^{-1}(u_{\mathrm{c}})$ solves the full system \eqref{linear + nonlinear}.
		\item $\mathcal{M}^{\ep}$ is a symplectic submanifold of $\calX$ when equipped with the symplectic form $\Upsilon^{\ep} |_{\mathcal{M}^{\ep}}$ and Hamiltonian $\eK(u_{\mathrm{c}}) = \eHam(u_{\mathrm{c}} + \Psi^{\ep}(u_{\mathrm{c}}))$.  The reduced system \eqref{reduced system reprise} corresponds to the Hamiltonian flow for $(\mathcal{M}^{\ep}, \Upsilon|_{\mathcal{M}^{\ep}}, \eK)$.  In fact, it is reversible and coincides with the restriction of the full Hamiltonian to the center manifold.
	\end{enumerate}
\end{theorem}

\subsection{Abstract global bifurcation theory}
\label{global bifurcation appendix}

Let $\genX, \genY$ be Banach spaces, $\genI$ an open interval (possibly unbounded) with $0 \in \overline{\genI}$, and $\genU \subset \genX$ an open set with $0 \in \partial \genU$. Consider the abstract operator equation
\[ \genG(x,\lambda) = 0, \]
where $\genG \maps \genU \by \genI \to \genY$ is an analytic mapping. Assume that for any $(x,\lambda) \in \genU \times \genI$ with $\genG(x,\lambda) = 0$, the Fr\'echet derivative $\genG_x(x,\lambda) \maps \genX \to \genY$ is Fredholm with index $0$.

\begin{theorem}[Chen, Walsh, and Wheeler \cite{chen2016announcement,chen2017existence}] \label{generic global theorem}
  Suppose that there exists a continuous curve $\mathscr{C}_\loc$ of solutions to $\genG(x,\lambda) = 0$, parametrized as
  \begin{align*}
    \mathscr{C}_\loc :=  \{(\tilde x(\lambda),\lambda) : 0 < \lambda < \lambda_*\}
    \subset \genG^{-1}(0)
  \end{align*}
  for some $\lambda_* > 0$ and continuous $\tilde x \maps (0,\lambda_*) \to \genU$.  If
  \begin{align}
    \label{gen limit inv}
    \lim_{\lambda \searrow 0} \tilde x(\lambda) = 0 \in \dell \genU,
    \qquad 
    \genG_x(\tilde x(\lambda),\lambda) \maps \genX \to \genY
    \textup{ is invertible for all $\lambda$},
  \end{align}
  then $\mathscr{C}_\loc$ is contained in a curve of solutions $\mathscr{C}$, parametrized as
  \begin{align*}
    \mathscr{C} := \{(x(t),\lambda(t)) : 0 < t < \infty\} \subset \genG^{-1}(0)
  \end{align*}
  for some continuous $(0,\infty) \ni t \mapsto (x(t),\lambda(t)) \in \genU \by \genI$, with the following properties.
  \begin{enumerate}[label=\rm(\alph*)]
  \item \label{gen alternatives} One of the following alternatives holds:
    \begin{enumerate}[label=\rm(\roman*)]
    \item \textup{(Blowup)} \label{gen blowup alternative}
      As $t \to \infty$,
      \begin{align}
        \label{gen global blowup}
        N(t):= \n{x(t)}_\genX + \frac 1{\dist(x(t),\dell \genU)} +
        \lambda(t) + \frac 1{\dist(\lambda(t),\dell \genI)} \to \infty.
      \end{align}
    \item \label{gen noncompact} \textup{(Loss of compactness)} There exists a sequence $t_n \to \infty$ such that $\sup_n N(t_n) < \infty$ but 
      $\{x(t_n)\}$ has no subsequences converging in $\genX$.
    \end{enumerate}
  \item \label{gen reparam} Near each point $(x(t_0),\lambda(t_0)) \in \mathscr{C}$, we can reparameterize $\mathscr{C}$ so that $t\mapsto (x(t),\lambda(t))$ is real analytic.
  \item \label{gen reconnect} $(x(t),\lambda(t)) \not\in \mathscr{C}_\loc$ for $t$ sufficiently large.
  \end{enumerate}
 \end{theorem} 

\bibliographystyle{siam}
\bibliography{projectdescription}

\end{document}